\documentclass[11pt, english, reqno]{amsart}
\usepackage{fullpage}
\usepackage[utf8]{inputenc}
\title{Stability of  equilibria uniformly in the inviscid limit for the
 Navier-Stokes-Poisson system}
\author{Fr\'ed\'eric Rousset, Changzhen Sun}
\address{Universit\'e Paris-Saclay,  CNRS, Laboratoire de Math\'ematiques d'Orsay (UMR 8628),  91405 Orsay Cedex, France}
\email{frederic.rousset@universite-paris-saclay.fr, changzhen.sun@universite-paris-saclay.fr }

\usepackage{graphics}
 \usepackage{amsmath,amssymb,stmaryrd,array,multirow,dsfont,mathrsfs}
\usepackage{amsthm}
\usepackage{enumerate}
\usepackage{xcolor}
\usepackage{color}
\usepackage{hyperref}
\usepackage{indentfirst}

\usepackage[utf8]{inputenc}
\usepackage[T1]{fontenc}
\usepackage{microtype}
\usepackage{graphics}
 \usepackage{amsmath,amssymb,stmaryrd,array,multirow,dsfont,mathrsfs}
\usepackage{amsthm}
\usepackage{xcolor}
\usepackage{color}
\usepackage{hyperref}

\graphicspath{{Images/}}

\xdefinecolor{darkgreen}{rgb}{0,0.4,0}

\newcommand{\beq}{\begin{equation}}
\newcommand{\eeq}{\end{equation}}
\newcommand{\beqs}{\begin{equation*}}
\newcommand{\eeqs}{\end{equation*}}
\newcommand{\ben}{\begin{eqnarray}}
\newcommand{\een}{\end{eqnarray}}
\newcommand{\beno}{\begin{eqnarray*}}
\newcommand{\eeno}{\end{eqnarray*}}
\renewcommand{\Re}{{\rm Re}\,}

\renewcommand{\div}{{\rm div}\,}

\newcommand{\Supp}{{\rm Supp}\,}

\newcommand{\Rmnum}[1]{\uppercase\expandafter{\romannumeral #1} }
 \numberwithin{equation}{section}

\newtheorem{thm}{Theorem}[section]
\newtheorem{lem}[thm]{Lemma}
\newtheorem{prop}[thm]{Proposition}
\newtheorem{rmk}[thm]{Remark}
\newtheorem{cor}[thm]{Corollary}

\def\curl{\mathop{\rm curl}\nolimits}
\let\r=\rho

\def \d {\mathrm {d}}

\def\cF{{\mathcal F}}
\def\cG{{\mathcal G}}

\def\cM{{\mathcal M}}

\def\cP{{\mathcal P}}

\let\f=\frac
\def \p {\partial}
\def\mR {\mathbb{R}}
\def\pa {\partial^{\alpha}}
\def\ep{\varepsilon}

\def \ltr{\langle t\rangle}

\def \vep {\varepsilon}
\def \bt {\tilde{b}}
\def \pab {\partial _{\xi}^{\alpha}\partial _{\eta}^{\beta}}

\def\lnr {\langle \nabla \rangle}
\def \lxr{\langle \xi \rangle}
\def \ler{\langle \eta \rangle}
\def \lxmer{\langle \xi-\eta \rangle}
\def \lxper{\langle \xi+\eta \rangle}

\def \pt {\partial_{t}}
\def \vr {\varrho}

\def\R {\mathcal{R}}

\def \si {\sigma}

\def\pta{\partial^{\tilde{\alpha}}}

\def \na{\nabla}
\def \om{\omega}
\def \lp {\lambda_{+}}
\def \lm {\lambda_{-}}

\def \define {\triangleq}
\def \kpz {\kappa_{0}}
\def \chidix {\chi_{\varepsilon,\kpz}(\xi)}

\def \cchidix{(1-\chi_{\varepsilon,\kpz})}
\def \chidid {\chi_{\vep,\kpz}(D)}
\def \cchidid {(1-\chi_{\varepsilon,\kpz})(D)}
\def \tchidix{\tilde{\chi}_{\varepsilon,\kpz}(\xi)}
\def \tchidie {\tilde{\chi}_{\varepsilon,\kpz}(\eta)}
\def \tchidiepx  {\tilde{\chi}_{\varepsilon,\kpz}(\xi+\eta)}
\def \tchidiemx  {\tilde{\chi}_{\varepsilon,\kpz}(\xi-\eta)}
\def \tchidid {\tilde{\chi}_{\varepsilon,\kpz}(D)}
\def \chidir {\chi_{\varepsilon,\kpz}(\lambda r)}
\def \indxf {\frac{3}{2}(1-\f{2}{p})}
\def \indxs {\sigma+3(1-\frac{2}{p})}
\def \ek {8_{\kappa}}

\begin{document}

\maketitle
\begin{abstract}
   We prove a stability result of constant equilibra for the  three dimensional Navier-Stokes-Poisson system uniform in the inviscid limit. We allow the initial density  to be  close to a constant and
    the potential  part of the  initial velocity to be  small independently of the   rescaled viscosity  parameter $\ep$ while
    the incompressible part of the initial velocity is assumed to be small compared to  $\ep$. We  then get  a unique global
    smooth solution.
    We also prove a uniform in $\ep$ time decay rate for these solutions.
    Our approach allows to combine the parabolic  energy estimates that are efficient for the viscous equation at $\ep$ fixed
     and the  dispersive techniques (dispersive estimates and normal forms) that are useful
      for the inviscid irrotational system.

\end{abstract}
\section{Introduction}

The Navier-Stokes-Poisson  system   is a hydrodynamical model of  
plasma
which describes the dynamics of  electrons  and  ions that  interact with its self-consistent electric field. If we neglect the motion of ions, then the dynamics of electrons can be described by the following electron Navier-Stokes-Poisson system (ENSP)
 \beq \label{NSPEO}
 \left\{
\begin{array}{l}
\displaystyle \pt \rho^{\varepsilon} +\div( \rho^{\varepsilon} u^\varepsilon)=0,\\
\displaystyle \pt ( \rho^{\varepsilon} u^{\varepsilon})+\div(\rho^{\varepsilon}u^{\varepsilon}\otimes u^{\varepsilon} )-\varepsilon \mathcal{L} u^{\varepsilon}+
\nabla p(\rho^{\varepsilon})-\rho^{\varepsilon}\nabla \phi ^{\varepsilon}=0 ,  \\
\displaystyle \Delta \phi^{\varepsilon} =\rho ^{\varepsilon}-1,\\
\displaystyle u|_{t=0} =u_0^{\varepsilon} ,\rho|_{t=0}=\rho_0^{\varepsilon}.
\end{array}
\right.
\eeq
We shall always consider in this paper that the spatial domain is the whole space,   $x \in \mathbb{R}^3$.
 Here  the unkowns $\rho^{\varepsilon} (t,x)\in \mathbb{R}_{+}$ , $u^{\varepsilon}\in \mathbb{R}^{3}$, $\na\phi^{\varepsilon}\in\mathbb{R}^{3}$ are the  electron density, the  electron velocity and the  self-consistent electric field respectively. The thermal pressure of electrons $p(\rho^{\varepsilon})$ is usually assumed to follow a  polytropic $\gamma$-law:  $p(\rho^{\varepsilon})= C(\rho^{\varepsilon})^{\gamma},\, \gamma>1$
 while the viscous term is under the form
 $$ \mathcal{L}u^{\varepsilon}=\mu \Delta u^{\varepsilon}+(\mu+\lambda)\na \div u^{\varepsilon}$$ where the Lam\'e coefficients  $\mu,\lambda$ are  supposed to be constants which satisfy the condition:
 $$\mu>0    \qquad  2\mu+\lambda> 0$$
 Note that we consider a scaled version of the system with the coefficient $\varepsilon$ in front of the diffusion terms which is the inverse of the Reynolds number and which will be assumed small in this paper.
 For the simplicity of the presentation, we shall assume  in this paper that  $\mu=1,\,\lambda=0$
 and that
 $p(\rho^\ep)=(\rho^\ep)^2/ 2.$ Nevertheless,
 there is no special cancellation arising from this choice (the easiest  case for the analysis 
 in this paper
  would be the choice  $\mu(\rho)= \rho,$ $\lambda = -\mu$, since in this case there are curl free solutions of \eqref{NSPEO}).
   The results of  this paper can thus be easily extended  to   general pressure and to general density dependent  $\mu,\lambda$
 as long as $\mu(1)>0$,  
 $2\mu(1)+\lambda(1)> 0$.
We shall also handle in this paper  a simplified system  for the dynamics of  ions, the electrons being considered
 in thermodynamical equilibrium which reads
  \beq\label{NSPION}
 \left\{
\begin{array}{l}
\displaystyle \pt \rho_{+}^{\varepsilon} +\div( \rho_{+}^{\varepsilon} u_{+}^\varepsilon)=0,\\
\displaystyle \pt ( \rho_{+}^{\varepsilon} u_{+}^{\varepsilon})+\div(\rho_{+}^{\varepsilon}u_{+}^{\varepsilon}\otimes u_{+}^\varepsilon )-\varepsilon \mathcal{L} u_{+}^{\varepsilon}+
\nabla p(\rho_{+}^{\varepsilon})-\rho_{+} ^{\varepsilon}\nabla \phi_{+} ^{\varepsilon}=0 ,  \\
\displaystyle \Delta \phi_{+}^{\varepsilon}-\phi_{+}^{\varepsilon} =\rho_{+} ^{\varepsilon}-1,\\
\displaystyle u_{+}|_{t=0} =u_{+0}^{\varepsilon} ,\rho_{+}|_{t=0}=\rho_{+0}^{\varepsilon}.
\end{array}
\right.
\eeq
There is  a large  body of literature dealing with the stability   
under small and smooth enough  perturbations of  the  constant equilibrium (say $(\rho^{\ep},u^{\ep})=(1,0)$)
 of (ENSP) when $\ep= 1$. Here stability means global existence (in a suitable Sobolev or  Besov space) and decay for  small perturbations.
  We refer for example to  \cite{MR2609958}  where  global existence in $H^{l}$ for $l\geq 4$ is proven under the assumption that the initial perturbation is small in $H^{l}$ and $L^1$. %
An explicit  time decay rate for the  perturbation  is  obtained by a careful  analysis of the Green function of the linearized system (we also refer to \cite{MR1355414}).
More recently, in \cite{MR2917409} global existence in $H^{N}$($N\geq 3$) of (ENSP) is obtained  by using only energy estimates  under the assumption that the initial perturbation belongs to $H^{N}$ and is small in $H^3$.
Moreover, as in works on the compressible Navier-Stokes system  \cite{MR3005540},
by assuming that the initial data is in  a negative Sobolev space $\dot{H}^{-s}$ ($0<s<\f{3}{2})$,
 explicit decay rates can be obtained  by using interpolation inequalities and energy estimates.
 These results use heavily the fact that the equation for the velocity is a parabolic equation
 and that the coupling between the two evolution equations of (ENSP) yields decay of the density.
 In \cite{MR2609958},  global existence in dimension $d$ is obtained  in hybrid Besov spaces 
    when the initial perturbation is close to equilibrium  in a
    $L^2$ critical norm  by using energy estimates and  by considering low  and high frequencies
    differently. This result was then generalized to a  $L^{p}$ critical frameworks \cite{MR2914601},\cite{MR3695805}.

  All  these works deal  with  an unscaled system, that is to say (ENSP) with $\varepsilon=1$.
   We can easily check that for the $\varepsilon$ dependent system, these works give global smooth solutions
   if the initial perturbation is  small enough compared to $\varepsilon$ and that the  obtained decay rates hold in terms of the slow
   time variable $\varepsilon t$ (for example \cite{MR2917409} would give that in $L^\infty,$  $(\rho^\varepsilon - 1)$
    is bounded by
   $\varepsilon(1+\varepsilon t)^{-\f{3}{2}}$).
   Indeed,  global existence is obtained by bootstrap  arguments and  a priori  estimates. There are roughly  two ways to get the a priori estimates. One way is, as in \cite{MR2609958}, \cite{MR2917409},
    to use energy estimates and to  get dissipation for $u^{\ep}$ by using  the diffusion term $\ep\Delta u^{\ep}$
   and dissipation for $\rho^{\ep}-1$
   by using  a "cross energy estimate".  The nonlinear terms   can be absorbed
    if some quantity is small compared to $\ep$.
   The other way is, as in  \cite{MR2523304}, \cite{MR2914601}, \cite{MR3695805} when considering global existence in critical Besov spaces
    is to use  the maximal  smoothing effect of the  heat kernel $e^{\ep t\Delta }$, which gives for example for the scaled heat equation
  $$\| e^{\ep t\Delta }f\|_{
   L^{1}(\mathbb{R}_{+}, \dot{B}_{p,1}^{s+2})}\lesssim \ep^{-1}
   \|f\|_{
   \dot{B}_{p,1}^s}.$$
Therefore,  to  control the nonlinear terms,  this also  leads to the assumption that the size of the  initial perturbation
    has to be  small compared  to $\ep$.

  Nevertheless, when $\ep=0$, the system $\eqref{NSPEO}$ reduces to the so-called electron Euler-Poisson (EEP) system. For the (EEP) system, the global existence of smooth solutions
  close to the  constant equilibrium $(1,0)$ was first obtained  by Guo \cite{MR1637856} under neutral, irrotational, small perturbation to the reference equilibrium $(\rho^{0},u^{0})=(1,0)$. The neutral assumption ($\int(\rho_{0}^{0}-1) \d x=0$) was then removed in  \cite{MR3032977}.
   The important property which was used in these works is that the (EEP) system has better dispersive properties
  than the Euler equations for compressible fluids due to the presence of the electric field.
  For example, when restricted to irrotational solutions, the linearized (EEP) system can be rewritten as a  Klein-Gordon equation which verifies in space dimension $d$ the decay estimate
    \beqs
    \|e^{it\lnr}f\|_{L^{\infty}}\lesssim (1+t)^{-\f{d}{2}}\|f\|_{W^{d,1}}
    \eeqs
    which is better than the one of the wave equation.
    Nevertheless, in dimension $3$, the only use of energy estimates and of the above dispersive decay
    (or its $L^p\rightarrow L^{p'} $ counterpart) is not enough to get global smooth solutions in the presence of quadratic
    nonlinearities. Some additional ingredient is thus needed namely either energy estimates using
     the  vector fields methods or the normal form method. For the Euler-Poisson system
   the normal form method of Shatah \cite{MR803256} or more generally, the ‘space-time resonances’ philosophy can  be
    used to control the nonlinear terms.
    We refer  to \cite{MR803256} and \cite{germain2010space}, \cite{MR3032977}
    for more information about normal form method and the  'space-time resonance' approach.
     This  type of  approach was recently successfully used to handle the (EEP) system in dimension two \cite{MR3274788}\cite{MR3024265}
     and one
     \cite{MR3595365}.

    Since in concrete  physical flows the Reynolds number is usually very high (thus $\ep$ very small),
    it is natural to ask for stability results that hold uniformly with respect to $\ep$ for (ENSP).
    Though the methods used in the  two lines of results that we just presented are completely different,
     it is rather natural to expect to get global smooth solutions for  perturbations
     of the constant equilibrium $(1,0)$  with a smallness assumption on the perturbation that is independent
      of $\ep$ except for the curl part of the velocity (remember that for $\ep= 0$  we have global smooth solutions
      only for irrotational data). This is the  result that we shall obtain in this paper.
       A first attempt to get such a result would be to write the solution of (ENSP) as the global solution of (EEP)
        plus a remainder and to try to control the remainder. Since the source term in the equation for the perturbation
        is of order $\ep$, one could hope to use the parabolic methods described above to control the remainder.
         Nevertheless, such a naive approach  cannot work. Indeed, even in dimension $3$,  the source term
         in the equation for the remainder  has  a non integrable decay in the energy norm so that there is no hope
         to be able to control the remainder globally in time. We thus really  need
          to develop  a method that allows to use the type of ideas introduced in the study of dispersive PDE
          when there is a small dissipative term in addition. This is the main aim of this paper.
       As far as we know, there are few works addressing this type of question, in \cite{MR3998637}
       it is the extension of the vector field method that is developed.
         The situation that we are dealing with here for (ENSP)
           occurs for many other systems of mathematical physics. Indeed, there are many other systems for
            which  we have
          for  the viscous version of the  physical model,  global existence for small,  viscosity dependent data  and for
          the inviscid version (which is often a dispersive perturbation of a compressible type Euler equation) global existence for small irrotational data. We can think about MHD, water-waves...We thus hope that the approach developed in this paper
          can be useful to handle other systems. As an illustration, we shall also handle the Navier-Stokes-Poisson
           system for ions, the results are described in the end of the introduction.


We shall denote  by   $\mathcal{P}$  the Leray projector on divergence free vector fields so that   $\mathcal{P}^{\perp}=Id-\mathcal{P}=\na \Delta^{-1} \div $.
 The following is our main result for the (ENSP) system:


\begin{thm}\label{thmelectric}
Let us set  $\na\phi_0^{\ep}=-\na (-\Delta)^{-1}(\rho_0^{\ep}-1).$
There exists $\delta_{0}>0$ such that for every family of  initial data that satisfy for every $\ep \in (0, 1]$ the estimates
: 
\beno
\|(\rho_0^{\varepsilon}-1,\na \phi_0^{\varepsilon},\mathcal{P}^{\perp}u_0^{\varepsilon})\|_{W^{\si+3,1}}+\|(\rho_0^{\varepsilon}-1,\na \phi_0^{\varepsilon},\mathcal{P}^{\perp}u_0^{\varepsilon})\|_{H^{N}}\leq \delta_{0}\\
\|\mathcal{P}u_0^{\varepsilon}\|_{H^{3}}\leq \delta_{0} \varepsilon
\eeno
with   $\sigma\geq 5$ and $N\geq \sigma+7$,
then, for every  $\varepsilon \in(0,1]$, there exists a unique global solution of the (ENSP) system \eqref{NSPEO} in $C([0,+\infty),H^{3})$.
 If in addition, we assume that $ \sup_{\ep \in (0, 1]}
\|\mathcal{P}u_0^{
\varepsilon}\|_{ \dot{H}^{-s}} <+ \infty$
for  some $0<s<\f{1}{2}$,  then we have the following time decay estimates that  are uniform  in $\ep$.  There exists  $C>0$
such that for every $\ep \in (0, 1]$, we have
 \beqs
 \|(\rho^{\varepsilon}-1,\na \phi^{\varepsilon},u^{\ep})\|_{W^{1,\infty}}\leq C\big( \min\{\ep,(1+ t)^{-\f{s}{2+s}}\}
 +(1+t)^{-(\f{11}{8}+)}\big), \quad \forall t \geq 0.
 \eeqs
 where
  $a^{+}$ stands for 
any number  strictly larger  arbitrarily  close to $a.$
\end{thm}

\begin{rmk}
  If in addition, $\mathcal{P}u_0^{\varepsilon}$ is in $H^{M}$ (say $\sup_{\ep\in(0,1]}
  \|\mathcal{P}u_0^{
\varepsilon}\|_{ \dot{H}^{M}}<+ \infty$)
 and  $\sigma\geq M+2>5$, then  the solution constructed in Theorem \ref{thmelectric} also belongs to $C([0,\infty),H^M).$
\end{rmk}

Note that the assumption that we make on the size of the  "curl" part of the initial data,  that is to say the assumption on
 $\mathcal P u_{0}^\ep$, seems to be  the natural one. Indeed, even if we assume that
  $\mathcal P u_{0}^\ep= 0$, this property is not propagated by the system (ENSP), the convection diffusion equation
  for the rotational part
  of the velocity is forced by a source term of size $\ep$ so that a curl part of size $\ep$ is instantaneously
  created.

  The main difficulty in order to get Theorem \ref{thmelectric} lies in the interaction between the dynamics of the potential
  part and the incompressible part of the solution. For the potential part we could expect a $L^\infty$ decay
  given by the linear inviscid dispersive estimates  of the order $(1+ t)^{-\frac{3}{2}}$. For the incompressible part,
  we expect that this component will remain of order $\ep$ in $H^s$  but its  decay is driven by the heat equation with diffusivity $
  \ep$, in terms of uniform in $\ep$ estimate this can only yield at best  a rather slow  decay rate of order $(1+t)^{-1}$ which could be difficult to handle especially in the control of the interaction with the potential part.
 Our strategy to prove Theorem \ref{thmelectric}  is to split the system into two viscous  systems, with initial data
$(\rho_0^{\varepsilon}-1,\na \phi_0^{\varepsilon},\mathcal{P}^{\perp}u_0^{\varepsilon})$ and $(0,0,\mathcal{P}u_0^{\varepsilon})$ respectively.
The first one  will have  global solutions under $\ep$-independent assumptions on the inital data $(\rho_0^{\varepsilon}-1,\na \phi_0^{\varepsilon},\mathcal{P}^{\perp}u_0^{\varepsilon})$ and the solutions
will enjoy the same decay estimates as the  (EEP) system.
 The other is just the perturbation of the original system \eqref{NSPEO} by the solution to the former one,
  the important points are that for this system the initial data and the source term are small compared to
  $\ep$ and that the source term has integrable decay in $L^2$. We can thus  use energy estimates  and  the good decay properties  of the solutions to the  former system to prove  global existence and decay.

 More precisely, we write the solution $(\rho^{\varepsilon},\na\phi^{\varepsilon},u^{\varepsilon})$
  of (ENSP) as
  $$(\rho^{\varepsilon},\na\phi^{\varepsilon},u^{\varepsilon})  =(\rho,\na\phi,u)+(n,\na\psi,v),$$
 where
$(\rho,\na\phi,u)$ and $(n,\na\psi,v)$ are the solutions of the following systems:
 \beq \label{NSPlow}
 \left\{
\begin{array}{l}
\displaystyle \pt \rho +\div( \rho u)=0,\\
\displaystyle \pt u+u \cdot {\na u}-\varepsilon \mathcal{L} u+
\nabla \rho-\nabla \phi=0 ,  \\
\displaystyle \Delta \phi =\rho -1,\\
\displaystyle u|_{t=0} =\mathcal{P}^{\perp}u_0^{\varepsilon} ,\rho|_{t=0}=\rho_0^{\varepsilon}.
\end{array}
\right.
\eeq
 \beq \label{NSPP}
 \left\{
\begin{array}{l}
\displaystyle \pt n +\div( \rho v+nu+nv)=0,\\
\displaystyle \pt v+u\cdot {\na v}+v\cdot (\na u+\na v)-\varepsilon\mathcal{L}v +\na n -\na \psi
=\varepsilon(\f{1}{\rho+n}-1) (\mathcal{L}v+\mathcal{L}u),   \\
\displaystyle \Delta \psi=n,\\
\displaystyle v|_{t=0} =\mathcal{P}u_0^{\varepsilon}, n|_{t=0}=0.
\end{array}
\right.
\eeq
Note that for  these two systems we skip the $\varepsilon$ dependence of the solutions in our notation.

We can set $\vr=\rho-1,$  to change  system \eqref{NSPlow}  into:
 \begin{equation}\label{NSPlow1}
 \left\{
\begin{array}{l}
\displaystyle \pt \vr +\div u=-\div(\vr u),\\
\displaystyle  \pt u+u \cdot {\na u}-\varepsilon \mathcal{L} u+
\nabla \vr-\nabla \phi=0, \\
\displaystyle \Delta \phi =\vr,\\
\displaystyle u|_{t=0} =\mathcal{P}^{\perp}u_0^{\varepsilon} ,\vr|_{t=0}=\vr_0=\rho_{0}^{\varepsilon}-1.
\end{array}
\right.
\end{equation}
Note that the  initial datum for the last system is such that  $\curl(\mathcal{P}^{\perp}u_0^{\varepsilon})=0$,
and this irrotational property will be propagated which means that a smooth solution of this system will remain irrotational.
This system is thus a really good viscous approximation
of the  Euler-Poisson system. As we shall see below, the linear part of this system  has the same decay properties  for low  frequencies as  the (EEP) system, that is  for localized initial data, the $L^{p}$ norm of $(\vr,\na \phi,u)$ decay like $(1+t)^{-\f{3}{2}(1-\f{2}{p})}$ uniformly for $\ep\in(0,1]$.



The following is the main result for the system \eqref{NSPlow1}.

\begin{thm}\label{thminviscid}
For any $6<p<+\infty$,  there exists $\delta _{0}>0$ 
such that for any family of initial data satisfying
\beno
&&\sup_{\varepsilon\in(0,1]} \left( \|(\vr_0^{\varepsilon},\na \phi_0^{\varepsilon},\mathcal{P}^{\perp}u_0^{\varepsilon})\|_{W^{\si+3,1}}+\|(\vr_0^{\varepsilon},\na \phi_0^{\varepsilon},\mathcal{P}^{\perp}u_0^{\varepsilon})\|_{H^{N}} \right) \leq \delta_0
\eeno
with  $\sigma \geq 3, N\geq \sigma+7,$
then 
for every $\ep \in (0, 1]$,
there exist a unique solution for system \eqref{NSPlow1} in
$C([0,\infty),H^N)$.
Moreover, we have  the following time decay estimates that are uniform  for   $\ep\in (0,1]$.  There exists a constant $C$
such that for every $\ep \in (0, 1]$, we have
\beq\label{decayeqlow}
\|(\vr,\na\phi,u)(t)\|_{W^{\si,p}}\leq C\delta_0  (1+t)^{-\f{3}{2}(1-\f{2}{p})}, \quad \forall t \geq 0.
\eeq

\end{thm}


Let us now  explain  the main ideas for the proof. Using the ‘curl-free’ condition, we consider the new unkown  $V=(\f{\lnr}{|\na|} \vr,\f{\div}{|\na|} u)$. The linearized  system for $V$ is
 $$\pt V+AV=0, \quad
A=\left(
  \begin{array}{cc}
    0&\lnr\\
    -\lnr&-2\vep \Delta\\
  \end{array}
\right).$$
where we  use  $\lnr=\sqrt{1-\Delta}$ the Fourier multiplier with symbol $\sqrt{1+|\xi|^2}$.
The eigenvalues for this system are
\beqs 
   \lambda _{\pm}=-\vep |\xi|^2\pm i \sqrt{1+|\xi|^2-\vep^2 |\xi|^4}\define -\vep |\xi|^2\pm ib(\xi)
   \eeqs
   A toy model to present the ideas is thus
   \beqs
   \left\{
   \begin{array}{c}
   \displaystyle\pt\beta-\lambda_{-}(D)\beta=\beta^2\\
   \displaystyle\beta|_{t=0}=\beta_{0}
   \end{array}
   \right.
   \eeqs
   The key observations are, on the one hand, when we focus on low frequencies, (say $\ep|\xi|^{2}\leq 2\kpz$ with $\kpz$ to be chosen  small but independent of $\ep$) then $b(\xi)$ is very close to $\lxr$, this 
   indicates that the  imaginary part $e^{itb(D)}$ should give us an  $L^{p}$ decay estimate ($p>2$) which is uniform for $\ep\in (0,1]$. On the other hand, when we deal with high frequencies (in the sense that $\ep|\xi|^{2}\geq\kpz$), direct computations show that there exists a positive constant $c=c(\kpz)$ such that $\Re(\lambda_{\pm})\leq -c(\kpz)$ for any $\ep\in (0,1]$, so we can expect that the high frequency part of the solution has good decay  even in $L^2$
   norm.

   Define $\beta=P^{L}\beta+P^{H}\beta=\beta^{L}+\beta^{H}$ where $P^{L},P^{H}$ are the Fourier multipliers that
    project    on low and high frequencies  in the above sense respectively. We then define the  norm
   \beq\label{norm for toy model}
   \|\beta\|_{X_{T}}=\sup_{t\in[0,T)}\left(\|\beta (t)\|_{H^{10}}+\ltr^{\f{3}{2}}\|\beta^{H}(t)\|_{H^{10}}+\ltr^{\f{3}{2}(1 - \frac{2}{p})}\|\beta^{L}(t)\|_{W^{3,p}}\right).
   \eeq
   where we use the notation  $\ltr=\sqrt{1+t^2}$.
   The first Sobolev norm can be estimated by standard energy  estimates. The other two terms involve time decay
   estimates.
   The high frequencies piece is easier because we have
   uniform (with respect to $\ep$) upper bounds for $\Re(\lambda_{\pm})$ and thus an $L^2 \rightarrow L^2$ type estimate
     with exponential decay uniformly in $\ep$ for the semi-group. The low frequency piece  is more  difficult to get.
  We first check  that  $e^{itb(D)}$ enjoys the same dispersive estimates as $e^{it\lnr}$ uniformly for $\ep\in(0,1]$.
  As for the (EEP) systems the linear dispersive estimates are not enough to control the quadratic nonlinearity,
  we thus have to use  normal form transformation to close the low frequencies decay estimate.
   In this step, we have to carefully track the contribution of the viscous term that creates new error terms.    More precisely,
   let us  write $\alpha=e^{-itb(D)}P^{L}\beta$ then, $\alpha$
   satisfies the equation
   $$\pt\alpha-\ep\Delta\alpha=e^{-itb(D)}(\beta^2)^{L}.$$
   By Duhamel's formula, we have:
   \beqs
   \beta=e^{itb(D)}\alpha=e^{itb(D)}(e^{\ep t\Delta}\beta^{L}_0+\int_{0}^{t}e^{\ep(t-s)\Delta}e^{-isb(D)}\chi^{L}(D)((\beta^{L})^{2}
   +\beta \beta^{H}+\beta^{H}\beta^{L})(s)\d s).
   \eeqs
   We focus only on the first term in the above integral, the decay for the  other terms is easy to obtain
   because of the $L^2$ decay of the high frequency part.
   We can see  the first term as
   \beq\label{toylow}
   e^{itb(D)}\cF^{-1}\int_{0}^{t}\int_{\mathbb{R}^3}e^{-\ep(t-s)|\xi|^{2}}e^{is\varphi}\hat{\alpha}(s,\xi-\eta)\widehat{\alpha}(s,\eta)\d \eta\d s
   \eeq
   where $\varphi=-b(\xi)+b(\eta)+b(\xi-\eta)>0$ for $\kpz$
   small enough.
 Following the ‘space-time resonance’ method,   by using the  identity $e^{is\varphi}=\f{1}{i\varphi}\partial_{s}e^{is\varphi}$, we  can integrate by parts in time so that  \eqref{toylow} becomes:
   \beno\label{toydifficult}
   i\int_{0}^{t}e^{i(t-s)b(D)}e^{\ep(t-s)\Delta}\big(\ep\Delta T_{\f{1}{\varphi}}(\beta^{L},\beta^{L})+T_{\f{1}{\varphi}}(\ep\Delta{\beta^{L}}+(\beta^{2})^{L},\beta^{L})\big)\d s
   \eeno
     plus boundary terms and symmetric term which are similar  to handle (we refer to Section  \ref{notations} for the definition of the bilinear operator $T_{\f{1}{\varphi}}.$)
    The last term is cubic and thus can be estimated as in the study of the (EEP) system (we shall check that
     for $\kpz$ sufficiently small the operator $T_{\f{1}{\varphi}}$ has the same continuity properties as in the inviscid case).
     The first two terms are still quadratic but are $\ep$ small, we can thus get
    additional decay by  using the decay provided by the heat equation: for example, we expect that
     the $L^2$ norm of $\ep\Delta\beta^{L}$ has decay like $(1+t)^{-1}$.
      This is enough to get Theorem \ref{thminviscid} for $6<p \leq 12$. To propagate  the estimate for larger p
       which involves a faster rate of decay, the previous $(1+t)^{-1}$ gain is not enough and we shall
       therefore perform another step of integration by parts in time in order to close the estimate.

Let us now consider the system \eqref{NSPP}.
We shall see  the system \eqref{NSPP} as a perturbation of \eqref{NSPEO} by $(\rho,\na\phi,u)$.
 Thanks to the  good  decay estimates for $(\rho,\na\phi,u)$(in the sense that the time decay of the  $L^{\infty}$ norm is integrable in time), we can still get global existence by energy estimates for this system. We will prove the following result.

\begin{thm}\label{thmp}
We fix the number $p\geq24$ in Theorem \ref{thminviscid}. Consider $(\varrho,u,\na\phi)$ and $\delta_0$ given by Theorem \ref{thminviscid}.
If $\delta_0$ is small enough and  $\|\mathcal{P}^{\perp}u_0^{\varepsilon}\|_{H^3}\leq \delta_0 \varepsilon$,
then the system \eqref{NSPP}  has  a solution in $C([0,+\infty),H^3)$
and $$\sup_{0\leq t<+\infty}\|(n,\na\psi,v)(t)\|_{H^3}\leq 8\delta_0 \varepsilon.$$
Moreover, if we assume in addition that  for some $s$,   $0<s<\f{1}{2}$,
$ \sup_{\ep \in (0, 1)} \|\mathcal{P}u_0^{
\varepsilon}\|_{ \dot{H}^{-s}}<+\infty$, then
 we
 have the following uniform in $\ep$ time decay estimates for  $(n,\na\psi,v)$.  There exists  $C>0$ which does not depend on $\ep$, such that
 \beqs
     \|\na^{l}(n,\na\psi,v) (t)\|_{H^{3-l}}\leq C \min \{\ep,(1+t)^{
     -\min\{\f{l+s}{2+l+s},\f{1}{3}-\}}\}
     \eeqs
     where $l=0,1,2$ and $a^{-}$ stands for a real number smaller but arbitrarily close to $a.$

\end{thm}

\begin{rmk}
If in addition, $\mathcal{P}u_0^{\varepsilon}$ is in $H^M$,
 where $3\leq M\leq \sigma-2$, 
 then the solution to \eqref{NSPP} constructed above belongs to  $C([0,+\infty),H^M)$.
Besides, as we do not assume that  $\|\mathcal{P}u_0^{\varepsilon}\|_{H^M}$ is small, we have some time decay estimate in terms of the slow variable $'\ep t'$:
  \beqs
 \|\na^{k}(n,\na\psi,v)(t)\|_{H^{M-k}}\leq  C(1+\ep t)^{-\min\{\f{k+s}{2},\f{1}{3}-\}}
  \eeqs
  where $k=0,1,2\cdots M-1$.
\end{rmk}
Inspired by \cite{MR3005540} \cite{MR2917409} , we use merely energy estimates to prove global existence.
By using a modified energy functional  $\tilde{\mathcal{E}}_M$ that roughly controls the same Sobolev norms as
the usual energy functional
$$
    \mathcal{E}_M(n,v,\na\psi)
    =\sum_{|\alpha|\leq M}\f{1}{2}\int\rho|\pa v|^2+|\pa n|^2+|\pa \na \psi|^2\d x,
$$
for $M\geq 3$, we shall get that  if $\mathcal{E}_3 \leq \delta \ep^2$, and $\delta$ small enough, we have a positive constant $c$ such that the inequality
\beq\label{ineqenergy}
\pt \tilde{\mathcal{E}}_M +c\ep (\|n\|_{H^M}+\|\na u\|_{H^M}^2)\lesssim \delta^3 \ep^2(1+t)^{-\f{5}{3}}
\eeq
holds. Note that the interest of this modified functional is that it detects also damping of the $n$ component.
The global existence then follows from continuation arguments.

For the decay estimate, we first prove that the solution remains bounded in  $\dot{H}^{-s}$ if the initial data is in $\dot{H}^{-s}$. Then by using an interpolation inequality and \eqref{ineqenergy}, we can  obtain the energy inequality:
$\pt \tilde{\mathcal{E}}_M+c\ep(\tilde{\mathcal{E}}_M)^{1+\f{1}{s}}\lesssim \ep^{2 } (1+t)^{-\f{5}{3}}$
from which we get the desired decay estimate.

Once we have proven  Theorem \ref{thminviscid} and Theorem \ref{thmp}, Theorem \ref{thmelectric} is an easy consequence of them.

In the last part of the paper, we shall explain how  we can  also handle
the Navier-Stokes-Poisson system for the ions dynamics (INSP) introduced in
(\ref{NSPION}) by using the same approach. Note that we have used the so-called linearized approximation 
since in the (INSP) system,  we have replaced the Poisson equation  $ \Delta \phi_{+}^{\ep}=\rho_{+}^{\ep}-e^{-\phi_{+}^{\ep}}$,
 by a linearized version.  This is not a stringent assumption since we are again dealing with small perturbations of the constant equilibrium
  $(1, 0)$.
 For the  Euler-Poisson system describing ions dynamics (IEP) (that is $\varepsilon=0$ in \eqref{NSPION}), global smooth irrotational solutions with small amplitude  have been constructed
 by Guo and Pausader \cite{MR2775116}. The idea is again  to find dispersive estimates for the linearized system (which turn out to be  weaker than the one of the  linear Klein-Gordon equations) and  to use the  normal form method.
 Nevertheless, the analysis for this model is much more involved. Indeed,
 the dispersion relation  is closer to the one of the wave equation  which leads to the appearance of "time resonances". For example, the 'time resonances' of the phase function
 $\Phi_{++}=-p(\xi)+p(\xi-\eta)+p(\eta)$,$(p(\xi)=|\xi|\sqrt{\f{2+|\xi|^2}{1+|\xi|^2}})$ is $\{\eta=0\}\cup\{\xi-\eta=0\}$.
 After integration in time, the multilinear operators now have a singular kernel and
  to control them the use of  $\dot{H}^{-1}$ norms is needed.

 We now state the counterpart  of Theorem \ref{thmelectric}.

 \begin{thm}\label{thmion}
  Let us fix some absolute number $\kappa >0$ small enough.
 There exists  $\delta_2>0$  such that for any family of initial conditions that satisfy for every $\ep \in (0,1]$
 the estimates
\beno
\||(\rho_{+0}^{\varepsilon}-1,\mathcal{P}^{\perp}u_{+0}^{\varepsilon})\|_{W^{\si+3,8_{\kappa}'}}+\||\na|^{-1}(\rho_{+0}^{\varepsilon}-1,\mathcal{P}^{\perp}u_{+0}^{\varepsilon})\|_{H^{N}}\leq \delta_{2},\\
\|\mathcal{P}u_{+0}^{\varepsilon}\|_{H^{3}}\leq \delta_{2} \varepsilon
\eeno
with $8_{\kappa}=\f{8}{1-3\kappa},$ $8_{\kappa}'=\f{8_{\kappa}}{8_{\kappa}-1}$, $\sigma\geq 6$, $N\geq 2\sigma+1$, then
we have that  for every  $\varepsilon \in(0,1]$  there exists a unique global solution for system \eqref{NSPION} in $C([0,+\infty),H^{3})$.
Besides, if $\sup_{\ep\in (0, 1]}
\|\mathcal{P}u_{+0}^{\ep}\|_{ \dot{H}^{s}}<+\infty$ with $s<\f{3}{8}$, then we have the following time decay estimates.
 There exists $C>0$ such that for every $\ep \in (0, 1]$, we have the estimate
\beqs
 \|(\rho_{+}^{\varepsilon}-1, u_{+}^{\ep})\|_{W^{1,\infty}}\leq C\big( \min\{\ep,(1+ t)^{-\min\{\f{s}{2+s},\f{\kappa}{2}\}}\}
 +(1+t)^{-(1+\kappa)}\big), \quad \forall t \geq 0.
 \eeqs
 \end{thm}
 \textbf{Organization of the paper.} In the second section, we introduce some notations. In Section 3, we establish some useful preliminary estimates (in particular linear decay estimates) in order  to prove Theorem \ref{thminviscid}. Then, we prove Theorem  \ref{thminviscid} and Theorem \ref{thmp}   in Section 4, and Section 5 respectively.
In Section 6, we will explain briefly the modifications needed  to extend the results to general pressure laws and density-dependent viscosities.
 In Section 7, we shall explain how to  deal with the ions system.
Finally, we recall some classical inequalities in the  appendix.

\section{Some Notations}\label{notations}
\begin{itemize}
\item  We define $\varphi_{0} (\xi),\chi(\xi)$ as two radial symmetric $C_{c}^{\infty}$ functions, which are both supported on $\{\xi\big||\xi|\leq 2\}$ and equal to 1 when
$\{\xi\big||\xi|\leq 1\}$, and $\tilde{\chi}\in C_{c}^{\infty}$ equal to 1 on
$\{\xi\big||\xi|\leq 3\}$ and vanish on
$\{\xi\big||\xi|\geq 4\}$.
\item  We shall also use the truncation function  $\chi_{\vep,\kpz}(\xi)=\chi\big(\sqrt{\frac{\varepsilon}{\kappa_0}}\xi\big)$
  in the proof of Theorem \ref{thminviscid}.
  
 \item We denote by  $m(D)$ the Fourier multiplier defined by $m(D)f=\cF^{-1}(m(\cdot)\cF{f}(\cdot))$.

\item   We also introduce the classical Littlewood-Paley decomposition:
 define $\varphi(\xi)=\varphi_{0}(\xi)-\varphi_{0}(2\xi)$
and $\varphi_{j}=\varphi(\f{\xi}{2^{j}})$, $j\in \mathbb{N}^{*}$, $\Delta_j f=\cF^{-1}(\varphi_j(\xi)\cF f(\xi)$), $j\in \mathbb{N}$. The norm in  the inhomogeneous Besov space $B_{p,r}^s$($p,r\geq 1, s\in\mathbb{R}$)
 is defined by 
$\|f\|_{B_{p,r}^s}=(\sum_{j=0}^{\infty}\|\Delta_j f\|_{L^p}^{r}2^{jsr})^{\f{1}{r}}.$
\item  For a given function $m(\zeta, \eta)$, we
  define the  bilinear operator $T_m(f,g)$ as:
 \ben\label{eqbilinear}
 T_m(f,g)&\define&\cF^{-1}(\int m(\xi-\eta,\eta)\hat{f}(\xi-\eta)\hat{g}(\eta)\d \eta)\nonumber\\
 &=&\f{1}{(2\pi)^3}\int m(\zeta,\eta)\hat{f}(\zeta)\hat{g}(\eta)e^{ix(\zeta+\eta)}\d \zeta\d \eta
 \een
 
 \item   We use  $\langle \cdot\rangle$  for $\sqrt{1+|\cdot|^2}.$
 \item We denote $a^{+}$ a constant which is larger and  arbitrarily
 close to $a$.
 \item  We shall always use the notation $\lesssim$ for $\leq C$ for  $C>0$ a harmless number that can
  be chosen independent of $\ep \in (0, 1]$ and $t >0$.

  \end{itemize}

\section{Preliminary estimates}\label{section useful lemmas}

In this section, we analyze the system \eqref{NSPlow1}.
At first, we observe that as long as a smooth solution exists on an  interval $[0,T]$, then $\omega(t) \triangleq \curl  u(t)=0$ on this interval.
Indeed, by taking  the $\curl$ in the second equation of system \eqref{NSPlow1}, we get the equation for $\omega$
 \beqs
 \left\{
\begin{array}{l}
\displaystyle \pt \om-\varepsilon \Delta \om+ \om \div u +(u\cdot \na)\om-(\om \cdot\na)u=0 \\

\displaystyle \om|_{t=0} =0
\end{array}
\right.
 \eeqs
  By the standard energy estimate and Gr\"{o}nwall's inequality, we have
  \beqs
  \|\om(t)\|_{L^2}^2\leq e^{c\int_{0}^{T}\|\na u(s)\|_{L^{\infty}}\d s}\|\om _0\|_{L^2}^2 =0
 \eeqs
 A direct consequence is that $u=\mathcal{P}^{\bot}u=\na{\Delta}^{-1}\div u$. Thus by using the  identity $\curl\curl u=-\Delta u+\na\div u$, the second equation of system \eqref{NSPlow1} turns out to be:
 $$\pt u+u \cdot {\na u}-2\varepsilon \Delta u+
\nabla \varrho-\nabla \phi=0.$$
Based on the above facts, let us set
 \beqs
 h=\f{\lnr}{|\na|} \vr,\quad c=\f{\div}{|\na|} u ,\qquad V=(h,c)^{\top},
 \eeqs
we then obtain  that $(h,c)$ satisfies the system:
  \beq \label{eqsym}
 \left\{
\begin{array}{l}
\displaystyle \pt h + \lnr c=-\lnr\f{\div} {|\na|}\big(\f{|\na|}{\lnr}h\cdot \R c\big)=\lnr \R^{*}\big(\f{|\na|}{\lnr}h\cdot \R c\big), \\
\displaystyle \pt c-\lnr h-2\vep \Delta c=-\f{1}{2}\f {\div}{|\na|} \na |\R c|^2=\f{1}{2}|\na| |\R c|^2,\\
\displaystyle h|_{t=0}=\f{\lnr} {|\na|}\vr_0,c|_{t=0}= \f {\div}{|\na|}u_0. \\
\end{array}
\right.
 \eeq
which we shall rewrite as:
\begin{equation}\label{eqV}
\pt V+  A ( D)
V=\left(
  \begin{array}{c}
    \lnr \R^{*}\big(\f{|\na|}{\lnr}h\cdot \R c\big) \\
  \R^{*} \na |\R c|^2\\
  \end{array}
\right)
\triangleq B(V,V), \quad  A ( D)= \left(
  \begin{array}{cc}
    0&\lnr\\
    -\lnr&-2\vep \Delta\\
  \end{array}
\right).
\end{equation}
In the above systems,
 $\R$ is the  vectorial Riesz transform:
 $\R =\f{\na}{|\na|}$ and $\R^{*}=-\f{\div}{|\na|}$ is its adjoint for the $L^2$ scalar product.\\
By elementary computations, we get that the eigenvalues of $-A(\xi)$ are:
   \beq \label{eqA}
   \lambda _{\pm}=-\vep |\xi|^2\pm i \sqrt{1+|\xi|^2-\vep^2 |\xi|^4}\define -\vep |\xi|^2\pm ib(\xi)
   \eeq
   where we cut the lower half imaginary axis 
   to define the  square root of a complex number. Note that $b$ is in fact dependent on $\ep$, but we do not write it explicitly for simplicity. One can  easily check that the Green matrix is
   \beno
   e^{-t A(\xi)}=\frac{1}{\lambda_{+}-\lambda_{-}}
   \left(
  \begin{array}{cc}
    \lp e^{\lm t}-\lm e^{\lp t}&(e^{\lm t}-e^{\lp t})\lxr\\
    (e^{\lp t}-e^{\lm t})\lxr&\lp e^{\lp t}-\lm e^{\lm t}\\
  \end{array}
\right)
   \define
    \left(
  \begin{array}{cc}
   \mathcal{G}_1(t,\xi)&-\mathcal{G}_2(t,\xi)\\
   \mathcal{G}_2(t,\xi)&\mathcal{G}_3(t,\xi)\\
  \end{array}
\right).
   \eeno
  Note that  $\mathcal{G}_1,\mathcal{G}_2,\mathcal{G}_3$ are actually  well defined everywhere since there is no singularity when $\lambda_{+}=\lambda_{-}$ (see the proof of Lemma \ref{lemmahf}).

   Let us observe  that for low frequencies, ie, when $\vep|\xi|^2\leq 2\kpz<<1$ (since the eigenvalues
    do not cross), we can smoothly diagonalize  $A$ under the form:
   \ben\label{Qdef}
A(D)&=&\left(
  \begin{array}{cc}
    1&1\\
    -\f{\lm(D)}{\lnr}&-\f{\lp(D)}{\lnr}\\
  \end{array}
\right)
\left(
  \begin{array}{cc}
    -\lm & 0\\
   0 & -\lp \\
  \end{array}
\right)
\left(
  \begin{array}{cc}
    \lp&\lnr\\
    -\lm &-\lnr\\
  \end{array}
\right)\f{1}{2ib}\nonumber\\
&\define&Q\left(
  \begin{array}{cc}
    -\lm & 0\\
   0 & -\lp \\
  \end{array}
\right)Q^{-1}, \quad Q^{-1}=\left(
  \begin{array}{cc}
    \lp&\lnr\\
    -\lm &-\lnr\\
  \end{array}
\right)\f{1}{2ib}.
\een
 Since by Duhamel principle, we can rewrite \eqref{eqV} as
 \beq
 \label{DuhamelV}
 V=e^{-tA}V_0+\int_{0}^{t}e^{-(t-s)A}B(V,V)(s)\d s,
 \eeq
we shall first study the main   properties of $e^{-tA}$ and $B(V,V)$ in the following two subsections.

 \subsection{Linear estimates}
 This subsection is devoted to the study of $e^{-tA}$.
 We shall  carry out the analysis in any space dimension  $\mathbb{R}^{d}$, $d\geq 2$ although in this paper, we only use it for dimension $3$. The behavior will be different for low frequencies $\ep |\xi |^2 \lesssim 1$ where
  uniform in $\ep$ decay estimates will come from the dispersive behavior and for high frequencies
     $\ep |\xi |^2 \gtrsim 1$ where dissipative damping dominates.

 \subsubsection{Linear estimates for low frequencies: $\ep|\xi|^2\leq 2\kpz$}
 For low frequencies, we can get decay estimates that are similar to the  ones of  the linear Klein-Gordon equation
 by using  dispersive properties.
Let us recall that  we  use the notation  $\chi_{\vep,\kpz}(\xi)=\chi\big(\sqrt{\f{\varepsilon}{\kappa_0}}\xi\big)$ (see Section \ref{notations}).
We will fix the threshold $\kpz$ in  the proof of Lemma \ref{lemlow}.

   \begin{lem}\label{lemlow}
   There exists $\kappa_{0} >0$, small enough such that
  uniformly for  $\ep \in (0, 1], $ and for every
   $f\in B_{1,2}^{d}$, we have the estimate
  \beqs
   \|e^{itb(D)}\chidid f\|_{B_{\infty,2}^{0}}\lesssim_{\kpz} (1+|t|)^{-\f{d}{2}}\|f\|_{B_{1,2}^{d}}.\quad \forall t\in \mathbb{R}
  \eeqs
   \end{lem}
  \begin{proof}
    Note that on the support of $\chi_{\ep,\kpz}$, $b(\xi)$ behaves like $\lxr$, thus, to prove this lemma, we can follow the  proof of the dispersive estimate for the  linear Klein-Gordon equation
   by keeping track of the perturbation. The key point is that this dispersive estimate is uniform with respect to  $\vep$.

  The proof will thus follow from  the following two lemmas.

  \begin{lem}
  For every  $\kpz$  small enough,
     we have uniformly for  $\ep\in(0,1],$ the estimate
  \beqs
  \|e^{itb(D)}\chidid \varphi_{0}(D)f\|_{L^{\infty}}\lesssim
  (1+|t|)^{-\f{d}{2}}\|f\|_{L^{1}}. \quad \forall t\in \mathbb{R}
  \eeqs
  \end{lem}
   \begin{proof}
   By  using the Fourier transform, we only need to show that
   \beqs
   \|\int e^{itb(\xi)}e^{ix\cdot \xi}\chidix \varphi_{0}(\xi) \d \xi\|_{L^{\infty}}\lesssim_{\kpz} (1+|t|)^{-\f{d}{2}}.
   \eeqs
  At first, note that:
   \beqs
   \|\int e^{itb(\xi)}e^{ix\cdot \xi}\chidix  \varphi_{0}(\xi) \d \xi\|_{L^{\infty}}\lesssim \|\varphi_0\|_{L^1}.
   \eeqs
  Thus in the following, we only prove that:
     \beqs
   \|\int e^{itb(\xi)}e^{ix\cdot \xi}\chidix  \varphi_{0}(\xi) \d \xi\|_{L^{\infty}}\lesssim_{\kpz} |t|^{-\f{d}{2}}.
   \eeqs
Let us write
    \beqs
    \int e^{itb(\xi)}e^{ix\cdot \xi}\chidix  \varphi_{0}(\xi) \d \xi
    =\int e^{it\Phi(\xi)}\chidix  \varphi_{0}(\xi) \d \xi, \quad  \Phi(\xi)=b(\xi)+\f{x}{t}\cdot \xi.
    \eeqs
   By direct computations, the first and second derivative of $\Phi(\xi)$ are given by the following expressions:
    \beqs
    \na_{\xi} \Phi(\xi)=\na_{\xi}b+\f{x}{t}=\f{(1-2\ep^{2}|\xi|^{2})}{b(\xi)}\xi +\f{x}{t}
    \eeqs
    \beqs
    \partial_{\xi_i}\partial_{\xi_j}\Phi(\xi)=\f{1-2\ep^2|\xi|^2}{b(\xi)}(\delta_{ij}-\f{(1+4\ep^2)\xi_i\xi_j}{b^{2}(\xi)(1-2\ep^2|\xi|^2)}).
    \eeqs
   We then obtain that on  the support of $\chidix \varphi_{0}(\xi)\subset{\{\xi\big|   |\xi|\leq 2,\ep|\xi|^2\leq 2\kpz\}}$,
    we have
    \beno
   \det (\na^{2}\Phi(\xi))=
   (\f{1-2\ep^2|\xi|^2}{b(\xi)})^{d}\f{1-6\ep^2|\xi|^{2}-3\ep^{2}|\xi|^{4}+2\ep^{4}|\xi|^{6}}{b^{2}(\xi)(1-2\ep^2|\xi|^2)}
& \geq & \f{(1-4\ep\kpz)^{d}(1-12\ep\kpz -12\kpz^{2})}{b^{d+2}(\xi)}\nonumber\\
&\geq&\f{1}{2^{d+1}b^{d+2}(\xi)}\geq \f{1}{2^{d+1}\cdot 5^{\f{d+1}{2}}}
    \eeno
   for  $\ep \in (0,1]$  as long as $\kpz$ is  small enough.

  By  using the classical stationary phase lemma (we refer to \cite{MR2847755},\cite{MR2952218}
   for example), we  arrive at the desired result.
\end{proof}

   \begin{lem}
   Suppose $d\geq 2.$ For every $\kappa_{0}>0$  small enough and  for every $\lambda \geq 1$,
    we have uniformly for $\ep\in(0,1],$
   \beqs
   \|e^{itb(D)}\chidid \varphi(\f{D}{\lambda})f\|_{L^{\infty}}\lesssim_{\kpz}|t|^{-\f{d}{2}}\lambda^{\f{d+2}{2}}\|f\|_{L^{1}},\qquad \forall t\in \mathbb{R}
  \eeqs
    \beqs
   \|e^{itb(D)}\chidid \varphi(\f{D}{\lambda})f\|_{L^{\infty}}\lesssim_{\kpz}\lambda^{d}\|f\|_{L^{1}}.
  \eeqs
   \end{lem}

   \begin{proof} It suffices to prove:
    \beqs
   \|\int e^{itb(\xi)}e^{ix\cdot \xi}\chidix \varphi(\f{\xi}{\lambda}) \d \xi\|_{L^{\infty}}\lesssim_{\kpz} |t|^{-\f{d}{2}}\lambda^{\f{d+2}{2}}, \qquad \forall t\in \mathbb{R}
   \eeqs
      \beqs
   \|\int e^{itb(\xi)}e^{ix\cdot \xi}\chidix \varphi(\f{\xi}{\lambda}) \d \xi\|_{L^{\infty}}\lesssim_{\kpz}\lambda^{d}.
   \eeqs
 The second estimate just  comes from a  change of variable, we thus only need to  prove the first one. We will also restrict ourselves to the case $t>0$ as the other case is similar.
   As $\chi_{\ep,\kpz},\phi,b$ are all radially symmetric, we actually have:
   \beno
   \int e^{itb(\xi)}e^{ix\cdot \xi}\chidix \varphi(\f{\xi}{\lambda}) \d \xi &=&\int_{0}^{+\infty}e^{itb(r)}\chi_{\ep,\kpz}(r) \phi(\f{r}{\lambda})\cF (\si_{\mathbb{S}^{d-1}})(|x|r)r^{d-1}\d r \nonumber\\
   &=& \lambda^{d}\int_{\f{1}{2}}^{2}e^{itb(\lambda r)}\chidir\phi(r) \cF (\si_{\mathbb{S}^{d-1}})(\lambda|x|r)
   r^{d-1}\d r
   \eeno
   where we have used the fact that the Fourier transform of  the Lebesgue measure on the sphere $\cF(\si_{\mathbb{S}^{d-1}})(x)$ is (see  \cite{MR2445437}, Appendix B)
   \beqs
   \cF(\si_{\mathbb{S}^{d-1}})(x)=|x|^{-\f{d-2}{2}}J_{\f{d-2}{2}}(|x|)=e^{i|x|}Z(|x|)-e^{-i|x|}\bar{Z}(|x|)
   \eeqs
   where $J_{\f{d-2}{2}}(s)$ is the Bessel function and $Z(s)$ satisfies \cite{MR2445437}  for all integer $k\geq 0$ and all $s>0,$
   \beq \label{eq2}
   |\partial^{k}Z(s)|\lesssim _{k,d}(1+s)^{-\f{d-1}{2}-k}.
   \eeq
Therefore, we can write:
   \beqs
   \int e^{itb(\xi)}e^{ix\cdot \xi}\chidix \varphi(\f{\xi}{\lambda}) \d \xi=\sum_{\pm} \lambda^{d}\int _{\f{1}{2}}^{2}e^{it\Phi_{\lambda}^{\pm}(r)}\chidir\phi(r)Z_{\pm}(\lambda |x| r) \d r
   \eeqs
   where $Z_{\pm}=Z,\bar{Z}$ and $\Phi_{\lambda}^{\pm}(r)=b(\lambda r)\pm \lambda r\f{|x|}{t}$.

For the $'+'$
   case, we can easily check that
\beqs
   \partial_{r}\Phi_{\lambda}^{+}=\lambda b'(\lambda r)+\f{\lambda |x|}{t}=\lambda \f{\lambda r(1-2\ep^2\lambda^2 r^2)}{b(\lambda r)}+\f{\lambda |x|}{t}\geq  \f{\lambda r(1-2\ep^2\lambda^2 r^2)}{b(r)}\gtrsim_{\kpz}\lambda
   \eeqs
   as long as $\kpz$ is small enough.
   Moreover,  for $k\geq 2$, we have
   \beqs
   |\partial_{r}^{k}\Phi_{\lambda}^{+}|=|\lambda ^{k}\partial_{r}^{k}b (\lambda r) |\lesssim_{\kpz} \lambda.
   \eeqs
This yields by  direct induction, that
   \beq \label{eq3}
  | \partial_{r}^{k}\f {1}{\partial_{r}\Phi_{\lambda}^{+}}| \lesssim_{\kpz,k} \lambda^{-1}.
   \eeq
  In addition, by \eqref{eq2}, we have on the support of $\phi$, that
   \beq \label{eq4}
   \partial_{r}^{k} \big(Z_{\pm}(\lambda |x| r)\big)\lesssim (\lambda |x|)^{k}(1+\lambda |x| r)^{-\f{d-1}{2}-k}\lesssim(1+\lambda |x| r)^{-\f{d-1}{2}}\leq 1.
   \eeq
 Consequently,  by using   the classical (non-)stationary phase lemma  and \eqref{eq3},\eqref{eq4}, we have that for any integer $N\geq 0$
   \beqs
  \big |\int _{\f{1}{2}}^{2}e^{it\Phi_{\lambda}^{+}(r)}\chidir\phi(r)Z_{+} \d r\big|\lesssim _{\kpz}(\lambda t)^{-N}.
   \eeqs
To conclude,
 we choose $N=\f{d}{2}$ if $d$ is even, and  we choose $N=\f{d-1}{2}$ and $N=\f{d+1}{2}$ if $d$ is odd to get:
    \beqs
  \big |\int _{\f{1}{2}}^{2}e^{it\Phi_{\lambda}^{+}(r)}\chidir\phi(r)Z_{+} \d r\big|\lesssim _{\kpz}(\lambda t)^{-\f{d}{2}}\lesssim \lambda^{-\f{d-2}{2}}t^{-\f{d}{2}}
   \eeqs
 which is the desired result for '+' case.

    For the $'-'$ case, the first derivative of $\Phi_{\lambda}^{-}(r)$ can vanish. Indeed, we have
    \beqs
   \partial_{r}\Phi_{\lambda}^{-}=\lambda b'(\lambda r)-\f{\lambda |x|}{t}=\lambda \f{\lambda r(1-2\ep^2\lambda^2 r^2)}{b(\lambda r)}-\f{\lambda |x|}{t}.
   \eeqs
 At first,
    if $|x|\leq \f{t}{100}$ or $t\leq \f{|x|}{100} $ , then we have
    \beqs
    | \partial_{r}\Phi_{\lambda}^{-}|\gtrsim_{\kpz}\f{\lambda}{t}(|x|+t)
    \eeqs
and  for $k\geq 2$,
    \beqs
      |\partial_{r}^{k}{\Phi_{\lambda}^{-}}|=|\partial_{r}^{k} {\Phi_{\lambda}^{+}}|\lesssim_{\kpz} \lambda \leq \f{\lambda}{t}(|x|+t).
    \eeqs
 As before, this yields  by induction, for any $l\geq 0,$
    \beqs
  | \partial_{r}^{l}\f {1}{\partial_{r}\Phi_{\lambda}^{-}}| \lesssim_{\kpz} \big(\lambda (|x|+t)\big)^{-1}t.
   \eeqs
Consequently, by using  again the   (non-)stationary phase method, we get:
    \beqs
   |\int _{\f{1}{2}}^{2}e^{it\Phi_{\lambda}^{-}(r)}\chidir\phi(r)\bar{Z}(\lambda r) \d r|\lesssim _{\kpz} \langle\lambda(t+|x|)\rangle^{-N}.
   \eeqs
 If $|x|\approx t$, ie $\f{1}{100}\leq\f{|x|}{t}\leq 100$, we first notice that if $\kpz$ is sufficient small, then on the support of $\chi_{\ep,\kpz}(\lambda r)$, one has $\partial_{r}^{2}\Phi_{\lambda}^{-} =\lambda ^2\f{1-6\ep^2(\lambda r)^2-3\ep^2(\lambda r)^4+2\ep^4 (\lambda r)^6}{b^{3}(\lambda r)}\gtrsim_{\kpz}{\lambda}^{-1}$. Combining this fact with the behavior  of $Z$ (see \eqref{eq2}),
  we then apply Van der Corput Lemma (see for example \cite{stein2016harmonic})
    to get
   \beno
   && |\int _{\f{1}{2}}^{2}e^{it\Phi_{\lambda}^{-}(r)}\chidir\phi(r)\bar{Z}(\lambda |x|r)\d r\\
   &\lesssim_{\kpz}&  (\lambda^{-1} t)^{-\f{1}{2}}\left(\f{1}{2^{d-1}} \chi(\f{1}{2}(\f{\ep}{\kpz})^{\f{1}{2}}\lambda )\phi(\f{1}{2})\bar{Z}(\f{1}{2}\lambda |x|)+ \sup_{r}\partial_{r}(r^{d-1}\chidir \phi(r)\bar{Z}(\lambda |x|r)) \right)\\
   &\lesssim_{\kpz}&(\lambda^{-1} t)^{-\f{1}{2}}(1+\lambda |x|)^{-\f{d-1}{2}}
   \lesssim_{\kpz} \lambda^{-\f{d}{2}+1}t^{-\f{d}{2}}.
   \eeno
   This ends the proof.
    \end{proof}
    Once we have the above two lemmas, we can sum the frequencies over the dyadic decomposition to get Lemma \ref{lemlow}.
    \end{proof}


 From now on, we fix  $\kpz$ sufficiently small independent of $\ep$  such that the statement of Lemma \ref{lemlow} and proposition \ref{elementary for phase} in Appendix holds.

\begin{cor}\label{corgreenlowf}
  For $j=1,\,2,\,3$  and $f \in B_{1,2}^{d}$, we have uniformly in $\ep \in (0, 1]$ the estimate
    \beqs
   \|\mathcal{G}_{j}(t,D)\chidid f\|_{B_{\infty,2}^{0}}\lesssim
   (1+t)^{-\frac{d}{2}}\|f\|_{B_{1,2}^{d}}, \quad \forall t>0.
  \eeqs
\end{cor}
\begin{proof}

We focus on the proof  for $\mathcal{G}_1$, the other two terms can be handled with  similar arguments. Simple computations show that:
\beqs
\cG_1(t,D)=\f{1}{2}e^{\ep t\Delta} \big((e^{ib(D)t}+e^{-ib(D)t})+i \f{\vep \Delta}{b(D)}(e^{ib(D)t}-e^{-ib(D)t})\big).
\eeqs

 By Lemma \ref{lemlow} and the continuous property of the  operator $e^{\vep t\Delta}$ on  $L^p $, $1\leq p\leq +\infty$, it suffices for us to show the same result as in  Lemma \ref{lemlow} when $e^{itb(D)}\chidid$ is changed into  $e^{itb(D)}\chidid \f{\ep \Delta}{b(D)}$.  The proof is similar to that of Lemma \ref{lemlow}  once we notice that on the support of $\chi_{\vep,\kpz}$,
 $
 \partial_{r}^{k}\big(\f{\ep r^2}{b(r)}\big)\lesssim C
$, we thus omit the details.

 \end{proof}

   \subsubsection{Linear estimates for high frequencies: $\ep|\xi|^2\geq \kpz$}

\begin{lem}\label{lemmahf}
 There exists $c_0>0$  such that, for $j=1,2,3$ and for every $\ep \in (0, 1]$, we have the estimate
 \beqs
 |(1-\chi_{\ep,\kpz}) \mathcal{G}_j(t,\xi)|\lesssim e^{-c_0t}, \quad \forall t \geq 0.
 \eeqs
 \end{lem}
 \begin{proof}
 There are two cases:\\
 Case 1: $1+|\xi|^2\geq \ep^2 |\xi|^4$.
  We first observe that
 \beqs
 \mathcal{G}_1(t,\xi)=\frac{\lp e^{\lm t}-\lm e^{\lp t}}{\lp-\lm}=e^{-\ep |\xi|^{2}t}\big(\cos(bt)+\ep \f{\sin(bt)}{b}|\xi|^2\big),
  \,
 \eeqs
 \beqs
 \mathcal{G}_3(t,\xi)=\frac{\lp e^{\lp t}-\lm e^{\lm t}}{\lp-\lm}=e^{-\ep |\xi|^{2}t}
 ((\cos(bt)-\ep \f{\sin(bt)}{b}|\xi|^2)).
\eeqs
 Therefore, for $k=1,3$, we have:
$$
 |\cchidix \mathcal{G}_k| \leq |\cchidix| e^{-\ep |\xi|^{2}t}(1+\ep |\xi|^{2}t)
 \lesssim  |\cchidix|e^{-\frac{1}{2}\ep |\xi|^{2}t} \lesssim e^{-\f{1}{2}\kpz t}.
 $$
 For $\mathcal{G}_2$,
 if $b(\xi)\geq \f{\lxr}{2},$ then we have
 \beqs
 |\cchidix \mathcal{G}_2|=|\cchidix |e^{-\ep |\xi|^{2}t}\big|\frac{\sin(bt)}{b}\lxr\big| \lesssim e^{-\f{1}{2}\kpz t}.
 \eeqs
 If $b(\xi)\leq \f{\lxr}{2},$
 we have $\lxr \leq \f{2}{\sqrt{3}}\ep|\xi|^2$, thus
 \beqs
  |\cchidix \mathcal{G}_2|\leq e^{-\ep |\xi|^2 t}2\ep|\xi|^2 t I_{\{\ep|\xi|^2\geq \kpz\}}\lesssim  e^{-\f{1}{2} \kpz t}.
 \eeqs
 Case 2.  $1+|\xi|^2\leq \ep^2 |\xi|^4$.  Let us introduce  $\bt=\sqrt{\ep^2 |\xi|^4-(1+|\xi|^2)} $ then $\lambda_{\pm}=-\ep |\xi|^2 \mp \bt (\xi)$.

Firstly, we have
 \beno
 |\cchidix \mathcal{G}_1|&=&|\cchidix|e^{\lm t}(1+\frac{1-e^{-2\bt t}}{2\bt}(-\lm))|\\
 &\leq& e^{\lm t}(1+(-\lm )t)
 \lesssim e^{\frac{1}{2}\lm t }
 \lesssim e^{-\frac{1}{4\ep}t}\lesssim e^{-\frac{1}{4}t}
 \eeno
 Here we have used the fact $\lp >-2\ep |\xi|^2, \lm=\frac{1+|\xi|^2}{\lp}\leq -\f{1}{2\ep}$.

Secondly, we also have
 \beno
  |\cchidix \mathcal{G}_2|&\leq&\big|e^{\lm t}(\f{1-e^{-2\bt t}}{2\bt})\big|
  \lesssim e^{\lm t}t
  \lesssim  2\ep  e^{-\f{1}{4\ep}t}\lesssim e^{-\f{1}{4\ep}t}\lesssim e^{-\f{1}{4}t}.
  \eeno
Finally, for
  $\mathcal{G}_3=e^{\lm t}[1+\lp \frac{1-e^{-2\bt t}}{2\bt}]$, we write
\begin{itemize}
 \item if $\bt >\frac{\ep |\xi|^2}{2},$  then
  $
  |\cchidix \mathcal{G}_3| \leq e^{\lm t}(1+\frac{-\lp}{\bt})\lesssim5 e^{-\lm t}\lesssim e^{-\f{1}{2}t};
 $
\item  if $0\leq \bt \leq \f{\ep |\xi|^2}{2},$ then $\lm \leq -\frac{1}{2}\ep |\xi|^2$, and  therefore,
  \beno
  |\cchidix \mathcal{G}_3|\leq e^{\lm t}(1+(-\lp)t)I_{\{\ep|\xi|^2\geq \kpz\}}
  &\lesssim &e^{\lm t}(1+2\ep |\xi|^2 t)I_{\{\ep|\xi|^2\geq \kpz\}}\\
&  \lesssim &  e^{-\f{1}{4}\ep|\xi|^2 t}I_{\{\ep|\xi|^2\geq \kpz\}}
 \lesssim  e^{-\f{1}{4}\kpz t}.
  \eeno
  \end{itemize}
  This ends the proof.
 \end{proof}

\subsubsection{Additional estimates of $e^{-t A}$}
   \begin{lem} \label{LemmaHN}
    For j=1,2,3, for every  $s \geq 0$ and uniformly for $\ep \in (0, 1]$, we have the estimate
    \beqs
    \|\mathcal{G}_j(t,D)f\|_{H^s}\lesssim
    \|f\|_{H^{s}}.
    \eeqs
   \end{lem}
  \begin{proof}
   We only need to show that  $|\mathcal{G}_j(t,\xi)|
  \leq C$. 
   Note that  we have proven in the last lemma that if $\ep |\xi|^2\geq \kpz$,
   then we have $|\mathcal{G}_j(t,\xi)|\leq e^{-c_0t}$. In the remaining region  $\ep |\xi|^2\leq 2\kpz $, we have,
  $$ |\mathcal{G}_1|=|e^{-\ep |\xi|^{2}t}(\cos(bt)+\ep |\xi|^2\f{\sin(bt)}{b})|\leq e^{-\ep |\xi|^{2}t}(1+\ep |\xi|^{2}t )\leq C,$$
   $$|\mathcal{G}_2|=|e^{-\ep |\xi|^{2}t}\f{\sin(bt)}{b}\lxr|\leq e^{-\ep |\xi|^{2}t}\frac{\lxr}{\sqrt{1-4\kpz^{2}+|\xi|^2}}\leq C.$$
  The estimate of $\cG_3$ is similar to that of $\cG_1$. This ends the proof.
 \end{proof}
By  combining Corollary \ref{corgreenlowf} and Lemma \ref{LemmaHN}, we also obtain:
   \begin{cor}\label{corlowf}
   For $p\geq 2$, we have uniformly for $\ep \in (0, 1]$ the estimates
   \beno
   \|e^{itb(D)}\chidid f\|_{L^p}&\lesssim
   & (1+ |t|)^{-\f{d}{2}(1-\f{2}{p})}\|f\|_{W^{(1-\frac{2}{p})d,p'}},\\
   \|\mathcal{G}_j(t,D)\chidid f\|_{L^p}&\lesssim
   & (1+t)^{-\f{d}{2} (1-\f{2}{p})}\|f\|_{W^{(1-\frac{2}{p})d,p'}}.
   \eeno
   \end{cor}

    \begin{cor}\label{cor100}
    For j=1,2,3, we have uniformly for $\ep \in (0, 1]$ the estimate
    \beqs
    \|\mathcal{G}_j(t,D)f\|_{L^{\infty}}\lesssim (1+t)^{-\f{d}{2}}\|f\|_{W^{d ,1}}+e^{-c_0t}\|f\|_{H^{\f{d+1}{2}}}.
    \eeqs
    \end{cor}
    \begin{proof}[Proof of Corollaries \ref{corlowf},\, \ref{cor100}]

    For Corollary \ref{corlowf}, we can interpolate in a classical way between the estimates of
    Corollary \ref{corgreenlowf} and Lemma \ref{LemmaHN} and
    use the  embeddings $B_{p,2}^{s}\hookrightarrow  W^{s,p}$, $W^{s,p'}\hookrightarrow B_{p',2}^{s}$ with $p\geq 2,s\geq 0$.
    One can refer for instance  to the books \cite{bahouri2011fourier}
    \cite{MR2463316} for the relations between Besov spaces and Sobolev spaces.

     For Corollary \ref{cor100}, we 
    write  $\mathcal{G}_j(t,D)f=\mathcal{G}_j(t,D)\chidid f+\mathcal{G}_j(t,D)\cchidid f,$
    and the result follows from Corollary \ref{corlowf}, Lemma \ref{lemmahf} and the inequality:
    $\|\hat{f}\|_{L^1}\lesssim \|f\|_{H^{\f{d+1}{2}}}.$
\end{proof}

  \begin{lem}\label{lemmultiplier}
  Let us define the operators
  \beqs
  n_{1}(D)=|\na| \quad or \quad
   \ep \Delta  \tchidid \f{\div}{|\na|} \quad or \quad ib(D)\tchidid \f{\div}{|\na|},
  \eeqs
  \beqs
  n_2=\f{\ep\Delta+b(D)}{\lnr}\tchidid \quad  or \quad \R \quad or
  \quad \f{|\na|}{\lnr}.
  \eeqs
Then, for any $p\in (1,\infty),$ we have the estimate:
  \beqs
  \|n_1(D)f\|_{L^p}\lesssim \|f\|_{W^{1,p}}\qquad
  \|n_2(D)f\|_{L^p}\lesssim \|f\|_{L^p}
  \eeqs
  \end{lem}
\begin{proof}
We can apply the  H\"ormander-Mikhlin Theorem (we refer for instance  to Theorem 5.2.7 in \cite{MR2445437}). One can easily check  that $n_1(\xi)$, $ n_2(\xi)$ satisfy homogeneous 0 type conditions uniformly in $\ep\in (0,1]$.
    \end{proof}

 From the definition of $Q(D), \,Q^{-1}(D)$ (see \eqref{Qdef}),  we also have  the following property for $Q(D),\, Q^{-1}(D)$: 
     \begin{cor}\label{corQ}
     For any $1< p< +\infty$,
     $\chidid Q(D),\chidid Q^{-1}(D)$  are both continous in $L^{p}$ uniformly in $\ep \in (0, 1]$:
          \beno
     \|\chidid Q(D)F\|_{L^p}\lesssim
     \|F\|_{L^p}, \qquad
     \|\chidid Q^{-1}(D)F\|_{L^p}\lesssim
     \|F\|_{L^p}.
     \eeno
     \end{cor}

We will  also need to use some elementary parabolic estimates.

\begin{lem}\label{lemed}
 For any integer $k\in \mathbb{N}^{*}$ and $1<q< +\infty$, we have:
\beqs
\|e^{\ep t\Delta}(\ep \Delta)^{k}\chidid f\|_{L^{q}}\lesssim (1+t)^{-k}\|f\|_{L^{q}}.
\eeqs
\end{lem}

\begin{proof}
On the  one hand, by Young's inequality,  we have 
\beqs
\|e^{\ep t\Delta}(\ep \Delta)^{k}\chidid f\|_{L^{q}}
\lesssim t^{-k}\|\chidid f\|_{L^{q}}
\lesssim  t^{-k}\|f\|_{L^{q}}.
\eeqs
On the other hand, as $(\ep \Delta)^{k}\chidid$ is a $L^{q}$ multiplier, we also have:
\beqs
\|e^{\ep t\Delta}(\ep \Delta)^{k}\chidid f\|_{L^{q}}
\lesssim \|(\ep \Delta)^{k}\chidid f\|_{L^{q}}\lesssim \|f\|_{L^q}.\\
\eeqs
\end{proof}

\subsection{Nonlinear and bilinear estimates}

 \begin{lem}\label{lemnonlinear}
 For every  $1<p< +\infty,$ $\f{1}{p}=\f{1}{q_1}+\f{1}{r_1}=\f{1}{q_2}+\f{1}{r_2}$
$1<r_1,q_1< +\infty$, $1< r_2,q_2< +\infty$, we have the estimate
 \beq
 \|B(V,V)\|_{W^{s,p}}\lesssim \|V\|_{W^{s+1,q_1}}\|V\|_{L^{r_1}}+\|V\|_{L^{r_2}}\|V\|_{W^{s+1,q_2}}.
 \eeq
 \end{lem}
\begin{proof}
By the definition of B, the boundedness of the  Riesz transform  in $L^q$$(1 <q< +\infty)$ and  the Kato-Ponce inequality (recalled in  Lemma \ref{lemmakp}), 
we have
\beno
\|B(V,V)_{1}\|_{W^{s,p}}&\lesssim&\|(|\na|\lnr^{-1}h) \R c\|_{W^{s+1,p}}\\
&\lesssim&\||\na|\lnr^{-1}h\|_{W^{s+1,q_1}}\|\R c\|_{L^{r_1}}+\||\na|\lnr^{-1}h\|_{L^{r_2}}\|\R c\|_{W^{s+1,q_2}}\\
&\lesssim&\|h\|_{W^{s+1,q_1}}\|c\|_{L^{r_1}}+\|h\|_{L^{r_2}}\|c\|_{W^{s+1,q_2}}.
\eeno
The estimates of the other components follow from the same arguments, we omit the proof.
\end{proof}

We finally 
state the bilinear estimate
that will be heavily used in Section 4. We will give the proof in the appendix.
   \begin{lem} \label{lembilinear}
     Let us assume that  d=3, and let us define
     $$\phi_{j,k}(\xi,\eta)=(-1)^{j+1}b(\xi)+(-1)^{k+1}b(\eta)-b(\xi+\eta ), \quad j,k=1,2$$
     and
     \beno
    m(\xi,\eta)=\tchidix \tchidie \tchidiepx 
    \f{\lxper}{2ib(\xi+\eta)}.
     \eeno
     Then, we have the following estimates that are uniform for $\ep \in (0, 1]$:
     \beqs
     \|T_{\f{m}{\phi_{jk}}}(f,g)\|_{W^{\sigma,p}}\lesssim \|f\|_{W^{\sigma+2_{+},q_1}}\|g\|_{W^{2,r_1}}+\|f\|_{W^{2,r_2}}\|g\|_{W^{\sigma+2_{+},q_2}},
     \eeqs
 \beqs
 \|T_{\f{m}{\phi_{jk}^{2}}}(f,g)\|_{W^{\sigma,p}}\lesssim \|f\|_{W^{\sigma+2_{+},q_1}}\|g\|_{W^{3,r_1}}+\|f\|_{W^{3,r_2}}\|g\|_{W^{\sigma+2_{+},q_2}}.
\eeqs
     where    $\f{1}{p}=\f{1}{q_1}+\f{1}{r_1}=\f{1}{q_2}+\f{1}{r_2},$ $1<r_1,r_2\leq +\infty, 1\leq q_1,q_2<+\infty$
     and $T_{\f{m}{\phi_{jk}}}$ is the bilinear operator defined in \eqref{eqbilinear}.
     \end{lem}


\section{Proof of theorem \ref{thminviscid} }
This section is devoted  to the proof of Theorem \ref{thminviscid}. Let us observe that from a  standard iteration argument, (similar to the one for the  compressible Navier-Stokes system as in  \cite{MR564670}), one can show that the system \ref{NSPlow} admits a unique solution in  $C([0,T_{\ep}),H^3)$ for some $T_{\ep}>0$ and
that if the initial data are in $H^l$, $l \geq 3$, then this additional regularity also propagates on $[0, T_{\ep})$
. We thus only focus on the proof of  a priori estimates that are uniform in time and in $\ep$.

We shall use the norms :
\begin{align*}
\|X\|_{Y}&\define \|X\|_{W^{\si+3(1-\f{2}{p}),p'}}+\|X\|_{H^{N}},\\
\|U\|_{X_{T}}&\define \sup_{t\in [0,T)} \left( \ltr ^{\frac{3}{2}(1-\f{2}{p})}\|U (t)\|_{W^{\si,p}}+\ltr ^{\frac{3}{2}(1-\f{2}{p})}\|
\chi^{H}(D)U(t)\|_{H^{N-1}}
+\|U(t)\|_{H^N}\right).
\end{align*}
where $U=(\vr,\na \phi,u)$,
$\chi^{H}(D)=(1-\chi_{\ep,\kpz})(D)$,
 $6<p<\infty,\sigma\geq 3, N\geq \sigma+7$.

By  standard  bootstrap argument, it suffices to prove that there exists  $\tilde \delta_1>0$,  and $C>0$
that are independent of $T$  such that for every $\delta_{1} \in (0, \tilde \delta_{1}],$
    if $\|U\|_{X_T} \leq \delta_1$, then we have uniformly for $\ep \in (0, 1]$  an estimate under the form
\beq \label{eqbtstrap}
\|U\|_{X_T}\leq  C\left(\|U(0)\|_{Y}+\|U\|_{X_T}^{\f{3}{2}}+\|U\|_{X_T}^2+\|U\|_{X_T}^{3}\right).
\eeq
 Indeed, let us set
 \beqs
 T_{*}=\sup\{T \in [0, T_{\ep}), \, \|U\|_{X_T} \leq \delta_{1}\}.
 \eeqs
 Then, we can deduce from \eqref{eqbtstrap} that $T_{*}= T_{\ep}= +\infty$ by choosing
  $\delta$ (which is such that $\|U(0)\|_{Y} \leq \delta$)  and $\delta_{1}$ small enough such that by \eqref{eqbtstrap},
  $\|U\|_{X_T} \leq C (\delta + 3 \delta_{1}^{\frac{3} {2}})< \delta_{1}$ for any $T<T_{*}$.
   The result follows by time continuity and a local well-posedness result.

The a priori estimate \eqref{eqbtstrap} will follow  from the following two propositions.
\begin{prop}[Energy estimates]
\label{propenergy}

 We define the energy functional $$E_N=\sum_{|\alpha|\leq N}E_{\alpha}=\sum_{|\alpha|\leq N}\int \frac{|\pa \vr|^2}{2}+\frac{|\pa \nabla \phi|^2}{2}+\rho \frac{|\pa u|^2}{2}\d x. $$
  Assuming that  $\|\vr\|_{H^{2}}\leq\delta_{1}$ and that  $\delta_1$ is small enough  so that $\|\vr\|_{L^{\infty}}\leq \f{1}{6}, \|\na \vr\|_{L^3}<\f{1}{2\tilde{c}}$
  where $\tilde{c}$ is the biggest one among the Sobolev constants coming from the embedding $H^2\hookrightarrow L^{\infty}, \dot{H}^{\f{1}{2}}\hookrightarrow L^3, \dot{H}^{1}\hookrightarrow L^6$,
   then there exists a constant $C>0$ which depends only on $\tilde{c}$, such that
  \beq \label{eq5}
  \sup_{0\leq t<T}E_N(t)\leq E_N(0)+ C \|U\|_{X_T}^3.
  \eeq
   \end{prop}
 \begin{proof}
By taking the time derivative of  the energy functional and by using the equations, we get:
 \beno
\frac{\d}{\d t}E_{\alpha}
&=&-\int \div(\rho u)\frac{|\pa u|^2}{2}+\rho \pa u\cdot\pa \big[(u\cdot \na u)+\na \vr-\na \phi-2\ep \Delta u\big]\\
&&\qquad\qquad+\pa\vr\pa\div(\rho u)-\pa\na\phi\cdot\pa\na\pt \phi \d x\\
&=&\int\rho \pa u\cdot \big[u,\pa\big]\na u+\pa \rho \div\big([\rho,\pa ]u\big)\\
&&+\big[\pa\na\phi \cdot\pa\pt\na \phi+\rho\pa u\cdot\pa\na \phi\big]+2
\ep \rho \pa u\cdot\pa \Delta u \d x\\
&\triangleq&J_1+J_2+J_3+J_4.
\eeno
 We now estimate these four terms.
For $J_1$,  by using Lemma \ref{lemmakp} in the appendix, we have,
 \beno
 J_1=-\int\rho \pa u\cdot \big[\pa,u\big]\na u\d x
 &\leq&\|\rho\|_{L^{\infty}}\|\pa u\|_{L^2}\|\big[\pa,u\big]\na u\|_{L^2}\leq  2\|u\|_{\dot{H}^{|\alpha|}}^2\|\na u\|_{L^{\infty}}.
 \eeno
 For $J_2$, which is non-zero only if  $|\alpha|\geq 1$, by
 using Lemma \ref{lemmakp} again, we have
 \beno
 J_2&=&\int \pa \vr \big(\div(\rho \pa u)-\pa\div(\rho u)\big)\nonumber\\
 &=&\int |\pa \vr|^2 \div u-\pa\vr([\pa,u]\na\vr+[\pa,\rho]\div u-\na\vr\pa u) \d x\nonumber\\
 &\lesssim& \| ( u,  \vr)\|_{\dot H^{|\alpha|}}^2( \| (\na u,\na \vr)\|_{L^{\infty }}.
 \eeno
In a similar way, we estimate $J_3$ and $J_4$ as follows:
 \beno
 J_3&=&\int\rho \pa u\pa\na\phi+\pa\phi\pa\div(\rho u)\d x
 =-\int \pa\na\phi[\pa,\rho]u \d x \\
 &\lesssim&\|\na\phi\|_{\dot H^{|\alpha|}}(\|u\|_{\dot H^{|\alpha|-1}}\|\na\vr\|_{L^{\infty}}+\|u\|_{L^{\infty }}\| \vr\|_{\dot H^{|\alpha|}}),
 \eeno
 \beno
 J_4&=& 2\varepsilon \int \rho \pa u\pa \Delta u \d x =-2\varepsilon \int \rho|\pa\na u|^2 +\na \vr \pa u\cdot \pa \na u  \d x.
 \eeno
 We estimate the second term in the above equality by
 $ \tilde{c}\ep\|\na \vr\|_{L^3}\|\na u\|_{\dot{H}^{|\alpha|}}^{2}$
 where $\tilde{c}$ is the Sobolev constant associated to Sobolev embedding $\dot{H}^{1}\hookrightarrow L^6$.
We finally get:
\beq\label{energyineq}
\frac{\d}{\d t}E_N +\varepsilon\sum_{|\alpha|\leq N} \int \rho|\pa\na u|^2 \d x \lesssim (\|u\|_{W^{1,\infty}}+\| \vr\|_{W^{1,\infty}}) \|U\|_{H^N}^2.
\eeq
By integrating in time and by  using the definition of $\|U\|_{X_T}$, 
we get the  inequality \eqref{eq5}.

\end{proof}

  \begin{rmk}
Note that  by  the assumptions of Proposition \ref{propenergy},
  we have $\f{5}{6}\leq \rho\leq \f{7}{6}$  so that $ E_N\approx \|U\|_{H^N}^2 $
which combine with \eqref{eq5} gives:
$
\sup_{0\leq t<T}\|U(t)\|_{H^N}\lesssim \|U(0)\|_{H^N}+\|U\|_{X_T}^{\f{3}{2}}.
$
\end{rmk}

Now we begin to deal with the other two terms in the definition of the  $X_T$ norm.
By the definition of $V$ and the boundedness of the  Riesz transform  in $L^q$,     $1<q<+\infty$, we have $\|U\|_{X_{T}}\sim \|V\|_{X_{T} }$ 
 which leads us to prove the corresponding estimate for $V$ :
 \begin{prop}\label{propdecay}
For any $6<p<\infty$, we have the decay estimate:
 \begin{equation*}
\sup_{t\in [0,T)} \left( \ltr ^{\frac{3}{2}(1-\f{2}{p})}\|V\|_{W^{\si,p}}+\ltr ^{\frac{3}{2}(1-\f{2}{p})}\|\chi^{H}(D)V\|_{H^{N-1}} \right) \lesssim \|V(0)\|_{Y}+\|V(0)\|_{Y}^{2}+\|V\|_{X_T}^2+\|V\|_{X_T}^3.
 \end{equation*}
 \end{prop}
\begin{proof}
 For notational convenience, we denote $V^{L}=\chi^{L}(D)V=\chi_{\ep,\kpz}(D)V, V^{H}=\chi^{H}(D)V=(1-\chi_{\ep,\kpz})(D)V$.  By using \eqref{DuhamelV}, Lemma \ref{lemmahf}, Lemma \ref{lemnonlinear}  and Sobolev embedding, we get:
 \beno
 \|V^{H}(t)\|_{H^{N-1}}
 &\lesssim&\|e^{-tA}\chi^{H}(D) V(0)\|_{H^{{N-1}}}+\int_{0}^{t}\|e^{-(t-s)A}\chi^{H}(D) B(V,V)(s)\|_{H^{{N-1}}}\d s\\
 &\lesssim&e^{-c_0t}\|V(0)\|_{H^{{N-1}}}+\int_{0}^{t}e^{-c_0(t-s)}\|B(V,V)(s)\|_{H^{{N-1}}}\d s\\
 &\lesssim&e^{-c_0t}\|V(0)\|_{H^{{N-1}}}+\int_{0}^{t}e^{-c_0(t-s)}\|V(s)\|_{H^{N}}\|V(s)\|_{W^{\si,p}}\d s\\
 &\lesssim&e^{-c_0t}\|V(0)\|_{H^{{N-1}}}+(1+t)^{-\indxf}\|V\|_{X_T}^2.
 \eeno
 For the $W^{\si,p}$ estimate, we just use Sobolev embedding and the above estimate,
$$\| V^{H}(t)\|_{W^{\si,p }}
\lesssim \| V^{H}(t)\|_{H^{\sigma+3(\f{p-2}{2p}) }}
\lesssim e^{-c_0t}\|V(0)\|_{H^{N-1}}+(1+t)^{-\indxf}\|V\|_{X_T}^2.$$
We shall now prove the decay estimate for low frequencies.
 By applying  $Q^{-1}\chi^{L} $  to  the system for $V$ (see \eqref{eqV}) and
 by setting $R=(r_1,r_2)^{\top}=Q^{-1}\chi^{L}(D) V,$ we find that  $R$ solves  the system:
 \beq \label{systemforR}
 \pt R+\left(
  \begin{array}{cc}
    -\lm(D)&0\\
    0& -\lp(D)
  \end{array}
\right)R
=Q^{-1}\chi^{L}(D) B(V,V)
\eeq
with initial data $R(0)=Q^{-1}\chi^{L}(D) V(0).$ We thus obtain from the Duhamel formula that
 \ben \label{eq for R}
 R&=&\left(
  \begin{array}{cc}
    e^{\lm(D) t}&0\\
    0& e^{\lp(D) t}\\
  \end{array}
\right)R_0
+\int_{0}^{t}\left(
  \begin{array}{cc}
    e^{\lm(D) (t-s)}&0\\
    0& e^{\lp(D) (t-s)}\\
  \end{array}
\right)Q^{-1}\chi^{L}(D) B(V,V) \d s \nonumber\\
&\define&J_1+J_2+J_3+J_4
 \een
 where:
 \beno
 J_1&=&\left(
  \begin{array}{cc}
    e^{\lm(D)t}&0\\
    0& e^{\lp(D) t}\\
  \end{array}
\right)R_0 ,\\
J_2&=&\int_{0}^{t}\left(
  \begin{array}{cc}
    e^{\lm(D) (t-s)}&0\\
    0& e^{\lp(D) (t-s)}\\
  \end{array}
\right)Q^{-1}\chi^{L}(D) B\big( V^{H},V\big) 
\d s,\\
J_3&=&\int_{0}^{t}\left(
  \begin{array}{cc}
    e^{\lm(D)(t-s)}&0\\
    0& e^{\lp{(D)} (t-s)}\\
  \end{array}
\right)Q^{-1}\chi^{L}(D) B\big( V^{L},V^{H}\big) \d s,\\
J_4&=&\int_{0}^{t}\left(
  \begin{array}{cc}
    e^{\lm(D) (t-s)}&0\\
    0& e^{\lp(D) (t-s)}\\
  \end{array}
\right)Q^{-1}\chi^{L}(D) B( V^{L}, V^L) \d s.
 \eeno
 For the term $J_1$, note that $R_0=Q^{-1}\chi^{L}(D) V(0)=\tchidid Q^{-1}\chi^{L}(D) V(0)$, thus by Corollary \ref{corlowf}, \ref{corQ} we have:
 \beno
 \|J_1\|_{W^{\si,p}}
\lesssim (1+t)^{-\indxf}\|Q^{-1}\chi^{L}(D) V(0)\|_{W^{\indxs,p'}}
\lesssim(1+t)^{-\indxf}\|V(0)\|_{W^{\indxs,p'}}.
 \eeno
 For the term $J_2$, we use Corollaries \ref{corlowf}, \ref{corQ} and  Lemma \ref{lemnonlinear} to get:
 \beno
  \|J_2\|_{W^{\si,p}}
&\lesssim&\int_{0}^{t} (1+t-s)^{-\indxf}\|Q^{-1}\chi^{L}(D) B\big(V^{H},V\big)\|_{W^{\indxs,p'}}\d s\\
&\lesssim & \int_{0}^{t} (1+t-s)^{-\indxf}(\|V^H\|_{H^{\indxs+1}}\|V\|_{L^{\f{2p}{p-2}}}+\|V^H\|_{L^2}\|V\|_{W^{\indxs+1,\f{2p}{p-2}}})\d s\\
&\lesssim&\int_{0}^{t} (1+t-s)^{-\indxf}(1+s)^{-\indxf}\|V\|_{X_T}^2 \d s
\lesssim(1+t)^{-\indxf}\|V\|_{X_T}^{2}.
 \eeno
 The estimate for $J_3$ is similar to the one for  $J_2$, we thus skip it.

 It remains to estimate  $J_4$ which is the most difficult one. To get over the difficulty of the  quadratic nonlinearity, we need to use  the normal form method.
By the definition of $Q, Q^{-1}$ in \eqref{Qdef} and $R=Q^{-1}\chi^{L}(D)V$, we have that
 \beno
 B( V^L,V^L)=B(QR,QR)\nonumber
 &=&
 \left(
  \begin{array}{c}
    \lnr \R^{*}\{\f{|\na|}{\lnr}(r_1+r_2) \R[i\f{b(D)}{\lnr}(r_2-r_1)+\f{\ep\Delta}{\lnr}(r_1+r_2)]\}\\
    |\na| \big| \R[i\f{b(D)}{\lnr}(r_2-r_1)+\f{\ep\Delta}{{\lnr}}(r_1+r_2)]\big|^2\\
  \end{array}
\right)  \nonumber\\
&\define& \left(
 \begin{array}{c}
 B_1(R,R)\\
 B_2(R,R)
 \end{array}
 \right).
 \eeno
 Define:
 \beno 
 A(R,R)&=& Q^{-1}B(QR,QR)
 =
 \left(
 \begin{array}{c}
 \f{1}{2ib}(\lp B_1+\lnr B_2)\nonumber\\
 \f{-1}{2ib}(\lm B_1-\lnr B_2)
 \end{array}
 \right)
 \define \left(
 \begin{array}{c}
 A_1(R,R)\\
 A_2(R,R)\\
 \end{array}
 \right).
\eeno
 We shall only study  the first component,
 ie.
 $$J_{41}=\int_{0}^{t}e^{(t-s)\lm(D)}\chidid A_1(R,R)\d s, $$
 the other can be handled in a  similar way.
 For  notational convenience (although with a little abuse of notation), we write $A_1(R,R)=\f{\lnr}{2ib}\sum n_1(D)(n_2(D)r_1\cdot n_2(D)r_2)$,  here the summation runs over all the possibilities in the definition of $n_1(D),n_2(D)$ defined in Lemma $\ref{lemmultiplier}$ from the definition of $\lambda_{\pm}$.

 Set $\tilde{R}= n_2(D)R$, 
 then by recalling $\tilde{\chi}_{\ep,\kpz}\chi_{\ep,\kpz}=\chi_{\ep,\kpz}$,
 $J_{41}$ is the sum of the following term:
 \beno
 G_{jk}=e^{-ib(D)t}\cF^{-1}\bigg(\int_{0}^{t}\int_{\mathbb{R}^3} e^{-\ep|\xi|^2 (t-s)}e^{ib(\xi)s}m(\xi-\eta,\eta)n_1(\xi)\chidix \hat{\tilde{r}}_{j}(s,\xi-\eta)\hat{\tilde{r}}_{k}(s,\eta)\d \eta \d s\bigg).
 \eeno
 where $m(\xi-\eta,\eta)=\tchidiemx \tchidie \tchidix \f{\lxr}{2ib(\xi)}
 $.

Set  $W=(W_1,W_2)^{\top}=\left(
  \begin{array}{cc}
    e^{ib(D)t}&0\\
    0& e^{-ib(D) t}\\
  \end{array}
\right)R$,
then from \eqref{systemforR}, $W$ satisfies:
\beno
\pt W=\left(
  \begin{array}{cc}
    e^{ib(D)t}&0\\
    0& e^{-ib(D) t}\\
  \end{array}
\right)\big[\ep \Delta R+Q^{-1}\chi^{L}(D) B(V,V)\big].
\eeno
 By the definition of $W$, we have $w_j=e^{i(-1)^{j+1}b(D)}r_j$, so by defining $\tilde{w}_j=e^{i(-1)^{j+1}b(D)}\tilde{r}_j$ we get:
 \beqs
 G_{jk}=e^{-ib(D)t}\cF^{-1}\bigg(\int_{0}^{t}\int_{\mR^3} e^{-\ep|\xi|^2 (t-s)}e^{-is\phi_{jk}}m(\xi-\eta,\eta)n_1(\xi)\chidix \widehat{\tilde{w}}_{j}(s,\xi-\eta)\widehat{\tilde{w}}_{k}(s,\eta)\d \eta \d s\bigg)
 \eeqs
 thus, by using  that $\phi_{jk}$ does  not vanish in the support of $\chidid$, we can integrate by parts in time:
 \begin{align}\label{def of Ij}
 e^{ib(D)t} G_{jk} & =
 \cF^{-1}\int_{0}^{t}\int_{\mR^3} e^{-\ep|\xi|^2 (t-s)}\f{\p_{s}e^{-is\phi_{jk}}}{-i\phi_{jk}}m(\xi-\eta,\eta)n_1(\xi)\chidix \widehat{\tilde{w}}_{j}(s,\xi-\eta)\widehat{\tilde{w}}_{k}(s,\eta)\d \eta \d s \nonumber\\
 &= e^{ib(D) t}\big(\sum_{j=1}^{7}I_{j}\big) 
 \end{align}
 where
 \begin{align*}
  I_1 & =i\chidid n_1(D)T_{\f{m}{\phi_{jk}}}(\tilde{r}_{j}(t),\tilde{r}_{k}(t)),\\
 I_2&=-ie^{\ep t\Delta}e^{-itb(D)} \chidid n_1(D)T_{\f{m}{\phi_{jk}}}(\tilde{r}_{j}(0),\tilde{r}_{k}(0)),\\
 I_3&=i\int_{0}^{t}e^{\ep (t-s)\Delta}e^{i(t-s)b(D)} (\ep \Delta)\chidid n_1(D)T_{\f{m}{\phi_{jk}}}(\tilde{r}_{j}(s),\tilde{r}_{k}(s))\d s,\\
 I_4&=-i \int_{0}^{t}e^{\ep (t-s)\Delta}e^{i(t-s)b(D)} \chidid n_1(D) T_{\f{m}{\phi_{jk}}}(\ep \Delta \tilde{r}_j(s),\tilde{r}_{k}(s))\d s,\\
 I_5&=\int_{0}^{t}e^{\ep (t-s)\Delta}e^{i(t-s)b(D)} \chidid n_1(D) T_{\f{m}{\phi_{jk}}}(\tilde{B}_{j}(s),\tilde{r}_{k}(s))\d s,
 \end{align*}
 \begin{align*}
 I_6&=\int_{0}^{t}e^{\ep (t-s)\Delta}e^{i(t-s)b(D)} \chidid n_1(D) T_{\f{m}{\phi_{jk}}}(\tilde{r}_j(s),\ep \Delta \tilde{r}_{k}(s))\d s,\\
 I_7&=\int_{0}^{t}e^{\ep (t-s)\Delta}e^{i(t-s)b(D)} \chidid n_1(D) T_{\f{m}{\phi_{jk}}}(\tilde{r}_j(s),\tilde{B}_{k}(s))\d s
 \end{align*}
 and $\tilde{B}=n_2(D)Q^{-1}\chi^{L}B(V,V)$, we recall that $B(V,V),Q^{-1}$ are  defined in \eqref{eqV} and \eqref{Qdef}.
  We now estimate $I_1$ to $I_7$.
  In the following, we
  shall  use the   estimates  for the  bilinear operator $T_{\f{m}{\phi_{jk}}}$ in Lemma \ref{lembilinear} with
  the choice $(k)_{+}=\f{3}{p}$.

  By Lemma \ref{lembilinear} and  Sobolev embedding,
  we can estimate $I_1$ as follows:
 \beno
 \|I_1\|_{W^{\si,p}}
 \lesssim \|T_{\f{m}{\phi_{jk}}}(\tilde{r}_{j}(t),\tilde{r}_{k}(t))\|_{W^{\si+1,p}}
 &\lesssim&\|\tilde{r}_j(t)\|_{W^{\si+3_{+},p}}\|\tilde{r}_k(t)\|_{W^{2,\infty}}+\|\tilde{r}_j(t)\|_{W^{2,\infty}}
 \|\tilde{r}_k(t)\|_{W^{\si+3_{+},p}}\\
 &\lesssim&\|R(t)\|_{H^{\si+\f{3}{2}(1-\f{2}{p})+3_{+}}}\|R(t)\|_{W^{3,p}}
 \lesssim (1+t)^{-\indxf}\|V\|_{X_T}^2.
 \eeno
 By Corollaries \ref{corlowf}, \ref{corQ}, Lemma \ref{lembilinear} and the Sobolev embedding, we have for $I_2$:
 \beno
 \|I_2\|_{W^{\si,p}}
 &\lesssim&(1+t)^{-\indxf}\|T_{\f{m}{\phi_{jk}}}(\tilde{r}_{j}(0),\tilde{r}_{k}(0))\|_{W^{\indxs+1,p'}}\\
 &\lesssim&(1+t)^{-\indxf}(\|r_j(0)\|_{H^{\indxs+3_{+}}}\|r_k(0)\|_{W^{2,\f{2p}{p-2}}}+\|r_j(0)\|_{W^{2,\f{2p}{p-2}}}\|r_k(0)\|_{H^{\indxs+3_{+}}} )\\
& \lesssim& (1+t)^{-\indxf}\|R(0)\|_{H^{N}}^{2}
\lesssim (1+t)^{-\indxf}\|V(0)\|_{H^{N}}^{2}.
 \eeno
 For the term $I_5$, we have:
 \beno
 \|I_5\|_{W^{\si,p}}&\lesssim& \int_{0}^{t}(1+t-s)^{-\indxf}\|T_{\f{m}{\phi_{jk}}}(\tilde{B}_{j},\tilde{r}_k)\|_{W^{\indxs+1,p'}}\d s\\
 &\lesssim&\int_{0}^{t}(1+t-s)^{-\indxf}
 (\|B_{j}\|_{H^{\indxs+3_{+}}}\|r_k\|_{W^{2,\f{2p}{p-2}}}+\|B_{j}\|_{W^{2,\f{2p}{p-2}}}\|r_k\|_{H^{\indxs+3_{+}}}) \d s\\
 &\lesssim& \int_{0}^{t}(1+t-s)^{-\indxf}(\|B_{j}\|_{H^{\indxs+3_{+}}}\|R\|_{H^3}+\|{B}_{j}\|_{H^3}\|R\|_{H^{\indxs+3_{+}}})\d s\\
 &\lesssim&\int_{0}^{t}(1+t-s)^{-\indxf}(1+s)^{-\indxf}\|V\|_{X_T}^{3}\d s
 \lesssim(1+t)^{-\indxf}\|V\|_{X_T}^3.
 \eeno
 Here, we have used  Corollaries \ref{corlowf}, \ref{corQ}, Lemma
 \ref{lemnonlinear} \ref{lembilinear},  and Sobolev embedding. 
The estimate for $I_7$ is very similar to that of $I_5$, we omit the details.

The terms  $I_3, I_4$ correspond to the error terms created by $\ep\Delta$. As explained in the introduction, since one can only expect that  $\|\ep \Delta u\|_{H^{N-1}}\lesssim (1+t)^{-1}$ which is not a fast enough decay,  to  control them, we need to perform normal form transformation again.

 By integrating by parts again, we get
 \begin{align}
\nonumber I_4&=i\int_{0}^{t}e^{\ep(t-s)\Delta}e^{-i(t-s)b(D)}n_1(D)T_{\f{m}{\phi_{jk}}}(\ep\Delta \tilde{r}_j,\tilde{r}_k)\d s \\
\nonumber &=-\chidid n_1(D)T_{\f{m}{\phi_{jk}^2}}(\ep\Delta \tilde{r}_j(t),\tilde{r}_k(t))+e^{\ep t\Delta}e^{-itb(D)}\chidid n_1(D)T_{\f{m}{\phi_{jk}^2}}(\ep\Delta \tilde{r}_j(0),\tilde{r}_k(0))\\
\nonumber &\quad-\int_{0}^{t} e^{\ep(t-s)\Delta}e^{-i(t-s)b(D)}(\ep \Delta)\chidid n_1(D)T_{\f{m}{\phi_{jk}^2}}(\ep\Delta \tilde{r}_j,\tilde{r}_k)\d s \\
\nonumber& +\int_{0}^{t}e^{\ep(t-s)\Delta}e^{-i(t-s)b(D)}\chidid n_1(D) \big[T_{\f{m}{\phi_{jk}^2}}((\ep\Delta)^2 \tilde{r}_j+\ep \Delta \tilde{B}_{j},\tilde{r}_k)
  +T_{\f{m}{\phi_{jk}^2}}(\ep \Delta\tilde{r}_j,\ep\Delta \tilde{r}_k+\tilde{B}_{k})\big]\d s\\
\label{I4def} &\define I_{41}+\cdots +I_{47}.
 \end{align}
 The terms $I_{41}, I_{42}$ are similar to $I_1,I_2$.
 For instance, by  using the fact that $\ep\Delta\chidid$ is a bounded multiplier in $L^{p}$,  $1<p<\infty$, we get
 \beno
 \|I_{41}\|_{W^{\si,p}}&\lesssim&
 \|T_{\f{m}{\phi_{jk}^2}}(\ep\Delta \tilde{r}_j(t),\tilde{r}_k(t))\|_{W^{\si+1,p}}\\
 &\lesssim&\|\ep \Delta r_j(t)\|_{W^{\si+3_{+},\infty}}\|r_k(t)\|_{W^{3,p}}+
 \|\ep \Delta r_j(t)\|_{W^{3,p}}\|r_k(t)\|_{W^{\si+3_{+},\infty}}\\
 &\lesssim&\|R\|_{W^{\si+3_{+},\infty}}\|R\|_{W^{\si,p}}\lesssim\|R\|_{H^{\si+5}}\|R\|_{W^{\si,p}}
 \lesssim(1+t)^{-\indxf}\|V\|_{X_T}^2.
 \eeno
 Up to now, we have only used the dispersive estimates, but not yet the viscous dissipation, we shall  use it in the estimate for
 $I_{43}$.
 We write $$I_{43}=\int_{0}^{t-1}+\int_{t-1}^{t}e^{\ep(t-s)\Delta}e^{-i(t-s)b(D)}(\ep \Delta)\chidid n_1(D) T_{\f{m}{\phi_{jk}^2}}(\ep\Delta \tilde{r}_j,\tilde{r}_k)\d s\define I_{431}+I_{432}$$
 By Young's inequality, 
  Lemma \ref{lembilinear},
 Sobolev embedding, we get:
 \beno
 \|I_{431}\|_{W^{\si,p}}&\lesssim&\int_{0}^{t-1}\|\cF^{-1}(\ep |\xi|^2 e^{-\ep(t-s)|\xi|^2})\|_{L^{\f{2p}{p+2}}}\|e^{-i(t-s)b(D)} T_{\f{m}{\phi_{jk}^2}}(\ep\Delta \tilde{r}_j,\tilde{r}_k)\|_{H^{\si+1}} \d s\\
 &\lesssim&\int_{0}^{t-1} \ep^{\f{1}{4}+\f{3}{2p}}(t-s)^{-(\f{7}{4}-\f{3}{2p})}\| T_{\f{m}{\phi_{jk}^2}}(\Delta \tilde{r}_j,\tilde{r}_k)\|_{H^{\si+1}}\d s\\
 &\lesssim & \int_{0}^{t-1}(t-s)^{-(\f{7}{4}-\f{3}{2p})}\|R(s)\|_{H^{\si+\f{3}{p}+5_{+}}}\|R(s)\|_{W^{3,p}} \d s\\
 &\lesssim&\int_{0}^{t-1}(t-s)^{-(\f{7}{4}-\f{3}{2p})}(1+s)^{-\indxf}\|V\|_{X_T}^2\d s
 \lesssim(1+t)^{-\indxf}\|V\|_{X_T}^2.
 \eeno
  Here, we have used that:  $\ep^{\f{1}{4}+\f{3}{2p}}\leq \ep^{\f{1}{4}}\leq 1$,
  $\f{7}{4}-\f{3}{2p}>\indxf$
  and $\f{3}{p}<\f{1}{2}$ if  $6<p<+\infty$.
 By using again the fact that $\ep\Delta\chidid$ is a bounded multiplier in $L^{p}$ ($1<p<\infty$), Lemma \ref{lembilinear} and Sobolev embedding, we have:
 \beno
 \|I_{432}\|_{W^{\si,p}}&\lesssim&\int_{t-1}^{t}\| T_{\f{m}{\phi_{jk}^2}}(\ep \Delta \tilde{r}_j,\tilde{r}_k)\|_{H^{\si+1+3(\f{p-2}{2p})}}\d s\\
 &\lesssim&\int_{t-1}^{t}\|\ep \Delta r_j\|_{H^{\si+(\f{9}{2})_{+}}}\|r_k\|_{W^{3,p}}+\|\ep \Delta r_j\|_{W^{3,p}}\|r_k\|_{H^{\si+(\f{9}{2})_{+}}}\d s\\
 &\lesssim&\int_{t-1}^{t}\|R\|_{H^{\si+5}}\|R\|_{W^{\si,p}} \d s
 \lesssim (1+t)^{-\indxf}\|V\|_{X_T}^2.
 \eeno
 For $I_{44}$, we need to use the following lemma.
 \begin{lem}\label{lemepd}
 For $k\leq N-1,$ we have  the following uniform for $\ep \in (0, 1]$ estimates:
 \beno
 \|\ep\Delta R\|_{H^k}&\lesssim& (1+t)^{-1}(\|V(0)\|_{Y}+\|V\|_{X_T}^2),\\
 \|(\ep\Delta)^{2} R\|_{H^k}&\lesssim& (1+t)^{-\indxf}(\|V(0)\|_{Y}+\|V\|_{X_T}^2).
 \eeno
 \end{lem}
\begin{proof} By using the equation \eqref{eq for R} for $R$ ,
  we obtain from  Lemma \ref{lemed}, tame estimate and Sobolev embedding,
  \beno
  \|\ep\Delta R\|_{H^k}&\leq &\|e^{\ep t\Delta}\ep\Delta \chidid Q^{-1}V(0)\|_{H^{k}}+\int_{0}^{t}\|e^{\ep (t-s)\Delta}\ep\Delta \chidid Q^{-1}B(V,V)\|_{H^{k}}\d s\\
  &\lesssim&(1+t)^{-1}\|V(0)\|_{H^{k}}+\int_{0}^{t}(1+t-s)^{-1}\|\tchidid Q^{-1}B(V,V)\|_{H^{k}}\d s\\
  &\lesssim&(1+t)^{-1}\|V(0)\|_{H^{k}}+\int_{0}^{t}(1+t-s)^{-1}(1+s)^{-\indxf}\|V\|_{X_T}^2\d s\\
  &\lesssim&(1+t)^{-1}(\|V(0)\|_{H^{k}}+\|V\|_{X_T}^2).
  \eeno
 where we have used   that for $6<p<+\infty,$ $1<\indxf<\f{3}{2}.$
The other inequality follows from the same arguments by noticing that  $\f{3}{2}(1-\f{2}{p})<2$.
\end{proof}
  We go back to the estimate of $I_{44}$ in \eqref{I4def}.
  By using the last  lemma, we get:
   \begin{eqnarray*}
 &&\|I_{44}\|_{W^{\si,p}}\lesssim\int_{0}^{t}(1+t-s)^{-\indxf}\|T_{\f{m}{\phi_{jk}^2}}(\ep\Delta)^{2}\tilde{r}_j,\tilde{r}_k)\|_{W^{\indxs+1,p'}}\d s\\
 &\lesssim&\int_{0}^{t}(1+t-s)^{-\indxf}\big(\|(\ep\Delta )^{2}r_j\|_{H^{\indxs+3_{+}}}\|r_k\|_{W^{3,\f{2p}{p-2}}}+\|(\ep\Delta )^{2}r_j\|_{W^{3,\f{2p}{p-2}}}\|r_k\|_{H^{\indxs+3_{+}}}\big)\d s  \\
&\lesssim&\int_{0}^{t}(1+t-s)^{-\indxf}(1+s)^{-\indxf}\big(\|V\|_{X_T}^{3}+\|V\|_{X_T}^{2}+\|V(0)\|_{Y}^{2}\big) \d s\\
 &\lesssim&(1+t)^{-\indxf}\big(\|V\|_{X_T}^{3}+\|V\|_{X_T}^{2}+\|V (0)\|_{Y}^{2}\big).
 \end{eqnarray*}
 The term $I_{46}$ can be estimated in the same way as  $I_{44}$:
 \begin{eqnarray*}
&&\|I_{46}\|_{W^{\si,p}}
\lesssim\int_{0}^{t}(1+t-s)^{-\indxf}\|T_{\f{m}{\phi_{jk}^2}}(\ep\Delta \tilde{r}_j,\ep\Delta \tilde{r}_k)\|_{W^{\indxs+1,p'}}\d s \\
&\lesssim&\int_{0}^{t}(1+t-s)^{-\indxf}\big(\|\ep\Delta r_j\|_{H^{\indxs+3_{+}}}\|\ep\Delta r_k\|_{W^{3,\f{2p}{p-2}}}+\|\ep\Delta r_k\|_{H^{\indxs+3_{+}}}\|\ep\Delta r_j\|_{W^{3,\f{2p}{p-2}}}\big)\d s\\
&\lesssim&\int_{0}^{t}(1+t-s)^{-\indxf}(1+s)^{-2}\big(\|V\|_{X_T}^{2}+\|V(0)\|_{Y}\big)^{2}\d s\lesssim(1+t)^{-\indxf}\big(\|V\|_{X_T}^{4}+\|V(0)\|_{Y}^2\big).
\end{eqnarray*}
The terms  $I_{45},I_{47}$ are similar to $I_5,I_7,$ we thus skip them.

 It remains  to estimate
 $$I_3=i\int_{0}^{t}e^{\ep (t-s)\Delta}e^{i(t-s)b(D)} \ep \Delta\chidid T_{\f{m}{\phi_{jk}}}(r_{j}(s),r_{k}(s))\d s.$$
Integrating by parts in time again, we get
  \beno
 I_3&=& i\int_{0}^{t}e^{\ep (t-s)\Delta}e^{i(t-s)b(D)} (\ep \Delta) \chidid n_1(D)T_{\f{m}{\phi_{jk}}}(\tilde{r}_{j}(s),\tilde{r}_{k}(s))\d s\\
 &=&-\ep \Delta \chidid n_1(D)T_{\f{m}{\phi_{jk}^2}}( \tilde{r}_j(t),\tilde{r}_k(t))+e^{\ep t\Delta}e^{-itb(D)}\ep \Delta\chidid n_1(D)T_{\f{m}{\phi_{jk}^2}}( \tilde{r}_j(0),\tilde{r}_k(0))\\
 &&-\int_{0}^{t} e^{\ep(t-s)\Delta}e^{-i(t-s)b(D)}(\ep \Delta)^{2}\chidid n_1(D)T_{\f{m}{\phi_{jk}^2}}( \tilde{r}_j,\tilde{r}_k)(s)\d s \\
 &&+\int_{0}^{t}e^{\ep(t-s)\Delta}e^{-i(t-s)b(D)}\ep \Delta \chidid n_1(D) [T_{\f{m}{\phi_{jk}^2}}(\ep\Delta \tilde{r}_j+\tilde{B}_{j},r_k)+T_{\f{m}{\phi_{jk}^2}}(\tilde{r}_j,\ep\Delta \tilde{r}_k+ \tilde{B}_{k})]\d s\\
 &\define&I_{31}+\cdots +I_{37}.
 \eeno

Note that $I_{34}=I_{43},$ and  that the estimates for $I_{31},I_{32},I_{35},I_{37}$ are similar to the ones  for    $I_1,I_2,I_{45},I_{47},$ we thus skip them.

For $I_{33},$ we use Lemma $\ref{lembilinear}$ and Lemma $\ref{lemed}$ to get:
\beno
\|I_{33}\|_{W^{\si,p}}&\lesssim&
\int_{0}^{t}\|e^{\ep(t-s)\Delta} (\ep \Delta)^{2}\chidid n_1(D) T_{\f{m}{\phi_{jk}^2}}( \tilde{r}_j,\tilde{r}_k)(s)\|_{H^{\si+3(\f{1}{2}-\f{1}{p})}}\d s \\
&\lesssim&\int_{0}^{t}(1+t-s)^{-2}\|T_{\f{m}{\phi_{jk}^2}}( \tilde{r}_j,\tilde{r}_k)(s)\|_{H^{\si+3(\f{1}{2}-\f{1}{p})+1}}\d s\\
&\lesssim& \int_{0}^{t}(1+t-s)^{-2}(1+s)^{-\indxf}\|V\|_{X_T}^{2}\d s
\lesssim(1+t)^{-\indxf}\|V\|_{X_T}^{2}.
\eeno
 This ends the proof of Proposition \ref{propdecay}.
\end{proof}
 We thus have $\eqref{eqbtstrap}$ by combining Proposition \ref{propenergy} and \ref{propdecay}. Theorem \ref{thminviscid} then follows from the interpolation inequality:
for any $1<p'<2,$
 \beno
\|(\vr_0^{\varepsilon},\na \phi_0^{\varepsilon},\mathcal{P}^{\perp}u_0^{\varepsilon})\|_{W^{\si+3,p'}}\lesssim\|(\vr_0^{\varepsilon},\na \phi_0^{\varepsilon},\mathcal{P}^{\perp}u_0^{\varepsilon})\|_{W^{\si+3,1}}^{\theta}\|(\vr_0^{\varepsilon},\na \phi_0^{\varepsilon},\mathcal{P}^{\perp}u_0^{\varepsilon})\|_{H^{\si+3}}^{1-\theta}.
\eeno
\begin{rmk}
If we only prove the Theorem \ref{thminviscid} for $6<p\leq 12$, the decay estimate for $I_3,I_4$ will be easier, that is, we do not need to integrate by parts in time again.
Indeed, for example, when $p=12$, we could estimate $I_3$ as follows:
\beno
\|I_3\|_{W^{\si,12}}\lesssim\|I_3\|_{W^{\si+1,\f{12}{5}}}&\lesssim& \int_{0}^{t}(1+t-s)^{-\f{5}{4}}\|T_{\f{m}{\phi_{jk}}}(
\tilde{r},\tilde{r})\|_{W^{\si+\f{3}{2},\f{12}{7}}}\d s\\
&\lesssim&\int_{0}^{t}(1+t-s)^{-\f{5}{4}}\|\tilde{r}\|_{W^{2,12}}\|\tilde{r}\|_{H^{\si+(\f{7}{2})_{+}}}\d s\\
&\lesssim&\int_{0}^{t}(1+t-s)^{-\f{5}{4}}(1+s)^{-\f{5}{4}}\d s\|U\|_{X_T}^2\lesssim (1+t)^{-\f{5}{4}} \|U\|_{X_T}^3.
\eeno
For the estimate of $I_4$, we can use the  identity $$T_{\f{m}{\phi_{jk}}}( \ep\Delta\tilde{r}, \tilde{r})=\ep \Delta  T_{\f{m}{\phi_{jk}}}(\tilde{r}, \tilde{r})-2\sum_{l=1}^{3} T_{\f{m}{\phi_{jk}}}( \ep^{\f{1}{2}}\partial_{l}\tilde{r}, \ep^{\f{1}{2}}\partial_{l}\tilde{r}).$$
    and the  a priori estimates:
    \beno
   \|\ep^{\f{1}{2}}\na\tilde{r}\|_{H^{N-1}}\lesssim(1+s)^{-\f{1}{2}}\|U\|_{X}, \quad
    \|\ep^{\f{1}{2}}\na\tilde{r}\|_{W^{N-2,\f{12}{5}}}\lesssim(1+s)^{-\f{3}{4}}\|U\|_{X}.
    \eeno
Nevertheless, we are interested also in  $12<p<\infty$,  and in this case,  it is necessary
to use a normal form transformation again because $p'$ is too small to allows us to conclude the estimate directly.
\end{rmk}

\begin{rmk}\label{remdecay}
We now choose $24\leq p<+\infty.$
  By interpolation, for any $2\leq q\leq p$, we have the  decay estimate:
  \beqs
  \|(\varrho,u)\|_{W^{\sigma,q}}\lesssim(1+t)^{-\f{3}{2}(1-\f{2}{q})}\|(\varrho,u)\|_{X},
  \eeqs
  \beq\label{Linftydecay}
  \|(\varrho,u)(t)\|_{W^{\sigma,\infty}}\lesssim(1+t)^{-\f{4}{3}}\|(\varrho,u)\|_{X}.
  \eeq
   Indeed,  we only need to prove \eqref{Linftydecay} for $\na^{\sigma}(\varrho,u)$ as the other is almost obvious. By the Gagliardo-Nirenberg inequality, we have:
 \beqs
 \|\na^{\sigma}(\varrho,u)\|_{L^{\infty}}
 \lesssim\|(\varrho,u)\|_{\dot W^{\sigma,p}}^{\theta}\|(\varrho,u)\|_{\dot H^{\sigma+l}}^{1-\theta}\lesssim(1+t)^{-\f{3}{2}(1-\f{2}{p})\theta}\|(\varrho,u)\|_{X},
 \eeqs
where $\theta=1-\f{1}{(\f{l}{3}-\f{1}{2})p+1}$ and $l= 7$.
 When $p\geq24$, we have: $\f{3}{2}(1-\f{2}{p})\theta>\f{4}{3}$.

\end{rmk}

 \section {Proof of Theorem \ref{thmp}}

 Now our aim is to prove Theorem \ref{thmp}, that is to say, to get global existence for  system \eqref{NSPP} under the assumption that the incompressible part of the  initial velocity is small compared  to $\varepsilon$. We adapt the energy estimate showed in \cite{MR2917409} where the original (ENSP) system was treated. However, we need to focus more on the perturbation term where we should make the most use of the integrability of time decay of $(\varrho,u)$ in some Sobolev spaces. Global existence for $(n,u,\na\psi)$ follows from the energy estimate (see lemma \ref{lemenergy1} and lemma \ref{lemenergy2}) and classical bootstrap arguments. To prove the decay estimate, again, inspired by \cite{MR2917409}
 we use an interpolation argument between the energy estimate and an   $\dot{H}^{-s}$ estimate
  which is true if  the initial data is in this space. This yields
a  good energy inequality (see \eqref{eq22}), which finally gives  the decay estimate.

For the reader's convenience, we recall that we are talking about the system \eqref{NSPP} which takes the form:
 \beq \label{NSPP2}
 \left\{
\begin{array}{l}
\displaystyle \pt n +\div( \rho v+nu+nv)=0,\\
\displaystyle \pt v+u\cdot {\na v}+v\cdot (\na u+\na v)-\varepsilon\mathcal{L}v +\na n -\na \psi
=\varepsilon(\f{1}{\rho+n}-1) (\mathcal{L}v+\mathcal{L}u) ,  \\
\displaystyle \Delta \psi=n,\\
\displaystyle v|_{t=0} =\mathcal{P}u_0^{\varepsilon}, n|_{t=0}=0.
\end{array}
\right.
\eeq
    We define the  energy functional:
    \beq \label{energyfunele}
    \mathcal{E}_M(n,u,\na\psi)=\sum_{|\alpha|\leq M} \mathscr{E}_{\alpha}=\sum_{|\alpha|\leq M}\f{1}{2}\int\rho|\pa v|^2+|\pa n|^2+|\pa \na \psi|^2\d x.
    \eeq
     Denote also $\mathscr{E}_{k}=\sum_{|\alpha|=k} \mathscr{E}_{\alpha}$.
  We carry out energy estimates in the following two lemmas.
   \begin{lem}\label{lemenergy1}
   Assuming that  $(\rho=\vr+1,\na\phi,u)$ are given by Theorem
   \ref{thminviscid}, so that in particular
    $\|(\vr,\na\phi,u)\|_{H^3}\leq  C\delta,$  and  that $ \mathcal{E}_3\leq  C\delta\varepsilon$, with
    $C$ an absolute constant and  $\delta$ small enough, such that $\|\vr,n\|_{L^{\infty}}\leq C\|\vr,n\|_{H^2}\leq C\delta\leq \f{1}{6}$.  Then the following energy inequality holds: 
   for any $k=1,2\cdots M$  we have:
   \beq\label{eqenergy1}
   \begin{split}
  &\f{\d}{\d t}\mathscr{E}_{k}+\f{1}{2} \varepsilon\|\na v\|_{\dot{H}^{k}}^2
   \lesssim
   \|(u,\varrho)\|_{W^{M+2,\infty}}^{\f{3}{4}}\mathscr{E}_{k}+(\mathcal{E}_{3}^{\f{1}{2}}+\varepsilon \|\varrho\|_{H^M})\|(\na v,n)\|_{\dot{H}^{k}}^2
   \\
   &\qquad\qquad\qquad\qquad+{\varepsilon}^2\| u\|_{W^{k+2,\infty}}^{\f{5}{4}}(\|\varrho\|_{H^{|\alpha|}}^{2}+\mathcal{E}_{3})+\mathcal{E}_{3}\|\varrho\|_{W^{k+1,\infty}}^{\f{5}{4}}+\varepsilon \mathcal{E}_{3}^{\f{1}{2}}\|\varrho\|_{\dot{W}^{k,6}}^{2}.
   \end{split}
   \eeq
 where  
 $3\leq M\leq \si-2$.
    \end{lem}
 \begin{proof}
   By local existence, we have enough regularity to do energy estimates.  We take the time derivative of the  energy functional $\mathscr{E}_{\alpha},|\alpha|=k$, and  we make use of the equation \eqref{NSPP2}
 to  get:
\ben \label{energyeq1}
 \frac{\d}{\d t}\mathscr{E}_{\alpha}+\varepsilon\int \rho \big(|\pa \na v|^2+|\pa \div v|^2\big)\d x
 \triangleq\sum_{j=1}^{10}F_j
 \een
 where
 \begin{align*}
F_1&=\int \pa n\big(\div(\rho\pa v)-\pa\div(\rho v)\big)\d x,\qquad
F_2=\int\pa \na \psi \cdot\rho \pa v-\pa \na \psi \cdot \pa 
(\rho v)\d x,\\
F_3&=\int\rho \pa v \big(u\cdot \pa \na v-\pa(u\cdot \na v)\big)\d x, \qquad
F_4=-\int \rho \pa v\pa(v\cdot \na u)\d x,
\\
F_5&=-\int\pa n \pa\div(nu)\d x,
\qquad\qquad \qquad\qquad
F_6=-\int \pa \na\psi\cdot \pa(nu)\d x,
\\
 F_7&=-\varepsilon\int \na\vr\cdot\pa v\pa\div v+(\na\vr\otimes\pa v):\pa \na v\d x,\\
 F_8&=\varepsilon\int \rho\pa v\pa\big((\f{1}{\rho+n}-1)\mathcal{L}v\big)\d x,\qquad\qquad
 F_9=\varepsilon\int \rho\pa v\pa\big((\f{1}{\rho+n}-1)\mathcal{L}u\big)\d x, \nonumber\\
 F_{10}&=-\int \rho \pa v\pa(v\cdot \na v)+\pa n \pa\div(nv)+\pa \na\psi \pa
 (nv)\d x.\nonumber
\end{align*}
One first notice that $F_1,F_2,F_3$ equal to 0 when $|\alpha|=k=0.$
When $k\geq 1$,
using product estimate and Young's inequality, we have for $F_1$
\ben \label{eq10}
     F_1 &=& -\int\pa n \big([\pa,\rho]\div v+\pa(\na\vr\cdot v)-\na\vr\cdot\pa v\big) \d x \nonumber\\
     &\lesssim&\|n\|_{\dot{H}^{|\alpha|}}\|\na\vr\|_{W^{|\alpha|,\infty}}(\|v\|_{\dot{H}^{|\alpha|}}+\|v\|_{H^1})\nonumber\\
     &\lesssim&\|\na\vr\|_{W^{|\alpha|,\infty}}\| (n,v)\|_{\dot{H}^{|\alpha|}}^2
     +\|\na\varrho\|_{W^{|\alpha|,\infty}}^{\f{3}{4}}\|n\|_{\dot{H}^{|\alpha|}}^{2}+\|\na\varrho\|_{W^{|\alpha|,\infty}}^{\f{5}{4}}\|v\|_{H^{1}}^{2}\nonumber\\
     &\lesssim&(\|\na \varrho\|_{W^{|\alpha|,\infty}}+\|\na\varrho\|_{W^{|\alpha|,\infty}}^{\f{3}{4}})\| (n,v)\|_{\dot{H}^{|\alpha|}}^2+\mathcal{E}_3 \|\na\varrho\|_{W^{|\alpha|,\infty}}^{\f{5}{4}}\nonumber\\
      &\lesssim&\|\na\varrho\|_{W^{|\alpha|,\infty}}^{\f{3}{4}}\mathscr{E}_{k}+\mathcal{E}_{3}\|\na \varrho\|_{W^{k,\infty}}^{\f{5}{4}}.
     \een
 Here,  in the last inequality we have used that $\|\na \vr\|_{W^{M,\infty}}$ is small. We point out that we use the  power $\f{5}{4}$ in the above mainly to get more time integrability for the ‘perturbation term’, that is we want $b$ larger than $\f{3}{2}$ in \eqref{eq19} and \eqref{eq20} which will lead to the better decay estimate for $(n,\na\psi,v)$.

    The estimates for $F_2$ and $F_3$. For $F_2$, we write
    \begin{align}
F_2 &=\int \pa\na \psi\cdot(\vr\pa v-\pa(\vr v))\d x
    \leq\|\pa\na \psi\|_{L^2}\| \vr\|_{W^{|\alpha|,\infty}}(\|v\|_{L^2}+\|v\|_{\dot{H}^{|\alpha|}})\nonumber\\
      &\lesssim(\| \varrho\|_{W^{|\alpha|,\infty}}+\| \varrho\|_{W^{|\alpha|,\infty}}^{\f{3}{4}})(\|\na \psi\|_{\dot{H}^{|\alpha|}}^2+\|v\|_{\dot{H}^{|\alpha|}}^2)+\mathcal{E}_{3}\|\na\varrho\|_{W^{|\alpha|-1,\infty}}^{\f{5}{4}}\nonumber\\
      &\lesssim\|\varrho\|_{W^{k,\infty}}^{\f{3}{4}}\mathscr{E}_{k}+\mathcal{E}_{3}\| \varrho\|_{W^{k,\infty}}^{\f{5}{4}},
     \end{align}
   and we can get in the same way
     \begin{align}
     F_3
      &\lesssim\|\na u\|_{W^{k-1,\infty}}^{\f{3}{4}}\mathscr{E}_{k}+\mathcal{E}_{3}\|u\|_{W^{k,\infty}}^{\f{5}{4}}.
     \end{align}
     We now estimate $F_4-F_7$ with $|\alpha|=k\geq 0$. By product estimates and Young's inequality again, we have for $F_4$,
         \begin{multline}
         F_4
      \lesssim  \|\pa v\|_{L^2}\big(\|\na u\|_{W^{|\alpha|,\infty}}
     (\|v\|_{L^2}+\|v\|_{\dot H^{|\alpha|}})\big)
     \lesssim(\|\na u\|_{W^{|\alpha|,\infty}}+\|\na u\|_{W^{|\alpha|,\infty}}^{\f{3}{4}})\|v\|_{\dot{H}^{|\alpha|}}^2+\mathcal{E}_{3}\|\na u\|_{W^{|\alpha|,\infty}}^{\f{5}{4}}  \\
     \lesssim \|\na u\|_{W^{k,\infty}}^{\f{3}{4}}\mathscr{E}_{k}+\mathcal{E}_{3}\|\na u\|_{W^{k,\infty}}^{\f{5}{4}}.
     \end{multline}
    For $F_5$, we integrate by parts for the first term and use H\"{o}lder's inequality for the other two terms to get
     \begin{align}
     F_5
     & =-\int \pa n \big(\pa \na n\cdot u+[\pa,u ]\na n+\pa(n\div u)\big)\d x\nonumber \\
    &  \lesssim \|\pa n\|_{L^2}^2\|\na u\|_{L^{\infty}}+\|\pa n\|_{L^2}\| \na u\|_{W^{|\alpha|,\infty}}(\| n\|_{\dot{H}^{|\alpha|}}+\|n\|_{H^1}) \nonumber\\
    &  \lesssim(\|\na u\|_{W^{|\alpha|,\infty}}+\|\na u\|_{W^{|\alpha|,\infty}}^{\f{3}{4}})\|n\|_{\dot{H}^{|\alpha|}}^2+\mathcal{E}_{3}\|u\|_{W^{|\alpha|+1,\infty}}^{\f{5}{4}}
      \lesssim \|\na u\|_{W^{k,\infty}}^{\f{3}{4}}\mathscr{E}_{k}+\mathcal{E}_{3}\|\na u\|_{W^{k,\infty}}^{\f{5}{4}}.
     \end{align}
In a similar way, we have
\ben
 F_6
  &\lesssim&\|u\|_{W^{|\alpha|,\infty}}^{\f{3}{4}}\|(n,\na\psi)\|_{\dot{H}^{|\alpha|}}^2
  +\mathcal{E}_{3}\|u\|_{W^{|\alpha|,\infty}}^{\f{5}{4}}
   \lesssim \|u\|_{W^{k+1,\infty}}^{\f{3}{4}}\mathscr{E}_{k}+\mathcal{E}_{3}\|u\|_{W^{k+1,\infty}}^{\f{5}{4}}.
\een
as well as
\ben
F_7
\lesssim \varepsilon\|\na \varrho\|_{L^{\infty}}(\|\na v\|_{\dot H^{k}}^{2} +\|v\|_{\dot{H}^{k}}^2).
 \een
 For $F_8$, we only handle  $k=|\alpha|>0$ since the case $k=|\alpha|=0$ is easier. Integrating by parts, and denoting $\pa=\partial_j \partial^{\tilde{\alpha}}$ and using Lemma \ref{lemGN} we get that:
     \ben
     F_8&=& -\varepsilon  \int (\rho\partial_j\pa v+\partial_j\varrho \pa v)\pta ((\f{1}{\rho+n}-1)\mathcal{L} v)
     \nonumber\\
     &\lesssim&\varepsilon(\|\rho\|_{L^{\infty}}+\|\na \varrho\|_{L^3})\|\na v\|_{\dot H^{|\alpha|}}\|(\f{1}{\rho+n}-1)\mathcal{L} v\|_{\dot H^{|\alpha|-1}} \nonumber\\
 &\lesssim&\varepsilon
 \|\na v\|_{\dot H^{|\alpha|}}(\|\varrho+n\|_{L^{\infty}}\|\na ^2 v\|_{\dot{H}^{|\alpha|-1}}+\|\f{1}{\rho+n}-1\|_{\dot{W}^{|\alpha|-1,6}}\|\na^2 v\|_{L^3})\nonumber\\
 &\lesssim&\varepsilon \|\na v\|_{\dot H^{|\alpha|}}
 (\|\varrho+n\|_{L^{\infty}}\|\na ^2 v\|_{\dot{H}^{|\alpha|-1}}+\|\varrho +n\|_{\dot{W}^{|\alpha|-1,6}}\|\na ^2 v\|_{L^3})\nonumber\\
 &\lesssim& \varepsilon  (\mathcal{E}_{3}^{\f{1}{2}}+\|\varrho\|_{H^M})\|(n,\na v)\|_{\dot{H}^{k}}^2
 +\varepsilon \mathcal{E}_{3}^{\f{1}{2}}\|\varrho\|_{W^{k-1,6}}^{2},
 \een
     where we have used the fact that $\|\vr\|_{H^2}$ is bounded in the second inequality.\\
   We now deal with $F_{9}$ in the same fashion:
     \ben\label{eq11}
     F_9
     &\lesssim& \varepsilon \|\pa v\|_{L^2}\|\pa\big((\f{1}{\rho+n}-1)\mathcal{L}u\big)\|_{L^2}\nonumber\\
     &\lesssim&\varepsilon \|\pa v\|_{L^2}\big(\|\na^2 u\|_{W^{|\alpha|,\infty}}(\|\varrho+n\|_{L^2}+\|\varrho+n\|_{\dot{H}^{|\alpha|}})\big)\nonumber\\
     &\lesssim&\|\na ^2 u\|_{W^{|\alpha|,\infty}}^{\f{3}{4}}\|v\|_{\dot{H}^{|\alpha|}}^2+\varepsilon^2 \|\na^2 u\|_{W^{|\alpha|,\infty}}^{\f{5}{4}}(\| \varrho\|_{{H}^{|\alpha|}}^2+\mathcal{E}_{3})\nonumber\\
     &&+\varepsilon \|\na^2 u\|_{W^{|\alpha|,\infty}} (\|v \|_{\dot{H}^{|\alpha|}}^2+\|n\|_{\dot{H}^{|\alpha|}}^2)\nonumber\\
     &\lesssim&\|\na ^2 u\|_{W^{k,\infty}}^{\f{3}{4}}\mathscr{E}_{k}+\varepsilon^2 \|\na^2 u\|_{W^{k,\infty}}^\f{5}{4}(\| \varrho\|_{{H}^{k}}^2+\mathcal{E}_{3}).
     \een
 Finally, for $F_{10}$, inspired by \cite{MR2917409}, we actually have:
\ben\label{eq12}
F_{10}\lesssim \mathcal{E}_{3}^{\f{1}{2}}(\|\na v\|_{\dot{H}^{k}}^2+\|n\|_{\dot{H}^{k}}^2).
\een
 We just give details for  the third term of $F_{10}$, the first two terms  are similar and easier. Integrating by parts and using the Poisson equation, we have:
 \begin{align}
&  \nonumber \int\pa\na\psi\cdot\pa
 (nv)\d x
 =\int\pa\na\psi\cdot\pa v \Delta \psi)\d x
 =-\int \pa\na^{2}\psi:\pa(\na\psi\otimes v)+\pa\na\psi\cdot\pa\big((\na\psi\cdot \na)v\big)\d x \\
 & \label{F_{10.3}} \define \eqref{F_{10.3}}_{1}+\eqref{F_{10.3}}_{2}.
 \end{align}
 For the estimate of $\eqref{F_{10.3}}_{1}$, we use Kato-Ponce inequality (see Lemma \ref{lemmakp}) again to get:
 \beno
 |\eqref{F_{10.3}}_{1}|&\lesssim&\|\na^2\psi\|_{\dot{H}^{|\alpha|}}(\|\na\psi\|_{\dot{W}^{|\alpha|,6}}\|v\|_{L^3}+\|v\|_{\dot W^{|\alpha|,6}}\|\na\psi\|_{L^3})
 \lesssim \mathcal{E}_{3}^{\f{1}{2}}(\|n\|_{\dot{H}^{k}}^2+\|\na v\|_{\dot{H}^{k}}^2).
 \eeno
 For $\eqref{F_{10.3}}_{2}$, by using Kato-Ponce inequality and Gagliardo-Nirenberg  inequality, we have:
 \beno
 |\eqref{F_{10.3}}_{2}|&\lesssim& \|\na\psi\|_{\dot{W}^{|\alpha|,6}}\|(\na\psi\na v)\|_{\dot{W}^{|\alpha|,\f{6}{5}}}\\
 &\lesssim&\|n\|_{\dot{H}^{|\alpha|}}(\|\na v\|_{\dot{H}^{|\alpha|}}\|\na \psi\|_{L^3}+\|\na\psi\|_{\dot{W}^{|\alpha|,3}}\|\na v\|_{L^2})\\
 &\lesssim&\|n\|_{\dot{H}^{|\alpha|}}(\|\na v\|_{\dot{H}^{|\alpha|}}\|\na \psi\|_{L^3}+\|\na \psi\|_{\dot{H}^{|\alpha|+1}}^{\theta}\|\na \psi\|_{\dot{H}^\f{1}{2}}^{1-\theta}\|v\|_{\dot{H}^{\f{1}{2}}}^{\theta}\|v\|_{\dot{H}^{|\alpha|+1}}^{1-\theta})\\
 &\lesssim&\mathcal{E}_{3}^{\f{1}{2}}\|(n,\na\psi)\|_{\dot{H}^{k}}^2,
 \eeno
   where in the above  $\theta=\f{|\alpha|}{|\alpha|+\f{1}{2}}$.\\
    Using (\ref{eq10}-\ref{eq12})
    , and summing up for any $|\alpha|=k$  we get the Lemma \ref{lemenergy1}.
\end{proof}
    As indicated in \cite{MR2917409}, to close the energy estimate, we must get some damping for $n$, this can be achieved by doing the ‘cross' energy estimate.
     \begin{lem}\label{lemenergy2}
    Under the assumption of Lemma \ref{lemenergy1}, we have for any $k=0,1\cdots M-1$,
     \beq\label{eqenergy2}
     \begin{split}
     &\sum_{|\alpha|=k}\f{d}{dt}\int\pa\na n\cdot\pa v\d x+\f{1}{2}(\|n\|_{\dot{H}^{k}}^2+\|n\|_{\dot{H}^{k+1}}^2) \\
     &\lesssim \quad (\| v\|_{\dot{H}^{k+1}}^2+\| v\|_{\dot{H}^{k+2}}^2)+\|(\varrho,u)\|_{W^{k+2,\infty}}^{\f{3}{4}}(\mathscr{E}_{k}+\mathscr{E}_{k+1}) \\
     &
     \qquad \qquad\qquad+\varepsilon\mathcal{E}_{3}(\|(\varrho,u)\|_{W^{k+2,\infty}}^{\f{5}{4}}+\|(\varrho,u)\|_{\dot{W}^{k,6}}^2).
     \end{split}
     \eeq
     \end{lem}

     \begin{proof}
Taking $\pa\na$ (respectively $\pa)$ on the first (respectively second) equation in system \eqref{NSPP2}, multiplying by $\pa v$(respectively $\pa\na n$), integrating in space and adding together, we get
   \begin{multline}
     \label{energyeq2}
    \f{\d}{\d t}\int\pa\na n\cdot\pa v\d x+\int|\pa\na n|^2+|\pa n|^2\d x   = G_1+G_2+G_3+G_4+G_5
  \\  \triangleq -\int\pa\na\div(\rho v+nu)\cdot\pa v\d x-\int\pa\na n\cdot\pa(u\cdot\na v+v\cdot\na u)\d x
    -\varepsilon\int\pa \na n\cdot \pa(\f{1}{\rho+n}\mathcal{L}v)\d x \\-\varepsilon\int\pa \na n\cdot \pa\big((\f{1}{\rho+n}-1)\mathcal{L}u\big)\d x
     -\int\pa\na\div(nv)\cdot\pa v+\pa\na n\cdot\pa(v\cdot\na v)\d x.
         \end{multline}
    We handle the estimates for  $|\alpha|=k\geq 1,$ $k= 0$ being easier.

     Similar to the estimate in Lemma \ref{lemenergy1}, by H\"{o}lder's and Young's inequality, we have that
     \ben \label{eq13}
     G_1&=&\int \pa \div v\cdot\pa\div(v+\varrho v+nu)\d x \leq\|\pa\div v\|_{L^2}^2+\|\pa\div v\|_{L^2}\|\varrho v+nu\|_{\dot H^{|\alpha|+1}}\nonumber\\
     &\lesssim&\|\pa\div v\|_{L^2}^2+\|\pa\div v\|_{L^2} \|(\varrho,u)\|_{W^{|\alpha|+1,\infty}}(\|(n,v)\|_{L^2}+\|(n,v)\|_{\dot H^{|\alpha|+1}})
     \nonumber\\
     &\lesssim&\|\na v\|_{\dot{H}^{|\alpha|}}^2+(\|(\varrho,u)\|_{ W^{|\alpha|+1,\infty}}+\|(\varrho,u)\|_{ W^{|\alpha|+1,\infty}}^{\f{3}{4}})\| (n,v)\|_{\dot{H}^{|\alpha|+1}}^2
     +\mathcal{E}_3\|(\varrho,u)\|_{ W^{|\alpha|+1,\infty}}^{\f{5}{4}}\nonumber\\
     &\lesssim&\|\na v\|_{\dot{H}^{k}}^2+
     \|(\varrho,u)\|_{ W^{k+1,\infty}}^{\f{3}{4}}
     \mathscr{E}_{k+1}+\mathcal{E}_3\|(\varrho,u)\|_{ W^{k+1,\infty}}^{\f{5}{4}},
     \een
     as well as
     \ben
     G_2&=&-\int\pa\na n\cdot\pa(u\cdot\na v+v\cdot\na u)\d x \nonumber\\
     &\lesssim&\|\pa \na n\|_{L^2}(\|u\cdot \na v\|_{\dot H^{|\alpha|}}+\|v\cdot\na u\|_{\dot H^{|\alpha|}})\nonumber\\
     &\lesssim&\|\pa \na n\|_{L^2}\| u\|_{W^{|\alpha|+1,\infty}}(\|v\|_{H^1}+\|\na v\|_{\dot{H}^{|\alpha|}}+\|v\|_{\dot{H}^{|\alpha|}})\nonumber\\
     &\lesssim&(\|u\|_{{W^{k+1,\infty}}}+\|u\|_{{W^{k+1,\infty}}}^{\f{3}{4}})(\mathscr{E}_{k}+\mathscr{E}_{k+1})+\mathcal{E}_3\|u\|_{{W^{k+1,\infty}}}^{\f{5}{4}}.
     \een
     By Lemma \ref{lemmakp}, \ref{lemGN} in the  appendix, we estimate $G_3$ as follows:
     \ben\label{eq520}
     G_3&=&-\varepsilon\int\pa \na n\cdot \pa(\f{1}{\rho+n}\mathcal{L}v)\d x\nonumber\\
     &\lesssim&\varepsilon \|\pa\na n\|_{L^2}(\|\f{1}{\rho+n}-1\|_{\dot{W}^{|\alpha|,6}}\|\na^2 v\|_{L^3}+\|\f{1}{\rho+n}\|_{L^{\infty}}\|\na^2 v\|_{\dot{H}^{|\alpha|}})\nonumber\\
     &\lesssim&\varepsilon \|\pa\na n\|_{L^2}\big(\|\na^2 v\|_{L^3}\|(\varrho,n)\|_{\dot W^{|\alpha|,6}}+\|\na^2 v\|_{\dot{H}^{|\alpha|}}\big)\nonumber\\
     &\leq&(\f{1}{8}+c\mathcal{E}_{3}^{\f{1}{2}})\| n\|_{\dot{H}^{|\alpha|+1}}^2+c\|\na^2 v\|_{\dot H ^{|\alpha|}}^2+c \varepsilon^{2}\mathcal{E}_3\|\varrho\|_{W^{|\alpha|,6}}^2\nonumber\\
      &\leq&\f{1}{4}\|\pa\na n\|_{L^2}^2+c\| \na^2 v\|_{ \dot{H} ^{k}}^2+c \varepsilon^{2}\mathcal{E}_3\|\varrho\|_{W^{k,6}}^2
     \een
    using that  $c\mathcal{E}_3^{\f{1}{2}}\leq c\delta \varepsilon\leq \f{1}{8}$ where $\delta$ is small enough. Note that the first term in \eqref{eq520} could be absorbed by the left hand side of \eqref{energyeq2}.
   Next, $G_4$ can be estimated exactly as $F_9$.
     For $|\alpha|=k\geq 1,$ we have:
     \ben
     G_4&=&-\varepsilon\int\pa \na n\cdot \pa\big((\f{1}{\rho+n}-1\big)\mathcal{L}u)\d x\nonumber\\
     &\lesssim&\varepsilon \|\pa\na n\|_{L^2} \|\na^2 u\|_{W^{|\alpha|,\infty}}(\|\varrho+n\|_{L^2}+\|\varrho+n\|_{\dot{H}^{|\alpha|}})
     \nonumber\\
      &\lesssim& (\|\na^2 u\|_{W^{|\alpha|,\infty}}+\|\na^2 u\|_{W^{|\alpha|,\infty}}^{\f{3}{4}})\|( \na n, n)\|_{\dot H^{|\alpha|}}^2+
      \varepsilon\mathcal{E}_{3}\|\na^2 u\|_{W^{|\alpha|,\infty}}^{\f{5}{4}}
      +\varepsilon^2\|\varrho\|_{H^{|\alpha|}}^{2}\|\na^2 u\|_{W^{|\alpha|,\infty}}^2 \nonumber\\
      &\lesssim& 
      \|\na^2 u\|_{W^{k,\infty}}^{\f{3}{4}}
      (\mathscr{E}_{k}+\mathscr{E}_{k+1})+
      \varepsilon\mathcal{E}_{3}\|\na^2 u\|_{W^{k,\infty}}^{\f{5}{4}}+\varepsilon^2\|\varrho\|_{H^{k}}^{2}\|\na^2 u\|_{W^{k,\infty}}^2.
     \een
     For $G_5$, as in \cite{MR2917409}, one can show that if $\mathcal{E}_{3}^{\f{1}{2}}\leq \delta\varepsilon$ with $\delta$ small enough, we have:
     \beq \label{eq14}
     G_5\leq\f{1}{8}(\|n\|_{\dot{H}^{k}}^2+\|\na n\|_{\dot{H}^{k}}^2)+c(\|\na u\|_{\dot{H}^{k}}^2+\|\na^2 u\|_{\dot{H}^{k}}^2).
     \eeq
     Summing up from \eqref{eq13} to \eqref{eq14}, we get Lemma \ref{lemenergy2}.
    \end{proof}

    \begin{proof}[Proof of Theorem \ref{thmp}]

   We first prove global existence. 
   Summing up from $k=0$ to $k=M$, we can conclude from Lemma \ref{lemenergy1} and Lemma \ref{lemenergy2} , Remark \ref{remdecay} that if $\mathcal{E}_{3}\leq \delta^2\ep^2$ and $\delta$ is small enough, then there are some constants which depend only on $M$, such that
     \beq \label{eq18}
     \begin{split}
   &\f{\d}{\d t}\mathcal{E}_M+C \varepsilon\|\na v\|_{H^{M}}^2 \\
   \leq &C_1 (1+t)^{-a} \|(\varrho,u)\|_{X}^{\f{3}{4}}\mathcal{E}_{M}
  +C_2 \delta\varepsilon\|(\na v, n)\|_{H^M}^{2}
   +C_3{\varepsilon}^2(1+t)^{-b}(\delta \|(\varrho,u)\|_{X}^2+\|(\varrho,u)\|_{X}^3)\\
  \leq& C_1 \delta^{\f{3}{4}} (1+t)^{-a}\mathcal{E}_{M}
  +C_2 \delta\varepsilon\|(\na v, n)\|_{H^M}^{2}+C_3\delta^{3}{\varepsilon}^2(1+t)^{-b},
  \end{split}
   \eeq
   and
    \beq \label{eq19}
    \begin{split}
     &\sum_{|\alpha|\leq M-1}\f{d}{dt}\int\pa\na n\cdot\pa v\d x+\f{1}{2}\|n\|_{H^{M}}^2 \\
     \leq & C_4\|\na v\|_{H^{M}}^2+C_5 (1+t)^{-a} \|(\varrho,u)\|_{X}^{\f{3}{4}}\mathcal{E}_{M} +C_6 \mathcal{E}_{3}(1+t)^{-b}\|(\varrho,u)\|_{X}\\
      \leq&C_4\|\na v\|_{H^{M}}^2+C_5 \delta^{\f{3}{4}}\mathcal{E}_{M}+C_6\delta^{3}\varepsilon^2(1+t)^{-b},
     \end{split}
      \eeq
     where $a>1,b>\f{5}{3}$ (here we use $\|(\varrho,u)\|_{W^{\sigma,\infty}}\lesssim(1+t)^{-\f{4}{3}}$).

    Multiplying \eqref{eq19} by $8C_2\delta \varepsilon$ and add it to \eqref{eq18}, if $\mathcal{E}_{3}^{\f{1}{2}}\leq \delta \varepsilon$ and $\delta$ is small enough, (say, $32C_2C_4\delta \leq C$) we get that there exist constant $C_7,C_8,C_9,$ such that
    \beq \label{eq20}
    \begin{split}
    &\f{\d}{\d t}(\mathcal{E}_M+8C_2\delta\varepsilon \sum_{|\alpha|\leq M-1}\int\pa\na n\cdot\pa v\d x)+C_7 \varepsilon\|(n,
    \nabla v)\|_{H^M}^2
    \leq C_8 (1+t)^{-a}\delta^{\f{3}{4}}\mathcal{E}_{M}+C_9  \delta^3\varepsilon^{2}(1+t)^{-b}.
    \end{split}
    \eeq


       Define $\tilde{\mathcal{E}}_M=\mathcal{E}_M+8C_2\delta\varepsilon \sum_{|\alpha|\leq M-1}\int\pa\na n\cdot\pa v\d x$, we see that $\tilde{\mathcal{E}}_M\approx \mathcal{E}_M$(say $\f{1}{2}\mathcal{E}_M \leq\tilde{\mathcal{E}}_M\leq 2\mathcal{E}_M$) if $\delta$ is small enough.

       From inequality \eqref{eq20},    Grönwall's inequality and the fact $\|(\varrho,u)\|_{X}\leq \delta$, we achieve that:
       \beq \label{eq21}
       \begin{split}
      \mathcal{E}_M(t)+C_7\varepsilon \int_{0}^{t}\|(n, \na v)\|_{H^M}^2\d s
      & \leq e^{C_8\delta^{\f{3}{4}} \int_{0}^{t}(1+s)^{-a}\d x}(4\mathcal{E}_M(0)+2C_9 \delta^3 \varepsilon^2\int_{0}^{t}(1+s)^{-b}\d s)\\
      &\leq C_{10}\mathcal{E}_M(0)+C_{11}\delta^3 \varepsilon^2.
      \end{split}
      \eeq
         Global existence of $(n,\na \psi, v)$ in $C([0,+\infty),H^3)$ then is direct by bootstrap arguments. Moreover, we have $\mathcal{E}_3(t)\leq \delta^2 \varepsilon^2$ if $\mathcal{E}_3(0)\leq \f{1}{16}\delta^2 \varepsilon^2$ and $\delta$ is small enough (Note that $C_{10}\leq 8$ if $\delta$ is small enough.)
        Finally, as can be seen easily from \eqref{eq21}, if in addition, $\mathcal{E}_{M}(0)<+\infty$, then the solution constructed also belongs to $C([0,\infty),H^M)$.

        \begin{rmk}
          If we define $\mathcal{E}_{k}^{M}=\sum_{l=k}^{M}\mathscr{E}_{l}$,
          then by adding \eqref{eqenergy1} from $k$ to $M$, \eqref{eqenergy2} from $k$ to $M-1$ and the same arguments for proving \eqref{eq20}, we can have(with another constant $C_7$):
     \beq \label{eqhigh}
        \begin{split}
    &\f{\d}{\d t}(\mathcal{E}_{k}^{M}+8C_2\delta\varepsilon \sum_{k\leq|\alpha|\leq M-1}\int\pa\na n\cdot\pa  v\d x)+C_7 \varepsilon(\|\na^{k} n\|_{H^{M-k}}^2+\|\na^{k+1} v\|_{H^{M-k}}^2)\\
    &\leq C_8 (1+t)^{-a}\delta^{\f{3}{4}}\mathcal{E}_{k}^{M}+C_9  \varepsilon^{2}\delta^{3}(1+t)^{-b}.
    \end{split}
    \eeq
     \end{rmk}

          Motivated by \cite{MR2917409},
       we can prove that if the initial data belongs to some negative Sobolev spaces, the solution for system \eqref{NSPP} will propagate in this space.  This will allow us to obtain some time decay for  $(n,\na\psi,v)$.

        \begin{lem}\label{lemne}
        For $0<s\leq \f{1}{2}$, we have:
        \beqs
        \begin{split}
        &\f{d}{\d t}\int|\Lambda^{-s}n|^2+|\Lambda^{-s}\na \psi|^2+|\Lambda^{-s}v|^2\d x+\int|\Lambda^{-s}\na v|^2+|\Lambda^{-s}\div v|^2\d x\\
       & \lesssim\|\Lambda^{-s}(n,\na\psi,v)\|_{L^2}
       \big(\|n\|_{H^2}^2+\|\na v\|_{H^1}^2+
        \|( \varrho, u)\|_{W^{2,3/s}}(\|(n,v)\|_{H^1}+\|(\vr,u)\|_{H^1})\big).\\
    \end{split}
    \eeqs
        \end{lem}
        \begin{proof}
        Applying  $\Lambda^{-s}$ to  the equations \eqref{NSPP} and multiplying by  $\Lambda^{-s}n,\Lambda^{-s}v$ respectively, we get, after using the Poisson equation:
        \beno
        &&\f{1}{2}\f{d}{\d t}\int|\Lambda^{-s}n|^2+|\Lambda^{-s}\na \psi|^2+|\Lambda^{-s}v|^2\d x+\int |\Lambda^{-s}\na v|^2+|\Lambda^{-s}\div v|^2\d x\nonumber\\
        &=&-\int \Lambda^{-s}v\Lambda^{-s}(u\cdot\na v+ v\cdot\na u+v\cdot\na v)-\Lambda^{-s}v\Lambda^{-s}(\f{1}{\rho+n}-1)(\mathcal{L}u+\mathcal{L}v))\d x
        \nonumber\\
         &&-\int \Lambda^{-s}n\Lambda^{-s}\div(\varrho v+nu+nv)+\Lambda^{-s}{\na \psi}\Lambda^{-s}(\varrho v+nu+nv)\d x
         \nonumber\\
         &\triangleq& H_1+H_2+H_3+H_4.
        \eeno
         We only estimate $H_1,H_2,$ since the other two terms can be handled by similar arguments.
        Using Hölder's inequality and the Hardy-Littlewood-Sobolev inequality: $$\|\Lambda ^{-s}f\|_{L^2}\lesssim\|f\|_{L^{\f{1}{\f{1}{2}+\f{s}{3}}}} \quad0\leq s<\f{3}{2},$$ we get:
        \beno
        H_1&=&-\int \Lambda^{-s}v\Lambda^{-s}(u\cdot\na v+ v\cdot\na u+v\cdot\na v )\d x\nonumber\\
        &\lesssim&\|\Lambda^{-s}v\|_{L^2}\big(\|\na v\|_{L^2}\|u\|_{L^{3/s}}+\|\na u\|_{L^{3/s}}\|v\|_{L^2}+\|\na v\|_{L^2}\|v\|_{L^{3/s}}\big)\nonumber\\
        &\lesssim&\|\Lambda^{-s}v\|_{L^2}\big(\|\na v\|_{H^1}^2+\|u\|_{W^{1,s/3}}\|v\|_{H^1}\big).
       \eeno
       \beno
       H_2&=&\int \Lambda^{-s}v\Lambda^{-s}(\f{1}{\rho+n}-1)(\mathcal{L}u+\mathcal{L}v))\d x\nonumber
       \lesssim \|\Lambda^{-s}v\|_{L^2}\big(\|\na^2 u\|_{L^{\f{3}{s}}}\|(\vr, n )\|_{L^2}
       +\|\na^2 v\|_{L^2}\|(\varrho, n)\|_{L^{3/s}}   \big)\nonumber\\
       &\lesssim&\|\Lambda^{-s}v\|_{L^2}\big(\|\na^2 u\|_{L^{\f{3}{s}}}\|(\vr, n)\|_{L^2}+\|\na^2 v\|_{L^2}^2+\|n\|_{H^2}^2+\|\varrho\|_{L^{3/s}}\|\na^2 v\|_{L^2}\big).
       \eeno

    This ends the proof.
\end{proof}
Now we can prove the decay estimate for $(n,v,\na\psi)$ which is stated in Theorem \ref{thmp}. Here we follow the arguments in \cite{MR2917409} with a few considerations on perturbation terms.

     Step 1: \\
    Prove $(n,\na \psi, v)$ propagate in the negative Sobolev space $\dot{H}^{-s}$.
   We should make use of the damping property of $(n,\na v)$ and decay estimate in time of $(\varrho,u)$ .

   Define $$\mathcal{E}_{-s}=\|(n,\na\psi,v)\|_{H^{-s}}^2.$$
   By Lemma \ref{lemne}, the decay estimate of $(\varrho,u)$: $\|(\vr,u)\|_{W^{2,\f{3}{s}}}\lesssim (1+t)^{-\f{3}{2}(1-\f{2s}{3})}$
    (note $\f{3}{2}(1-\f{2s}{3})>1$ if $0<s<\f{1}{2}$)
    and the damping property of $(n,\na v)$  (see \eqref{eq21}), we have:
    \beno
    &&\sup_{0\leq \tau \leq t}\mathcal{E}_{-s}(\tau)\\
    &\leq &\mathcal{E}_{-s}(0)+C\int_{0}^{t}\|n\|_{H^2}^2+\|\na v\|_{H^1}^2+
   \|( \varrho, u)\|_{W^{2,3/s}}(\|(n,v)\|_{H^1}+\|(\vr,u)\|_{H^1})\d\tau \sup_{0\leq \tau \leq t}\mathcal{E}_{-s}^{\f{1}{2}}\\
    &\leq&\mathcal{E}_{-s}(0)+(\sup_{0\leq \tau \leq t}\mathcal{E}_{-s})^{\f{1}{2}},
    \eeno
    which yields the boundedness of $\|(n,\na\psi,v)\|_{H^{-s}}$ if we suppose $\mathcal{E}_{-s}(0)< +\infty$.
\begin{rmk}
The case $s=\f{1}{2}$ is critical in the sense that the source term $\ep (\vr\mathcal{L}u)$ (which comes from $\ep[(\f{1}{1+\vr+n}-1)\mathcal{L}u)]$) has critical decay $(1+t)^{-1}$ in $\dot{H}^{\f{1}{2}}$.
\end{rmk}

    Step 2: Using interpolation and energy estimate to get new energy inequality,
    and then get the time decay estimate.\\
    By interpolation, we have:
    \beqs
    \|u\|_{L^2}\leq \|u\|_{\dot{H}^{-s}}^{\f{1}{1+s}}\| u\|_{\dot H ^{1}}^{\f{s}{1+s}},
    \eeqs
    which is equivalent to
    \beqs
    \| u\|_{\dot H ^{1}}\geq\|u\|_{L^2}^{\f{1+s}{s}}\|u\|_{\dot{H}^{-s}}^{-\f{1}{s}}.
    \eeqs
    Combined with \eqref{eq20}, we get that:
   \beq \label{eq22}
   \begin{split}
    &\f{\d}{\d t}\mathcal{E}_M+C_{12} \varepsilon(\|n\|_{H^M}^2+\|v\|_{H^M}^2)^{1+\f{1}{s}}
    \leq C_{13} (1+t)^{-b}\varepsilon^{2}+C_{14}  (1+t)^{-a}\mathcal{E}_{M}.
    \end{split}
    \eeq

    We now prove the time decay estimate when $M=3$.
     We recall that we assume $\mathcal{E}_{3}(0)$ is small respect to $\varepsilon$.
 Defining firstly $\beta_s=\f{2}{s}+1$, $f=\exp({-\f{C_{14}}{a-1}(1+t)^{-(a-1)}})\mathcal{E}_3$, then multiplying \eqref{eq22} by $(1+\varepsilon^{\beta_s}t)^{\gamma},
 (s<\gamma<b-1)$
    and integrating in time, we have by Young's inequality:
    \beno
    &&(1+\varepsilon^{\beta_s}t)^{\gamma}f+C_{12}\varepsilon \int_{0}^{t}(1+\varepsilon^{\beta_s}\tau)^{\gamma}f^{1+\f{1}{s}}(\tau)\d\tau\leq f(0)+\gamma \varepsilon^{\beta_s} \int_{0}^{t}(1+\varepsilon^{\beta_s}\tau)^{\gamma-1}f\d \tau+C_{15}\ep^2
    \nonumber\\
    &\leq&(f(0)+C_{15}\ep^2)+\f{C_{12}}{2}\varepsilon \int_{0}^{t}(1+\varepsilon^{\beta_s}\tau)^{\gamma}f^{1+\f{1}{s}}(\tau)\d\tau+C_{16}\varepsilon^2 (1+\varepsilon^{\beta_s}t)^{\gamma-s}\nonumber
    \eeno
   which yields:
   \beno
   f&\lesssim &\varepsilon^2(1+\varepsilon^{\beta_s}t)^{-\gamma}+\varepsilon^{2}(1+\varepsilon^{\beta_s}t)^{-s}
   \lesssim \varepsilon^2 (1+\varepsilon^{\beta_s}t)^{-s}.
   \eeno
 We thus get that:
        \beno
         \|(n,\na\psi,v)(t)\|_{H^3}\lesssim \varepsilon(1+\varepsilon^{\beta_s}t)^{-\f{s}{2}}.
        \eeno
   which, by considering $\ep^{\beta_s}t \lesssim 1$ and  $\ep^{\beta_s}t\gtrsim 1$ respectively, yields
   \ben\label{eqdecay}
   \|(n,\na\psi,v)(t)\|_{H^3}
   \lesssim \min \{\ep,(1+t)^{-\f{s}{2+s}}\}.
   \een
  This ends the proof of Theorem \ref{thmp}.
\end{proof}

   \begin{rmk}
For $M>3$, as we do not assume the initial data
$\|\cP u_{0}^{\ep}\|_{H^M}$ is small proportional to $\ep$, we do not expect that $\|(n,\na\psi,u)\|_{H^M}$has decay like \eqref{eqdecay} which is independent of $\ep$. However, we  could still get some decay in the slow variable "$\ep t$". Defining $g=\exp({-C_{14}\int_{0}^t(1+\tau)^{-a}\d \tau})\mathcal{E}_M=\exp({-\f{C_{14}}{a-1}(1+t)^{-(a-1)}})\mathcal{E}_M$.
    We choose again a constant $\gamma$ with condition $s<\gamma<b-1$  multiply \eqref{eq22} by $(1+\varepsilon t)^{\gamma}$,
   and  integrate then in time, we achieve that:
\begin{align*}
    &(1+\varepsilon t)^{\gamma}g+C_{12}\varepsilon \int_{0}^{t}(1+\varepsilon \tau)^{\gamma}g^{1+\f{1}{s}}(\tau)\d\tau\leq   g(0)+C_{15}\ep^2 +\gamma \varepsilon \int_{0}^{t}(1+\varepsilon \tau)^{\gamma-1}g(\tau)\d \tau\nonumber\\ 
   & \leq g(0)+C_{15}\ep^2+\f{C_{12}}{2}\varepsilon \int_{0}^{t}(1+\varepsilon\tau)^{\gamma}g^{1+\f{1}{s}}(\tau)\d\tau+C_{16}(1+\varepsilon t)^{\gamma-s}
    \end{align*}
   which yields  $g\lesssim(1+\varepsilon t)^{-s}. $
   \end{rmk}

   \begin{rmk}
     By \eqref{eqhigh} and interpolation $\|v\|_{\dot{H}^{k+1}}  \geq\|v\|_{\dot{H}^{k}}^{1+\f{1}{k+s}}\|u\|_{\dot{H}^{-s}}^{-\f{1}{k+s}}$,
     we can also have:
     \beqs
     \|\na^{k}(n,\na\psi,v)\|_{H^{M-k}}\lesssim 
     (1+\varepsilon t)^{-\min\{\f{k+s}{2},\f{1}{3}-\}},
     \eeqs
     \beqs
     \|\na^{l}(n,\na\psi,v)\|_{H^{3-l}}\lesssim   \varepsilon(1+\varepsilon^{\beta_{s,l}} t)^{-\min\{\f{l+s}{2},\f{1}{3}-\}}
     \lesssim \min \{\ep,(1+t)^{
     -\min\{\f{l+s}{2+l+s},\f{1}{3}-\}}\}.
     \eeqs
     where $k=0,1,2\cdots M-1$, $l=0,1,2$ and $\beta_{s,l}=1+\f{2}{l+s}$.
   \end{rmk}




\section{Remarks on more  general pressure laws  and  viscosity coefficients}
In this short section, we will explain briefly how  our results can be easily extended to  more  general pressure laws  and viscosity coefficients. Here, we suppose that the pressure $p(\rho^{\ep})=\f{1}{\gamma}(\r^{\ep})^{\gamma},\gamma>1$,
and $\mu=\mu(\r^{\ep}),\lambda=\lambda(\r^{\ep})$ are both density dependent. We assume that $\mu,\lambda$ are smooth 
 functions in the vicinity of $1$ and that  $\mu(1)>0,\,  2\mu(1)+\lambda(1)> 0$. As previously , we write $(\r^{\ep},\na\phi^{\ep},u^{\ep})=(1+\vr,\na\phi,u)+(n,\na\psi,v)$ where $(1+\vr,\na\phi,u)$ and $(n,\na\psi,v)$ satisfy the following two systems respectively:

 \begin{equation}\label{NSPlow1g}
 \left\{
\begin{array}{l}
\displaystyle \pt \vr +\div u=-\div(\vr u),\\
\displaystyle  \pt u+u \cdot {\na u}-\varepsilon \mathcal{L}_{1} u+
\nabla \vr-\nabla \phi=-\na\big(\f{(1+\vr)^{\gamma-1}}{\gamma-1}-\vr\big), \\
\displaystyle \Delta \phi =\vr,\\
\displaystyle u|_{t=0} =\mathcal{P}^{\perp}u_0^{\varepsilon} ,\vr|_{t=0}=\vr_0=\rho_{0}^{\varepsilon}-1.
\end{array}
\right.
\end{equation}

\beq \label{NSPPg}
\left\{
\begin{array}{l}
\displaystyle \pt n +\div v+\div( \vr v+nu+nv)=0,\\
\displaystyle \pt v+u\cdot {\na v}+v\cdot (\na u+\na v)-\varepsilon\mathcal{L}_{1}v + p'(1+\vr+n)\na n -\na \psi
=\varepsilon(\f{1}{1+\vr+n}-1) (\mathcal{L}_{\rho^{\ep}}v+\mathcal{L}_{\rho^{\ep}}u)   \\
\displaystyle \qquad\qquad\qquad
+\ep(\mathcal{L}_{\rho^{\ep}}-\mathcal{L}_{1})(u+v)+
(p'(1+\vr+n)-p'(1+\vr))\na\vr,\\
\displaystyle \Delta \psi=n,\\
\displaystyle v|_{t=0} =\mathcal{P}u_0^{\varepsilon}, n|_{t=0}=0.
\end{array}
\right.
\eeq
where we denote $$\mathcal{L}_1(u)=\mu(1)\Delta u+(\mu(1)+\lambda(1))\na\div u,
\qquad \mathcal{L}_{\rho^{\ep}}(u)=\div\big(\mu(\rho^{\ep})\na u \big)+\na\big((\mu+\lambda)\big(\rho^{\ep})\div u)\big).$$
For the system \eqref{NSPlow1g},
only minor modifications need to be taken into account due to the extra
term $((1+\vr)^{\gamma-2}-1)\na\vr$.
Indeed, on the one hand, the solution is still irrotational as long as it exists.
On the other hand, the term $((1+\vr)^{\gamma-2}-1)\na\vr$ is essentially quadratic term, since we can expand  it as $\big((\gamma-2)\vr+g(\vr)\big)\na\vr$ where $g(x)$ is smooth when $x>-1$ and $g(0)=g'(0)=0$, therefore we can consider the term $g(\vr)\na\vr$ as cubic term. In the process of decay estimate for the solutions to \eqref{NSPlow1}, we just have to  perform 
an additional  normal form transformation for the quadratic  term $(\gamma-2)\vr \na\vr$, since the term $g(\vr)\na\vr$ is already cubic.
 The energy estimates can be obtained in a classical way for general pressure laws.

As for the system \eqref{NSPPg}, one can still perform energy estimates.  The last two extra terms also can be controlled easily. In fact, we can write:
 \beno
 &&\qquad \qquad \qquad\ep(\mathcal{L}_{\rho^{\ep}}-\mathcal{L}_{1})(u+v)\\
&&= \mu'(1)\Delta(u+v)+(\mu+\lambda)(1)\na\div(u+v)+\div\big(h_1(\vr+n)\na (u+v)\big)+\na\big(h_2(\vr+n)\div(u+v)\big),\\
&&\qquad \qquad \qquad(p'(1+\vr+n)-p'(1+\vr))\na\vr
 =p''(1) n\na\vr+ (h_{3}(\rho)  +h_4 (\rho + n))\na \vr
\eeno
where $h_j(j=1,2,3,4)$ are smooth functions and satisfy:
$h(0)=h'(0)=0.$
All the terms in the above two identities can be controlled similarly as in
Section 5,  indeed the new  higher order terms have faster decay.
     \section {Navier-Stokes-Poisson system for ion dynamics}
    In this section, we consider the ion dynamic Navier-Stokes-Poisson system \eqref{NSPION}. We shall  give a sketch of the proof of  Theorem \ref{thmion}.
\subsection{A viscous perturbation of ion Euler-Poisson }

Following the global scheme of the proof for the electrons case, we shall first study
 the following intermediate system which has the property of propagating curl free solutions.
 \begin{equation}\label{NSPlowion}
 \left\{
\begin{array}{l}
\displaystyle \pt \vr +\div u=-\div(\vr u),\\
\displaystyle  \pt u+(u \cdot \na) u-2\ep \Delta u+
\nabla \vr-\nabla \phi=0, \\
\displaystyle (\Delta-1) \phi =\vr,\\
\displaystyle u|_{t=0} =\mathcal{P}u_0^{\varepsilon} ,\vr|_{t=0}=\vr_0=\rho_{0}^{\varepsilon}-1.
\end{array}
\right.
\end{equation}
  We first  prove the following result:
\begin{prop}\label{propiNSPlow}
  There exists $\delta _{3}>0$ 
such that for any family of initial data satisfying
\beno
&&\sup_{\varepsilon\in(0,1]} \left( \|(\vr_0^{\varepsilon},\mathcal{P}^{\perp}u_0^{\varepsilon})\|_{W^{\si+\f{9}{4}(1+\kappa),8_{\kappa}'}}+\||\na|^{-1}(\vr_0^{\varepsilon},\mathcal{P}^{\perp}u_0^{\varepsilon})\|_{H^{N}} \right) \leq \delta_3
\eeno
with  $\sigma \geq 5, N\geq 2\sigma+1,$ $8_{\kappa}=\f{8}{1-3\kappa},\kappa=\f{1}{200}.$
Then 
for every $\ep \in (0, 1]$,
there exist a unique solution for system \eqref{NSPlowion} in
$C([0,\infty),H^N)$.
Moreover, there exists a constant $C>0$ 
such that for every $\ep \in (0, 1]$, we have the estimate
\beqs
\|(\vr,\na\phi,u)(t)\|_{W^{\si,8_{\kappa}}}\leq C\delta_3 (1+t)^{-(1+\kappa)}, \quad \forall t \geq 0.
\eeqs

\end{prop}
\begin{rmk}
The choice of the $L^p$ type exponent $8_{\kappa}$ in the above result comes from  a constraint in order to get  continous properties of the bilinear operators used in the normal form transformation  and the slow decay of viscous term. More explanations will be given  after  Proposition \ref{sketch of multiplier}.
\end{rmk}


      Let $h=\sqrt{1+(1-\Delta)^{-1}}\varrho,$
    $c=\f{\div}{|\na|}u$, then we get as a counterpart of  \eqref{eqsym}, 
     \beqs
 \left\{
\begin{array}{l}
\displaystyle \pt h + p(|\na|) c=q(|\na|)\div\big(\f{h}{\sqrt{1+(1-\Delta)^{-1}}}\cdot \R c\big) ,\\
\displaystyle \pt c-p(|\na|) h-2\ep \Delta c=|\na| |\R c|^2,\\
\displaystyle h|_{t=0}=\f{1} {\sqrt{1+(1-\Delta)^{-1}}}\vr_0,c|_{t=0}= \R^{*} u_0 .\\
\end{array}
\right.
 \eeqs
 where $p(|\na|)=|\na|\sqrt{1+(1-\Delta)^{-1}}$.
 We note that we still have $\eqref{eqV}-\eqref{eqA}$ by replacing $\lnr$ with $p(|\na|)$.
 We will analyze the high and low frequency separately as before.
   As for high frequency, similar
   arguments as in Lemma \ref{lemmahf}
   show that the smoothing effect of
   $\chi^{H}e^{-tA}$ is still true. We now focus on the low frequency. To start, we  need to analysis $b(r)=\sqrt{p(r)^{2}-(\ep r^2)^2}$ precisely.

    \begin{lem}\label{elementary}
    Suppose $0<\ep\leq 1$, $\kpz$ small enough, then on the region  $\{\ep r^2\leq 2\kpz\}$, $b(r)$ satisfies the following property:\\
    1. $b'(r)\geq c_1(\kpz)>0,$\\
    2. $b''(r)$ only have one zero point $r_0^{\ep,\kpz}$ and $1\leq r_0^{\ep,\kpz}\leq 10,$\\
    3. There exists a  small interval 
    $[r_0^{\ep,\kpz}-\iota,r_0^{\ep,\kpz}+\iota]$
    st. $b'''(r)\geq c_2> 0$, where $c_2$ is a small constant independent of $\ep$.\\
    \end{lem}
   \begin{proof}
    \beqs
    b'(r)=\sqrt{\f{2+r^2}{1+r^2}-\ep^2 r^2}-r\f{\f{r}{(1+r^2)^2}+\ep^2 r}{\sqrt{\f{2+r^2}{1+r^2}-\ep^2 r^2}}=\f{1+\f{1}{(1+r^2)^2}-2\ep^2 r^2}{\sqrt{1+\f{1}{1+r^2}-\ep^2 r^2}}\geq \f{1}{2\sqrt{2}}
    \eeqs
    on the support of $\chidix$
    if $4\ep \kpz\leq 4\kpz\leq \f{1}{2}$.

    2. After direct computations, we have that:
    \beno
   && b''(r)=(1+r^2)^{-4}(\f{2+r^2}{1+r^2}-\ep^2 r^2)^{-\f{3}{2}}r\{[1-(5-8\ep^2)\ep^2r^4]r^4\\
   &&\qquad\qquad-[2-2\ep^4 r^8 +(22-12\ep^2)\ep^2 r^4]r^2-[6+(31-8\ep^2)\ep^2 r^4+(20\ep-2\ep^3)\ep r^2+6\ep^2]\}.
  \eeno

   Note that if $\ep\leq 1$ and  $\kappa_0$ is small enough, then on the region  $\{\ep r^2\leq 2\kpz\}$, we have
   that for $\kappa_{0}$ sufficiently small, the polynomial in the bracket is a small perturbation of
   $ r^4 -  2 r^2 -  6(1 + \ep^2)$
    which 
    has only one real simple positive root   that is uniformly in  $[2, 9]$.
    Therefore, for $\kappa_{0}$ small enough $b''$ has only one nonnegative zero
     which is uniformly for $\ep \in (0, 1]$ in $[1, 10]$.
%

   3. For simplicity, we write $r_0=r_{0}^{\ep,\kappa_0}$. One can check that:
   \beno
   b'''(r_0)
  &=&(1+r_{0}^2)^{-4}(\f{2+r_0^2}{1+r_0^2}-\ep^2r_0^2)^{-\f{3}{2}}\{[4-8(5-8\ep^2)\ep^2 r_{0}^{4}]r_0^{4}\\
  &&\qquad\qquad\qquad\qquad-[4-20\ep^4 r_0^{8}+6(22-12\ep^2)\ep^2r_0^{4}]r_0^2-[4(31-8\ep^2)\ep^2r_0^2-2\ep r_0^{2}]\}\\
  &\define&(1+r_{0}^2)^{-4}(\f{2+r_0^2}{1+r_{0}^2}-\ep^2r_0^2)^{-\f{3}{2}}
  (a_1 r_0^4+a_2 r_0^2-a_3)
  \eeno
 Notice that when $\kpz$ is very close to 0, $a_1$ and $a_2$ are very close to 4 while $a_3$ is very close to 0. 
   Therefore, as long as $\kpz$  is small enough, there exist constants $\iota, c_2$ which are independent of $\ep,$ st.
 $b'''(r)\geq c_2> 0$ on the interval $[r_{0}-\iota,r_0+\iota]$. 
\end{proof}
    This lemma in hand, we could keep track of the proof of Lemma 3.1-3.3 of \cite{MR2775116} to get:

    \begin{lem}\label{lemdecayion}
    Suppose $\kpz$ satisfy the assumptions of the Lemma \ref{elementary}, then
    \beqs
   \|e^{itb(D)}\chidid f\|_{L^{\infty}}\lesssim_{\kpz}(1+| t |)^{-\f{4}{3}}\|f\|_{W^{3,1}},\quad \forall t \in \mathbb{R}
  \eeqs
    \beqs
   \|e^{itb(D)}\chidid f\|_{L^{p}}\lesssim_{\kpz}(1+ | t |)^{-\f{4}{3}(1-\f{2}{p})}\|f\|_{W^{3(1-\f{2}{p}),p'}}, \quad \forall t \in \mathbb{R}.
    \eeqs
    \end{lem}
 We omit the proof, since thanks to the above properties of $b$, it  follows exactly the same lines as in \cite{MR2775116}
  in the same way as the proof of Lemma \ref{lemlow} was following the proof for the classical Klein-Gordon equation. Note again that the above estimates are independent of  $\ep$.

    To treat  low frequencies, we need also to get some continous property of $T_{m/\phi_{j,k}}$ on $L^p$.
    \begin{lem}\label{lembilinearion}
    Bilinear estimate:
Define $\phi_{j,k}(\xi,\eta)=(-1)^{j+1}b(\xi)+(-1)^{k+1}b(\eta)-b(\xi+\eta ),$ $j,k=1,2$
     \beno
    m(\xi,\eta)=\tchidix \tchidie \tchidiemx |\xi|n_1(\xi)n_2(\xi-\eta)n_3(\eta).
     \eeno
     where $n_1,n_2,n_3$ are homogeneous-0 functions whose corresponding multiplier is bounded in $L^p(1< p<+\infty)$.
     By choosing $\kpz$ smaller if necessary, we have similar results as in Proposition 6.1 in \cite{MR2775116}:
     ie.
     \ben\label{ineqbilinearion}
     \|T_{\f{m}{\phi_{j,k}}}(f,g)\|_{W^{\sigma,p'}}\lesssim_{\kpz}\||\na|^{-1}f\|_{H^{\sigma+\lambda}}\||\na|^{-1}g\|_{W^{\lambda,r}}+\||\na|^{-1}f\|_{W^{\lambda,r}}\||\na|^{-1}g\|_{H^{\sigma+\lambda}}.
     \een
     where $\lambda\geq \f{9}{4}+\kappa$, and $\f{1}{r}+\f{1}{p}=1-\f{\f{5}{4}-\kappa}{3}$, $2\leq p,r \leq \f{12}{1+4\kappa}$, $\kappa$ can be chosen very small.
    \end{lem}
    \begin{proof}
        This Lemma is a consequence  of the next proposition along with Theorem 6.1 of \cite{MR2775116} which
        states that if
        $\cM(\xi,\eta)$ satisfies
        $\|\cM\|_{L_{\xi}^{\infty}\dot{H}_{\eta}^{s}}+\|\cM\|_{L_{\eta}^{\infty}\dot{H}_{\xi}^{s}}<\infty$, then
        \beno
       \|T_{\cM}(f,g)\|_{L^{p'}}\lesssim\|g\|_{L^2}\|f\|_{L^{r}},
        \eeno
     where $\f{1}{r}+\f{1}{p}=1-\f{s}{3}, 2\leq p,r \leq \f{6}{3-2s}$.
     \end{proof}
   \begin{prop}\label{sketch of multiplier}
   Define $\cM_{jk}(\xi,\eta)=\f{\lxr^{\si}|\xi||\eta||\xi-\eta|}{\phi_{jk}\lxmer^{\lambda+\si}\ler^{\lambda}}\Phi(\f{|\eta|}{|\xi-\eta|})\tchidix \tchidie \tchidiemx$ where
   $\Phi\in C_{c}^{\infty}(\mathbb{R})$ is supported in $B_2(\mathbb{R})$, then for any $\kappa>0$, if $\lambda>\f{9}{4}+\kappa$, then
   the following estimate holds:
   \beqs
   \|\cM_{jk}\|_{L_{\xi}^{\infty}\dot{H}_{\eta}^{\f{5}{4}-\kappa}}+\|\cM_{jk}\|_{L_{\eta}^{\infty}\dot{H}_{\xi}^{\f{5}{4}-\kappa}}\lesssim _{\kappa} 1.
   \eeqs
   \end{prop}
   \begin{proof}
 For the proof of this
  proposition, we can  adapt the proof of Proposition 6.1 in \cite{MR2775116}. We only explain for the case $\phi_{11}$ as  other cases could be obtained by symmetry.
   We  split $\mathbb{R}^3$ into three regions $\{|\eta|<\f{1}{3}|\xi|\},
  $ $\{|\xi|<\f{1}{3}|\eta|\}$ and $\{\f{1}{4}\leq\f{|\xi|}{|\eta|}\leq 4\}$. For example, on the region $\{|\eta|<\f{1}{3}|\xi|\},$
  to estimate $\|\cM_{11}\|_{L_{\eta}^{\infty}\dot{H}_{\xi}^{s}}$, one first fix $\eta$ and compute the  $\|\varphi_{l}(\xi)\cM_{11}\|_{\dot{H}_{\xi}^s}$ norm by interpolation between $\|\varphi_{l}(\xi)\cM_{11}\|_{L^2}$ and
  $\|\varphi_{l}(\xi)\cM_{11}\|_{\dot{H}_{\xi}^2}$ for any $l$ (recall $\phi_{l}$ is $l$-th dyadic function),
  and find the optimal number $s$ (which finally turns out to be $\f{5}{4}-$) such that it is summable for $l$ uniformly for $\eta$.
  In light of this strategy, one sees that the main ingredients are the elementary estimates for
  $\phi_{11}$. 
   We list briefly the properties needed for $\phi_{11}$ which are essentially the same as Lemma 6.3 and Lemma 6.4 of \cite{MR2775116}.

  1. Lower boundedness of $\phi_{11}$.

  If $|\xi|\leq \min\{|\xi-\eta|,|\eta|\}$, then
   $|\phi_{11}(\xi,\eta)|=|b(\xi-\eta)+b(\eta)-b(\xi)|\geq_{\kpz} \max\{|\xi-\eta|,|\eta|\};$\qquad

   if $|\xi|$ is not smallest, for example, $|\eta|\leq\min\{|\xi-\eta|,|\xi|\}$, then $|\phi_{11}(\xi,\eta)|\gtrsim_{\kpz} \f{|\xi||\eta||\xi-\eta|}{\ler^2\lxmer\lxr}+|\eta|(1-\cos\beta+1-\cos\theta).$ where
   $\beta,\theta$ are the angle between $\eta$ and $\xi-\eta$, $\eta$ and $\xi$ respectively.

  2. The first and second derivative for $\phi_{11}$ can be estimated as
  \beno
  |\p_{\xi}\phi_{11}|\lesssim_{\kpz}
  \f{|\eta|}{\langle\max\{|\xi-\eta|,|\xi|\}\rangle\langle\min\{|\xi-\eta|,|\xi|\}\rangle^2}+2|\sin \f{\gamma}{2}|,\\
  |\p_{\eta}\phi_{11}|\lesssim_{\kpz} \f{|\xi|}{\langle\max\{|\xi-\eta|,|\xi|\}\rangle\langle\min\{|\xi-\eta|,|\xi|\}\rangle^2}+2|\sin \f{\beta}{2}|,\\
  |\Delta_{\xi}\phi_{11}(\xi,\eta)|\lesssim_{\kpz}\f{|\eta|}{|\xi-\eta||\xi|}, \qquad  |\Delta_{\eta}\phi_{11}(\xi,\eta)|\lesssim_{\kpz}\f{1}{\min\{|\xi-\eta|,|\eta|\}}.
  \eeno
  where $\gamma$ denotes the angle between $\xi$ and $\xi-\eta$.

  Nevertheless, as in \cite{MR2775116}, all the information needed for $b(r)$ and $q(r)=\f{1}{r}b(r)=\sqrt{\f{2+r^2}{1+r^2}-\ep^2 r^2}$ to prove the above two properties are the following facts which are consistent  with the case $\ep=0$.

  (1) $b''(r)\lesssim_{\kpz} 1$,

  (2) if $\kappa_0$ is sufficient small,
  one still has that  there exists two constants $K_1,K_2$ which are independent of $\ep\in(0,1]$, st.
$$
  -q'(r)\approx \f{1}{r}, \quad when \quad r\leq K_1,$$
  $$-q'(r)\approx \f{1}{r^3}, b''(r)\approx \f{1}{r^3}, b'''(r)\approx \f{1}{r^4} \quad when \quad r\geq K_2.$$

 Since the above two facts are easy to see, we omit the proof.
 \end{proof}
 From now on, we fix $\kpz$ such that Lemma \ref{elementary}(1-3), Lemma \ref{lemdecayion} and Lemma \ref{lembilinearion} holds.

   In view of Lemma \ref{lemdecayion}, \ref{lembilinearion},
   we can  only expect to get decay estimates in some $L^p$ framework with $8<p<12$ (due to the appearance of ‘time resonances’, we can only perform the  normal form transformation  one time). To overcome the difficulty that $\|\f{\ep \Delta}{|\na|}\chidid R\|_{L^2}$ decays only like $(1+t)^{-1}$, we need to use a  ‘slow’ decay estimate for $\||\na|^{-1}R\|_{L^r}$ where $r$ is larger but  close to 2. By Lemma \ref{lembilinearion}, if  we choose $p$ larger,  we need to estimate $\||\na|^{-1}R\|_{L^r}$ for a smaller $r$ which
   obviously has slower decay. Therefore, to close our decay estimate, we need to choose $p$ small,
   this is why we  choose $p=8_{\kappa}$, where $\f{1}{8_{\kappa}}=\f{1}{8}-\f{3\kappa}
   {8}$. By this choice, we have that:
    $$\|e^{itb(D)}\chidid f\|_{L^{8_{\kappa}}}\lesssim (1+t)^{-(1+\kappa)}\|f\|_{W^{\f{9}{4}(1+\kappa),8_{\kappa}'}}.$$

\begin{proof}[Proof of Proposition \ref{propiNSPlow}]
     We shall use the norm:
    \beno
    \|V\|_{X_T}\define (1+t)^{-(1+\kappa)}\|V\|_{W^{\sigma,\ek}}+(1+t)^{-(1+\kappa)}\|\cchidid V\|_{H^{N-2}}+\||\na|^{-1}V\|_{H^N},
    \eeno
    \beno
    \|f\|_{Y}\define\|f\|_{W^{\sigma+\f{9(1+\kappa)}{4},8_{\kappa}'}}+\||\na|^{-1}f\|_{H^N}.
      \eeno
      where $\si\geq 5, N\geq 2\sigma+1.$\\
      Global existence for $(\rho,u)$ follows if we prove the a priori estimate:
      \beq \label{ineqion}
 \|V\|_{X_T} \lesssim \|V_0\|_{Y}+\|V\|_{X_T}^{\f{3}{2}}+\|V\|_{X_T}^2+\|V\|_{X_T}^3.
     \eeq
      Sketch of the proof of \eqref{ineqion}:\\
      1. The bound for $H^N$ norm. We can perform energy estimates in the same way as in  Proposition \ref{propenergy}. One only needs to change the norm a little bit by
      \beno
     E_N=\sum_{|\alpha|\leq N}E_{\alpha}=\sum_{|\alpha|\leq N}\int \frac{|\pa \vr|^2}{2}+\frac{|\pa \lnr \phi|^2}{2}+\rho \frac{|\pa u|^2}{2}\d x.
      \eeno
      2. The bound for $H^{-1}$ norm.\\
    It can easily be seen that the nonlinear terms are under the form  $B_{l}(V,V)= \sum |\na|n_1(D)\big(n_2(D)V n_3(D)V\big)$ where $n_1(D),n_2(D)$, $n_3(D)$ are  $L^p (1<p<+\infty)$ multipliers.
    So by Duhamel's principle, tame estimates and Sobolev embedding, we have:
    \begin{align*}
    \|\na|^{-1}V\|_{L^2}&\lesssim \|\na|^{-1}V_0\|_{L^2}+\int_{0}^{t}\||\na|^{-1}B(V,V)\|_{L^2}\d s
   \lesssim\|\na|^{-1}V_0\|_{L^2}+\int_{0}^{t}\|V\|_{L^2}\|V\|_{W^{\sigma,\ek}}\d s\nonumber\\
&\lesssim\|\na|^{-1}V_0\|_{L^2}+\|V\|_{X_T}^2.
    \end{align*}
    3. Estimates of $\|\chi^{H} V\|_{H^{N-2}}$
    and $\|\chi^{H} V\|_{W^{\si,\ek}}$ can be performed  in the same fashion as in the electron case, we thus skip them.\\
   4. Estimate of $\|\chidid V\|_{W^{\si,\ek}}$.
    For clarity, we will use the same notation as in the electron case. More precisely, we set $R=Q^{-1}\chidid V=\sum_{k=1}^{4}J_k$
    where $J_1-J_4$ are defined in \eqref{eq for R}. Nevertheless,
    $J_1,J_2,J_3$ can be easily estimated using the Kato-Ponce inequality (Lemma $\ref{lemmakp}$), we thus  omit their estimate.

  Now, it remains to estimate the typical term of $J_4$: $G_{jk}=\sum_{j=1}^{7}I_j$,  which is  defined in the same way as in the electron case (with slightly adaptation of multiplier m and $n_j$), see \eqref{def of Ij}.
  We need to prove that
  \beno
  \|G_{j,k}\|_{W^{\sigma,8_{\kappa}}}\lesssim_{\kpz}(1+t)^{-(1+\kappa)}\|V\|_{X_T}.
  \eeno
 In the above decomposition,    $I_1,I_2,I_5,I_7$ correspond to boundary terms and cubic terms,
    which have essentially been treated in \cite{MR2775116} where the authors proved global existence for the ions Euler-Poisson system. Note that in \cite{MR2775116}, the authors
    proved $L^{10}$ decay estimate. Nevertheless, it is the same  to prove decay in $L^{8_{\kappa}}$ framework, we
    leave the details.
    We will only detail the estimate of the  'viscous term' $I_4$, since $I_6$ is ‘symmetric’ term and the estimate for  $I_3$  can be reduced to that for $I_4$.

    To start, we prove the following two claims:\\
    Claim 1:
    $$
    \||\na|^{-1}R(t)\|_{W^{\lambda,r}}\lesssim (1+t)^{-\kappa}\|V\|_{X_T},$$
    where $\f{1}{r}=\f{11+17\kappa}{24}$, $\lambda=\f{9}{4}+\kappa.$
    \\
 Claim 2:
    \beno
   && \|\f{\ep \Delta}{|\na|}R\|_{H^{N-1}}\lesssim (1+t)^{-1}(\||\na|^{-1}V_0\|_{H^{N-1}}+\|V\|_{X_T}^2),\\
&& \|\f{\ep \Delta}{|\na|}R\|_{W^{\lambda,r}}\lesssim (1+t)^{-(1+\kappa)}(\|V_0\|_{Y}+\|V\|_{X_T}^2).
    \eeno
{\bf Proof of  Claim 1: } By interpolation, we have
    \beno
    \||\na|^{-1}R\|_{W^{\lambda,\ek}}\lesssim \|R\|_{W^{\lambda,\f{24}{11-9\kappa}}}\lesssim \|R\|_{W^{\lambda,\ek}}^{\theta}\|R\|_{H^{\lambda}}^{1-\theta}\lesssim(1+t)^{-\f{1+9\kappa}{9}}\|V\|_{X_T}.
    \eeno
    where $\theta=\f{1+9\kappa}{9(1+\kappa)}$.
    Claim 1 follows from  another interpolation, that is:
    $$\||\na|^{-1}R(t)\|_{W^{\lambda,r}}\lesssim \||\na|^{-1}R(t)\|_{W^{\lambda,8_{\kappa}}}^{\vartheta}\||\na|^{-1}R(t)\|_{H^{\lambda}}^{1-\vartheta}\lesssim(1+t)^{-\f{1}{100}}\|V\|_{X_T}.$$
    where $\vartheta=\f{1-17\kappa}{9(1+\kappa)}$ and $\f{1-17\kappa}{9(1+\kappa)}\f{1+9\kappa}{9}\geq \f{1}{100}$ if we choose $\kappa$ small enough,
    say $\kappa\leq \f{1}{200}$. \\
{\bf Proof of  Claim 2:}
    The first inequality can be proved like Lemma \ref{lemepd}, we thus do not detail it.
    For the second inequality, we have by
    the decay estimate
    \eqref{lemdecayion} and  the Kato-Ponce inequality (Lemma \ref{lemmakp}) that
    \begin{align*}
     &\|\f{\ep \Delta}{|\na|}R\|_{W^{\lambda,r}}
    \lesssim\| \left(
  \begin{array}{cc}
    e^{\lm(D) t}&0\\
    0& e^{\lp(D) t}\\
  \end{array}
\right)\f{\ep \Delta}{|\na|}R_0\|_{W^{\lambda,r}}\\
&\qquad\qquad+\int_{0}^{t}\|\left(
  \begin{array}{cc}
    e^{\lm(D) (t-s)}&0\\
    0& e^{\lp(D) (t-s)}\\
  \end{array}
\right)\f{\ep\Delta}{|\na|}Q^{-1}\chidid B(V,V)\|_{W^{\lambda,r}}\d s\\
&\lesssim (1+t)^{-\f{10-17\kappa}{9}}\||\na|^{-1}V_0\|_{W^{\lambda+3(1-\f{2}{r}),r'}}+\int_{0}^{t} (1+t-s)^{-\f{10-17\kappa}{9}}\||\na|^{-1}B(V,V)\|_{W^{\lambda+3(1-\f{2}{r}),r'}}\d s\\
&\lesssim(1+t)^{-\f{10-17\kappa}{9}}\||\na|^{-1}V_0\|_{{W^{\si,r'}}}+\int_{0}^{t} (1+t-s)^{-\f{10-17\kappa}{9}}\|V\|_{H^{\si}}\|V\|_{W^{\lambda+3(1-\f{2}{r}),{\f{24}{1-17\kappa}}}}\d s\\
&\lesssim(1+t)^{-(1+\kappa)}(\||\na|^{-1}V_0\|_{{W^{\si,r'}}}+\|V\|_{X_T}^2)\lesssim (1+t)^{-(1+\kappa)}\|V_0\|_{Y}+\|V\|_{X_T}^2),
\end{align*}
  where $r'=\f{24}{13-17\kappa}$. Note that  $\lambda+3(1-\f{2}{r})\leq \si-1.$ In the last inequality, we used the fact $\f{1}{r'}+\f{1}{3}<\f{1}{\ek'}$ and interpolation to get:
    $$\||\na|^{-1}V_0\|_{W^{\si,r'}}\lesssim\|V_0\|_{W^{\si,\f{24}{21-17\kappa}}}\lesssim \|V_0\|_Y.$$
  These two claims, combine with
   the bilinear estimate \eqref{ineqbilinearion}, we can estimate
   $$I_4=-i \int_{0}^{t}e^{\ep (t-s)\Delta}e^{i(t-s)b(D)} \chidid T_{\f{m}{\phi_{jk}}}(\ep \Delta \tilde{r}_j,\tilde{r}_{k})$$ as follows:
    \beno
   \|I_4\|_{W^{\sigma,\ek}}&\lesssim&\int_{0}^{t}(1+t-s)^{-(1+\kappa)}\|T_{\f{m}{\phi_{jk}}}(\ep \Delta \tilde{r}_j,\tilde{r}_{k})\|_{W^{\si+\f{9(1+\kappa)}{4},\ek'}}\d s\nonumber\\
    &\lesssim&\int_{0}^{t}(1+t-s)^{-(1+\kappa)}(\|\f{\ep \Delta}{|\na|}\tilde{r}_j\|_{H^{\si+\lambda+\f{9(1+\kappa)}{4}}}\||\na|^{-1}\tilde{r}_{k}\|_{W^{\lambda, r}}+\|\f{\ep \Delta}{|\na|}\tilde{r}_j\|_{W^{\lambda,r}}\||\na|^{-1}\tilde{r}_{k}\|_{H^{\si+\lambda+\f{9(1+\kappa)}{4}}})\d s\nonumber\\
    &\lesssim&\int_{0}^{t}(1+t-s)^{-(1+\kappa)}(1+s)^{-(1+\kappa)}\|V\|_{X_T}^2\d s\lesssim(1+t)^{-(1+\kappa)}\|V\|_{X_T}^2.
    \eeno
   For the estimate of $I_3$,
we use the  identity
$$\ep \Delta  T_{\f{m}{\phi_{jk}}}(\tilde{r}, \tilde{r})=T_{\f{m}{\phi_{jk}}}( \ep\Delta\tilde{r}, \tilde{r})+2 \sum_{l=1}^{3} T_{\f{m}{\phi_{jk}}}( \ep^{\f{1}{2}}\partial_{l}\tilde{r}, \ep^{\f{1}{2}}\partial_{l}\tilde{r})$$
 and the following inequalities whose proofs are similar to that of Claim 2.
     \beno
   && \|\f{\ep ^{\f{1}{2}}\na}{|\na|}R\|_{H^{N-1}}\lesssim (1+t)^{-\f{1}{2}}(\|V_0\|_{Y}+\|V\|_{X_T}^2),\\
&& \|\f{\ep ^{\f{1}{2}}\na}{|\na|}R\|_{w^{\lambda,r}}\lesssim (1+t)^{-(\f{11-34\kappa}{18})}(\|V_0\|_{Y}+\|V\|_{X_T}^2).
    \eeno
This ends the proof of a priori estimate $\ref{ineqion}$.
\end{proof}


\subsection{Perturbing the ion Navier-Stokes-Poisson by the solution of (\ref{NSPlowion})} 

 As before, we consider now the following system:
  \beq \label{ionNSPP}
 \left\{
\begin{array}{l}
\displaystyle \pt n +\div( \rho v+nu+nv)=0,\\
\displaystyle \pt v+u\cdot {\na v}+v\cdot (\na u+\na v)-\varepsilon\mathcal{L}v +\na n -\na \psi

=\varepsilon(\f{1}{\rho+n}-1) (\mathcal{L}v+\mathcal{L}u), \\
\displaystyle \Delta \psi-\psi=n\\
\displaystyle v|_{t=0} =\mathcal{P}u_{+0}^{\varepsilon},

n|_{t=0}=0.
\end{array}
\right.
\eeq
 then
 $(\rho_{+}^{\varepsilon},u_{+}^{\varepsilon},\phi_{+}^{\varepsilon})=(n,\psi,v)+(\rho,u,\phi)$.

 We define the energy functional similar to \eqref{energyfunele}:
  \beqs\label{energyfunion}
    \mathcal{E}_M(n,u,\na\psi)=\sum_{|\alpha|\leq M} \mathscr{E}_{\alpha}=\sum_{|\alpha|\leq M}\f{1}{2}\int\rho|\pa v|^2+|\pa n|^2+|\pa \lnr \psi|^2\d x.
    \eeqs

 We can derive similar  energy estimates as in the electron case by using almost  the same  computations as  in Lemma \ref{lemenergy1} and Lemma \ref{lemenergy2}.
In fact, one can check that \eqref{energyeq1} in Lemma \ref{lemenergy1} do not change, while \eqref{energyeq2} in Lemma \ref{lemenergy2} is changed by replacing $\int |\pa n|^2\d x$ by $\int |\pa \psi|^2+|\pa\Delta\psi|^2\d x$.
We finally get  the following a priori estimate: if we have
 $\|(u,\vr)\|_{X}\lesssim \delta, \mathcal{E}_3 \lesssim \delta \ep$ for some $\delta$ sufficiently small independent
of $\ep$, then  we have uniformly in $\ep$:
$$\mathcal{E}_3(t)\lesssim \mathcal{E}_3(0)+\int_{0}^{t}(1+s)^{-(1+\kappa)}(\delta\mathcal{E}_3(s)+\ep^2\delta^3)\d s.$$
Global existence for system \eqref{ionNSPP} follows again  by Gr\"{o}nwall's inequality and
bootstrap arguments.
The decay estimate follows in the similar way as that in electron case, the only difference  is now  that it is the $L^8$ norm  rather than the $L^6$  has the critical decay $(1+t)^{-1}$.

    \section{Appendix}
    We first recall two classical estimates:
    \begin{lem}[Kato-Ponce inequality]\label{lemmakp}
    Given  real number $s>0$,
     two functions $f,g$, we have:
    \beq
       \|fg\|_{{W}^{s,q}}\lesssim\|f\|_{{W}^{s,p_1}}\|g\|_{L^{r_1}}+\|f\|_{L^{r_2}}\|g\|_{{W}^{s,p_2}},
    \eeq
    \beq
    \|fg\|_{\dot{W}^{s,q}}\lesssim\|f\|_{\dot{W}^{s,p_1}}\|g\|_{L^{r_1}}+\|f\|_{L^{r_2}}\|g\|_{\dot{W}^{s,p_2}}
    \eeq
    where $\f{1}{p_j}+\f{1}{r_j}=\f{1}{q}$, $q\leq p_j<+\infty,$ $q<r_j\leq +\infty$.
    \end{lem}
    \begin{lem}\label{lemGN}
    Suppose $F$: $\mathbb{R}\rightarrow \mathbb{R}$ is a smooth function with the condition $F(0)=0$. Then for any function $u$ that  belongs to $L^{\infty}\cap W^{k,p}$ ($k\geq 0$ is an integer and $1\leq p\leq+\infty$), we have:
    \beq
    \|F(u)\|_{\dot{W}^{k,p}}\lesssim C(\|u\|_{L^{\infty}})\|u\|_{\dot{W}^{k,p}}.
    \eeq
    \end{lem}

    \begin{proof}
    For $k=0$, we Taylor expand $F$ at first order. For $k>0$, we use the Gagliardo-Nirenberg  interpolation inequality.
    Indeed, for any $|\alpha|=k>0,$ we have:$$\partial^{\alpha}F(u)=\sum F^{(l)}(u)\partial^{\alpha_1}u\partial^{\alpha_2}u\cdots\partial^{\alpha_l}u.$$
    where $\alpha_1+\alpha_2+\cdots\alpha_l=\alpha$
     and by using
    \beqs
    \|u\|_{\dot{W}^{|\alpha_j|,p_j}}\lesssim\|u\|_{L^{\infty}}^{1-\f{|\alpha_j|}{k}}\|u\|_{\dot{W}^{k,p}}^{\f{|\alpha_j|}{k}}.
    \eeqs
    where $p_j|\alpha_j|=kp.$
    The result  follows from the  H\"{o}lder inequality.
\end{proof}
At last, we present the proof of the bilinear estimate stated in  Lemma \ref{lembilinear}. We first give a proposition which shows that $\f{m}{\phi_{jk}}$ has the same properties as the  Klein-Gordon phase $\f{1}{\pm \lxr \pm \ler -\lxper}$ (\cite{MR3024265}\cite{MR3274788}) as long as the threshold $\kpz$ is small enough.
\begin{prop}\label{elementary for phase}
Let $m$ and $\phi_{jk}$ defined in  \ref{lembilinear},  if $\kpz$ is small enough 
then for any multi-index $\alpha,\beta\in \mathbb{N}^3$, we have the following estimate uniformly in $\ep\in(0,1]$:
\beqs
|\pab \f{m}{\phi_{jk}}(\xi-\eta,\eta)|\lesssim_{\alpha,\beta,\kpz} \min\{\lxr, \ler, \lxmer\},
\eeqs
\beqs
|\pab \f{m}{\phi_{jk}^2}(\xi-\eta,\eta)|\lesssim_{\alpha,\beta,\kpz} \min\{\lxr^2, \ler^2, \lxmer^2\}.
\eeqs
\end{prop}

We  postpone the proof of this proposition 
and prove firstly Lemma \ref{lembilinear}.
\begin{proof}[Proof of Lemma \ref{lembilinear}]
 We choose two smooth functions $\psi_1,\psi_2\in C_{b}^{\infty}(\mathbb{R}^{6})$ which satisfy the following conditions:
 \begin{equation*}
\left\{
\begin{array}{l}
\displaystyle  \psi_1+\psi_2=1 \qquad\forall (\xi,\eta), \\
\displaystyle \Supp \psi_1 \subset \{ (\xi,\eta)\big| \lxmer\geq \f{\ler}{2}\},\\
\displaystyle \Supp \psi_2 \subset \{ (\xi,\eta)\big| \ler> \lxmer\}.\\
\end{array}
\right.
\end{equation*}
And write
\beno
\lxr^{\si}\frac{m}{\phi_{jk}}(\xi-\eta,\eta)&=&\f{m\psi_1(\xi-\eta,\eta)\lxr^{\si}}{\phi_{jk}\lxmer^{\si+2_{+}}\ler^2} \lxmer^{2_{+}}\ler^2+\f{m\psi_2(\xi-\eta,\eta)\lxr^{\si}}{\phi_{jk}\ler^{\si+2_{+}} \lxmer^2}\ler^{2_{+}} \lxmer^2\\
&\define& M_1(\xi-\eta,\eta)\lxmer^{2_{+}}\ler^2+M_2(\xi-\eta,\eta)\ler^{2_{+}} \lxmer^2.
\eeno
By Proposition \ref{elementary for phase}, we have for any $\alpha,\beta\in \mathbb{N}^3$ with $|\alpha|+|\beta|\leq 4$,
\beqs
|\pab{M_1}|\leq \mathrm{1}_{\lxmer\geq \f{\ler}{2}}\lxmer^{-2_{+}}\ler^{-1}.
\eeqs
In particular, we have proved that:
 $ M_{1}, \partial_{\xi}^{4}M_{1}, \partial_{\eta}^{4}M_{1}\in L^2(\mathbb{R}^{6})$.
 So we get that $\mathcal{F}^{-1}(M_1)(x,y)\in L^{1}_{x,y}$, as
 \beqs
 \|\mathcal{F}^{-1}(M_1)(x,y)\|_{L_{x,y}^{1}}\lesssim \|(1+|x|^{4}+|y|^{4})^{-1}\|_{L_{x,y}^2}(\|M_1\|_{L^2}+\|\partial_{\xi}^{4}M_{1}\|_{L^2}+\|\partial_{\eta}^{4}M_{1}\|_{L^2}).
 \eeqs
By using the definition of the bilinear operator $T_{m}$ (\ref{eqbilinear}) and properties of the Fourier transform, we can write:
\beqs
T_{M_{1}\lxr^{\si+2_{+}}\ler^2}(f,g)=\int (\cF^{-1}M_{1})(x',y')(\langle D_{x}\rangle^{\si+2_{+}}f)(x-x')\langle D_{x}\rangle^2 g(x-y')\d x'\d y',
\eeqs
thus by the  Minkowski's inequality, we have:
\beno
\|T_{M_{1}\lxr^{\si+2_{+}}\ler^2}(f,g)\|_{L^{p}}
&\leq& \int \|\langle D_{x}\rangle^2 g\|_{L^{r_1}}\|\int (\cF^{-1}M_{1})(x',y')(\langle D_{x}\rangle^{\si+2_{+}}f)(x-x')\d x'\|_{L^{p_1}}\d y'\nonumber\\
&\leq& \|\cF^{-1}M_{1}\|_{L_{x,y}^{1}} \|f\|_{W^{\si+2_{+},p_1}}\|g\|_{W^{2,r_1}}.
\eeno
The similar result for  $M_2$ can be derived in  the same fashion.
\end{proof}

\begin{proof}[Proof of Proposition \ref{elementary for phase}]

     We only prove the estimate  of  $\f{m}{\phi_{11}}$, the ones of  $\f{m}{\phi_{12}},\f{m}{\phi_{21}}$ can be obtained by symmetry, $\f{m}{\phi_{22}}$ is  easier.
    At first, we have
\beno
\f{1}{\phi_{11}}(\xi,\eta)&=&\f{b(\xi)+b(\eta)+b(\xi+\eta)}{(b(\xi)+b(\eta))^2-b^2(\xi+\eta)}
\define \f{b(\xi)+b(\eta)+b(\xi+\eta)}{A}.
\eeno
In the following, we will assume $\kpz\leq\f{1}{200}$, which ensure that: on the support $m(\xi,\eta)=\tilde{\chi}^{L}(\xi)\tilde{\chi}^{L}(\eta)\tilde{\chi}^{L}(\xi+\eta)\f{\lxper}{2ib(\xi+\eta)}$, we have that:
$\ep^2|\xi|^4\leq 16\kpz^2\leq \f{1}{2500}$ and $\f{99}{100}\lxr\leq b(\xi)\leq \lxr$. 
Under this assumption, $A$ has the lower bound:
\ben\label{Alowbdd}
A&=&1+2b(\xi)b(\eta)-2\xi\cdot\eta-\ep^{2}(|\xi|^{4}+|\eta|^{4}-|\xi+\eta|^{4})\nonumber\geq 1-32\kpz^2+2b(\xi)b(\eta)-2\xi\cdot\eta\nonumber\\
&\geq&\f{(1-32\kpz^2+2b(\xi)b(\eta))^{2}-4|\xi\cdot\eta|^2}{1-32\kpz^2+2b(\xi)b(\eta)+2\xi\cdot\eta}
\gtrsim\f{(b(\xi)+b(\eta))^{2}}{b(\xi)b(\eta)}\gtrsim \f{(\lxr+\ler)^{2}}{\lxr\ler}\gtrsim 1.
\een

Inspired by \cite{MR3274788}\cite{MR3024265}, we will prove that on the support of $m(\xi,\eta)$, for any multi-index $\alpha,\beta\in \mathbb{N}^3$, the following property holds:
\beq\label{ineq for A}
|\pab\f{1}{A}|\lesssim_{\alpha,\beta,\kpz}\f{1}{A}.
\eeq
This  is an easy consequence of Leibniz's rule and the estimate
\beq\label{derivative of A}
|\pab{A}|\lesssim_{\alpha,\beta,\kpz}A,\quad \forall \alpha,\beta\in \mathbb{N}^3.
\eeq
In the following, we will thus  prove \eqref{derivative of A}.
At first, we prove that
\beq\label{first derivative}
|\partial_{\xi,\eta} A|\lesssim_{\kpz}A.
\eeq
We will focus on  $\p_{\xi}A\lesssim_{\kpz}A$. One first notices that
on the support of $\tilde{\chi}^{L}(\xi)\tilde{\chi}^{L}(\eta)\tilde{\chi}^{L}(\xi+\eta)$
\beno
|\partial_{\xi} A|&=&\big|2\partial_{\xi}b(\xi)b(\eta)-2\eta-\ep^{2}(4|\xi|^{2}\xi-4|\xi+\eta|^{2}(\xi+\eta))\big|\\
&=&\big|2\f{b(\eta)}{b(\xi)}(1-2\ep^{2}|\xi|^{2})\xi-2\eta-\ep^{2}(4|\xi|^{2}\xi-4|\xi+\eta|^{2}(\xi+\eta))\big|\\
&\leq&2\big|\f{b(\eta)}{b(\xi)}(1-2\ep^{2}|\xi|^{2})\xi-\eta\big|+64\ep^{\f{1}{2}}\kpz^{\f{3}{2}}
\leq 2\big|\f{b(\eta)}{b(\xi)}(1-2\ep^{2}|\xi|^{2})\xi-\eta\big|+\f{4}{125}  ,
\eeno
if $\ep\leq 1,$ $\kpz\leq \f{1}{200}$.
Thus, by noticing again that 
$A\geq 1-32\kpz^2+2b(\xi)b(\eta)-2\xi\cdot\eta\geq \f{199}{200}++2b(\xi)b(\eta)-2\xi\cdot\eta$,
we only need to show that
$$\big|\f{b(\eta)}{b(\xi)}(1-2\ep^{2}|\xi|^{2})\xi-\eta\big|\lesssim_{\kpz}1+b(\xi)b(\eta)-\xi\cdot\eta.$$
Besides, we also observe that   on the support of $\tilde{\chi}^{L}(\xi)\tilde{\chi}^{L}(\eta)\tilde{\chi}^{L}(\xi+\eta)$, if $|\eta|\lesssim|\xi|$, we have $\f{b(\eta)}{b(\xi)}\ep^{2}|\xi|^{3}\leq
8\ep^{\f{1}{2}}\kpz^{\f{3}{2}}\leq \f{1}{250}$ and  if $|\xi|\leq|\eta|$, we have $\f{b(\eta)|\xi|}{b(\xi)|\eta|}\ep^{2}|\xi|^{2}|\eta|\leq \ep^{2}|\xi|^{2}|\eta| \leq 8\ep^{\f{1}{2}}\kpz^{\f{3}{2}} \leq \f{1}{250}$.
It thus suffices for us  to prove that
\beq\label{reductions}
\big|\f{b(\eta)}{b(\xi)}\xi-\eta\big|\lesssim_{\kpz}1+b(\xi)b(\eta)-\xi\cdot\eta.
\eeq
Let $\theta=\f{\xi\cdot\eta}{|\xi||\eta|}$, to prove \eqref{reductions}, we only need to prove that  there exists a constant $C$, $4<C<\infty$, such that:
\beqs
\f{b^{2}(\eta)}{b^{2}(\xi)}|\xi|^{2}+|\eta|^{2}-2\f{b(\eta)}{b(\xi)}|\xi||\eta|\theta\leq C[1+b^{2}(\xi)b^{2}(\eta)+|\xi|^{2}|\eta|^{2}\theta^{2}-2b(\xi)b(\eta)|\xi||\eta|\theta]
\eeqs
Define $F(\theta)=(|\xi|^{2}|\eta|^{2})\theta^{2}-2 b(\eta)|\xi||\eta|(b(\xi)-\f{1}{Cb(\xi)})\theta$.
The critical point of $F(\theta)$ is $\theta_{0}=\f{b(\eta)(b(\xi)-\f{1}{Cb(\xi)})}{|\xi||\eta|}$ and
\beno
&&\theta_{0}\geq 1
\Longleftrightarrow|\xi|^{2}|\eta|^{2}b^{2}(\xi)\leq b^{2}(\eta)(b^{2}(\xi)-\f{1}{C})^{2}\\
&\Longleftrightarrow&|\xi|^{2}|\eta|^{2}(1+|\xi|^{2}-\ep^{2}|\xi|^{4})\leq\big((1+|\xi|^{2}-\ep^{2}|\xi|^{4})^{2}+\f{1}{C^2}-\f{2}{C}(1+|\xi|^{2}-\ep^{2}|\xi|^{4})\big)
(1+|\eta|^{2}-\ep^{2}|\eta|^{4})
\eeno

Nevertheless, by the assumption $\kpz\leq\f{1}{200}$, and $C>4$,
we have that: $32\kpz^2+\f{2}{C}<1,(1-16\kpz^{2})^{2}+\f{1}{C^2}-\f{2}{C}>0$,
which leads to the lower bound of right hand side of the last inequality: $\big((1-16\kpz^2)^{2}+|\xi|^{4}+2(1-16\kpz^{2})|\xi|^2+\f{1}{C^2}-\f{2}{C}(1+|\xi|^{2})\big)(1-16\kpz^2+|\eta|^{2})\geq (1+|\xi|^2)|\xi|^2|\eta|^2$.
We thus have $\theta_0\geq 1$ and only need to prove \eqref{reductions} for $\theta=1$.
However,
\beno
\big|\f{b(\eta)|\xi|}{b(\xi)}-|\eta|\big|&\leq& (b(\eta)-|\eta|)\f{|\xi|}{b(\xi)}+(1-\f{|\xi|}{b(\xi)})|\eta|\\
&\leq& 1+\f{|\eta|(b(\xi)-|\xi|)}{b(\xi)}\leq 1+b(\xi)b(\eta)-|\xi||\eta|,
\eeno
this proves $\eqref{reductions}$ for $\theta=1$
which finish the proof of \eqref{first derivative}.


We now prove $\pab A\lesssim_{\kpz}A$ for $|\alpha|+|\beta|\geq 2$.
Indeed, it is direct to show $$\pab A\lesssim_{\kpz}\f{\ler}{\lxr}+\f{\lxr}{\ler}, |\alpha|+|\beta|\geq 2.$$
which, combined with \eqref{Alowbdd},
 yields $\pab A\lesssim_{\kpz}A$. This ends the proof of \eqref{derivative of A} and thus of \eqref{ineq for A}.

Next, we have
\beq
\f{1}{b(\xi)+b(\eta)-b(\xi+\eta)}= \f{b(\xi)+b(\eta)+b(\xi+\eta)}{A}\lesssim_{\kpz}\min\{b(\xi),b(\eta),b(\xi+\eta)\}.
\eeq
In fact, if $b(\xi+\eta)$ is not the biggest, we have $$b(\xi)+b(\eta)-b(\xi+\eta)\geq\min\{b(\xi),b(\eta)\}\geq 1.$$
Otherwise, by the lower bound for $A$ \eqref{Alowbdd},
$$\f{b(\xi)+b(\eta)+b(\xi+\eta)}{A}\lesssim_{\kpz}(b(\xi)+b(\eta)+b(\xi+\eta))\f{b(\xi)b(\eta)}{(b(\xi)+b(\eta))^{2}}\lesssim_{\kpz}\min\{b(\xi),b(\eta)\}.$$
Finally, by inequality \eqref{ineq for A}, we have:
\beno
\left|\pab \f {m(\xi,\eta)}{\phi_{11}}\right|&=&\left|\sum c_{\alpha_1\alpha_2\beta_1\beta_2}\partial_{\xi}^{\alpha_1}\partial_{\eta}^{\beta_1}m(\xi,\eta))
\partial_{\xi}^{\alpha_2}\partial_{\eta}^{\beta_2}\f {b(\xi)+b(\eta)+b(\xi+\eta)
}{A} \right|
\\
&\lesssim_{\kpz}&(b(\xi)+b(\eta)+b(\xi+\eta))\f{1}{A}\lesssim_{\kpz} \min \{\lxr,\ler,\lxmer\}.
\eeno
Similarly, one has, by choosing $\kpz$ small if necessary,
\beno
|\pab \f {m(\xi,\eta)}{\phi_{11}^2}|
&\lesssim_{\kpz}&\f{1}{\phi_{11}^2}
\lesssim_{\kpz} \min \{\lxr^2,\ler^2, \lxper^2\}.
\eeno
\end{proof}
%

\end{document}